\documentclass[a4paper,11pt]{article}
\usepackage{pdfsync}

\usepackage{imakeidx}
 \makeindex
 \usepackage{hyperref}
 \usepackage{amsfonts,amsmath,amssymb,amsthm,mathtools}
 \usepackage{latexsym,lscape,rawfonts,mathrsfs,slashed}
\usepackage{enumitem}
\usepackage{scalerel}[2014/03/10]
\usepackage[usestackEOL]{stackengine}
\usepackage{CJKutf8}
\usepackage{fmtcount}
\usepackage{needspace}
\usepackage{multicol}
\usepackage{titletoc}
\allowdisplaybreaks[4]
\usepackage[thinc]{esdiff}

 \usepackage[all]{xy}
 \usepackage{graphicx,psfrag}
 \usepackage{pstool}
 \usepackage{tikz-cd}
 \usepackage{upgreek}

 \usepackage{array,tabularx}

 \usepackage{setspace}

 \usepackage[titletoc,title]{appendix}
\usepackage{txfonts}

 \newcommand{\F}{\ensuremath{\mathbb{F}}}

 \newcommand{\R}{\ensuremath{\mathbb{R}}}
 
 \def\III{\mathbb{I}}
 \def\XX{\mathcal{X}}
 \def\RR{\mathcal{R}}
 \def\MM{\mathcal{M}}
 \def\LL{\mathcal{L}}
 \def\BB{\mathcal{B}}
 \def\Tr{\mathrm{Tr}}
 \def\KK{\mathcal{K}}

 \def\Var{\mathrm{Var}}
 \def\TT{\mathcal{T}}

 \def\VV{\mathcal{V}}
 \def\scal{\mathrm{R}}
 \def\rCd{\mathbf{r}_{C,\delta}'}
 \def\rE{\mathbf{r}_E}
 \def\loc{\mathrm{loc}}
 
 \def\rCd{\textbf{r}_{C,\delta}}
 \def\spt{\mathrm{supp}}

 \def\r{\textbf{r}}

 \newcommand{\ba}{\begin{align*}}
 \newcommand{\ea}{\end{align*}}
 \newcommand{\flow}{\mathcal{X}=\{M^n,(g(t))_{t\in I}\}}
 \newcommand{\na}{\nabla}
\newcommand{\la}{\langle}
\newcommand{\ra}{\rangle}

\newcommand{\lc}{\left(}
\newcommand{\rc}{\right)}
 
\newcommand{\ep}{\epsilon}

\def\De{\Delta}

\def\CC{\mathcal{C}}

\def\NN{\mathcal{N}}

\def\WW{\mathcal{W}}

\def\Var{\mathrm{Var}}

\def\Div{\mathrm{div}}
\def\dd{\mathrm{d}}

\newcommand{\di}{\text{div}}
\newcommand{\Rm}{\ensuremath{\mathrm{Rm}}}
\newcommand{\Ric}{\ensuremath{\mathrm{Ric}}}

\renewcommand{\t}{\mathfrak{t}}
\newcommand{\CF}{\mathfrak{C}}

\def\rB{\textbf{r}_{B}}
\def\rBC{\textbf{r}_{B,\sigma}}
\def\rA{\textbf{r}_A}
\def\MS{\mathcal{S}}
\newcommand{\IF}{\ensuremath{\mathbb{F}}}

 \makeatletter
 \def\ExtendSymbol#1#2#3#4#5{\ext@arrow 0099{\arrowfill@#1#2#3}{#4}{#5}}
 
 \makeatother

 \makeatletter
 \def\ExtendSymbol#1#2#3#4#5{\ext@arrow 0099{\arrowfill@#1#2#3}{#4}{#5}}
 
 \makeatother
 
\def\aint{\,\ThisStyle{\ensurestackMath{%
  \stackinset{c}{.2\LMpt}{c}{.5\LMpt}{\SavedStyle-}{\SavedStyle\phantom{\int}}}%
  \setbox0=\hbox{$\SavedStyle\int\,$}\kern-\wd0}\int}

\DeclarePairedDelimiter\abs{\lvert}{\rvert}%
\makeatletter
\let\oldabs\abs
\def\abs{\@ifstar{\oldabs}{\oldabs*}}
\makeatother

\numberwithin{equation}{section}

\newtheorem{thm}{Theorem}[section]
\newtheorem{cor}[thm]{Corollary}
\newtheorem{prop}[thm]{Proposition}
\newtheorem{lem}[thm]{Lemma}

\newtheorem{rem}[thm]{Remark}

\newtheorem{defn}[thm]{Definition}

\newtheorem{exmp}[thm]{Example}

\newtheorem{claim}[thm]{Claim}

\titlecontents{part}[0pt]  
{\addvspace{10pt}\bfseries} 
{\makebox[3em][l]{PART-\thecontentslabel}}      
{}                          
{}                          

\title{Strong uniqueness of tangent flows at cylindrical singularities in Ricci flow}
\author{Hanbing Fang \quad and \quad Yu Li} 
\date{\today}

\begin{document}
	\begin{CJK}{UTF8}{gbsn}

\maketitle
	\begin{abstract}
In this paper, we establish a Lojasiewicz inequality for the pointed $\mathcal{W}$-entropy in the Ricci flow, under the assumption that the geometry near the base point is close to a standard cylinder $\mathbb{R}^k \times S^{n-k}$ or the quotient thereof. As an application, we prove the strong uniqueness of the cylindrical tangent flow at the first singular time of the Ricci flow. Specifically, we show that the modified Ricci flow near the singularity converges to the cylindrical model under a fixed gauge.
	\end{abstract}
	
	\tableofcontents
\section{Introduction}

Ricci flow, introduced by Richard Hamilton in 1982 in his foundational paper \cite{hamilton1982three}, is a geometric evolution equation for Riemannian metrics. Given an initial Riemannian manifold $(M^n, g_0)$, the Ricci flow evolves the metric according to the nonlinear PDE:
\begin{equation*}
	\partial_tg(t)=-2\Ric(g(t)).
\end{equation*}
This equation may be viewed as a heat-type flow that tends to uniformize the geometry of the manifold by smoothing out curvature irregularities over time. Ricci flow has had a profound impact on geometric analysis and topology, culminating in Grigori Perelman's resolution of the Poincar\'e and Geometrization Conjectures (\cite{perelman2002entropy, perelman2003ricci, perelman2003finite}). 

In many cases, the development of short-time singularities is inevitable and has become a central focus of study. A short-time singularity is formed at the first singular time, near which the curvature of the Ricci flow is unbounded. Understanding the structure of such singularities is crucial for extending the Ricci flow past singular times, whether through geometric surgeries or via weak flow techniques.

A key strategy in analyzing singularities is blow-up analysis—rescaling the flow at the singularity to extract \textbf{tangent flows}. For Type I singularities, it follows from the work of Enders-M\"uller-Topping \cite{enders2011type} that any blow-up sequence converges smoothly to a self-similar ancient solution, specifically a Ricci shrinker. More generally, Bamler \cite{bamler2020structure} showed that blow-up sequences $\mathbb F$-converge to a Ricci shrinker with mild singularities, which is referred to as a metric soliton.

In \cite{fang2025RFlimit}, a spacetime distance $d^*$ was introduced for closed Ricci flows with entropy bounded below (see Definitions \ref{def:entropybound} and \ref{defnd*distance}). Using this notion, the time-zero slice—corresponding to the singular time—can be obtained as the metric completion under $d^*$. Specifically, consider a smooth closed Ricci flow $\XX=\{M^n, (g(t))_{t \in [-T,0)}\}$ with entropy bounded below by $-Y$, where $t=0$ is the first singular time. Then the space
  \begin{align*}
  	(Z,d_{Z} ,\t)
  \end{align*}
is defined as the metric completion of $\XX_{[-0.98T, 0)}$ with respect to $d^*$. It can be shown (see \cite[Section 9]{fang2025RFlimit}) that the completion only adds new points in the time-zero slice $Z_0$, and that $(Z,d_{Z} ,\t)$ is a noncollapsed Ricci flow limit space in the sense of \cite{fang2025RFlimit}, with singularities entirely contained in $Z_0$.
  
In this framework, a tangent flow $(Z', d_{Z'}, z', \t')$ at a point $z \in Z_0$ is defined as a pointed Gromov--Hausdorff limit of $(Z, r_j^{-1}d_Z, z, r_j^{-2}(\t-\t(z)))$ for a sequence $r_j \searrow 0$. It is proved in \cite[Section 7]{fang2025RFlimit} that $(Z', d_{Z'}, z', \t')$ is itself a noncollapsed Ricci flow limit space defined over $(-\infty,0]$. Moreover, its regular part $(\RR', \t', \partial_{\t'}, g^{Z'}_t)$, which is given by a Ricci flow spacetime, satisfies the Ricci shrinker equation
\begin{equation*}
\Ric(g^{Z'})+\na^2 f_{z'}=\frac{g^{Z'}}{2|\t'|},
\end{equation*}
where $f_{z'}$ is the potential function at $z'$ (see Definition \ref{def:chkm}). For further properties of tangent flows, we refer readers to \cite[Theorem 1.9]{fang2025RFlimit}.

A fundamental question in the analysis of geometric PDEs is whether the tangent space at a singularity is unique. When a singularity forms, one typically performs a blow-up procedure to study the local behavior. However, depending on the choice of rescaling sequences, non-uniqueness of tangent spaces is in principle possible. Conversely, uniqueness suggests that the underlying geometric and analytic structure is stable and coherent at all smaller scales.

Significant progress has been made in establishing the uniqueness of tangent spaces across various geometric contexts. For minimal surfaces, Allard and Almgren \cite{allard81} in 1981 proved the uniqueness of tangent cones under the assumption of smooth cross-sections and integrability conditions. This integrability assumption was later removed by Leon Simon in 1983 \cite{simon1983asymptotics}. White \cite{white1983} further showed that tangent cones are unique for two-dimensional area-minimizing integral currents.

In the setting of Ricci-flat manifolds with maximal volume growth, Cheeger and Tian \cite{Cheeger94} proved the uniqueness of tangent cones at infinity under quadratic curvature decay and an integrability assumption. Colding and Minicozzi later showed in \cite{colding2014uniqueness} that if one tangent cone has a smooth cross-section, then the tangent cone at infinity is unique. They also established the uniqueness of tangent cones at singular points in Gromov--Hausdorff limits of noncollapsed Einstein manifolds, again assuming the smoothness of the cross-section of a tangent cone. In the K\"ahler--Einstein setting, Donaldson and Sun \cite{ds2017} proved the full uniqueness of tangent cones, without assuming smooth cross-sections.

In mean curvature flow, tangent flows correspond to self-shrinkers, and uniqueness means that a particular self-shrinker is the unique blow-up limit at a singularity. Several notable results have established the uniqueness under various geometric assumptions. Schulze \cite{schulze2014uniqueness} proved uniqueness when one tangent flow is a compact smooth shrinker. Colding and Minicozzi \cite{colding2015uniqueness, colding2019regularity} showed that uniqueness holds when one tangent flow is the standard cylinder $\mathbb{R}^{k} \times S^{n-k}$. More recently, Chodosh--Schulze \cite{CS21} and Lee--Zhao \cite{LZ24} established the uniqueness in the case where the tangent flow is an asymptotically conical self-shrinker.

The uniqueness of tangent flows in Ricci flow remains a subtle and evolving topic. When a tangent flow at a singularity is given by a compact, smooth Ricci shrinker, it has been shown by Sun--Wang \cite{sun2015kahler} and Chan--Ma--Zhang \cite{CMZ24} that the tangent flow is unique. More recently, Choi--Lai \cite{Choilai25} obtained sharper convergence rates under the additional assumption that the Ricci shrinker is integrable. In the setting of Fano manifolds, the tangent flow at the single singularity is also known to be unique; see \cite[Theorem 1.2]{csw2018} and \cite[Corollary 1.4]{hanli2024}. However, in general, most tangent flows are non-compact, and the question of uniqueness in this broader context remains significantly more delicate and technically challenging.

One effective approach to proving the uniqueness of tangent flows is to study the \textbf{rigidity} of Ricci shrinkers in the moduli space of all shrinkers. Specifically, if a Ricci shrinker is isolated in this space---meaning no other nearby Ricci shrinkers exist---then, via a continuity argument, one can deduce that if a tangent flow is given by this shrinker, then it must be the unique tangent flow. This argument is plausible because all tangent flows at a singularity are derived from a single underlying closed Ricci flow, and the blow-up limits vary continuously under suitable topologies.

The first example of a rigid non-compact, nontrivial Ricci shrinker was given by Li--Wang \cite{li2024rigidity}, who proved the rigidity of the standard cylinder $\R \times S^{n-1}$. Subsequently, Colding--Minicozzi \cite{colding2021singularities} established the rigidity of the cylinders $\R^{k}\times S^{n-k}$ for all $1 \le k \le n-2$. Later, Li--Zhang \cite{li2023rigidity} extended the techniques of \cite{colding2021singularities} to prove the rigidity of more general cylinders of the form $\R^{k} \times N^{n-k}$ and their quotients, where $N$ is an $H$-stable Einstein manifold with obstruction of order $3$ (see \cite[Definitions 2.11, 2.20]{li2023rigidity}). 

These rigidity results yield powerful consequences for tangent flow uniqueness. In particular, Colding--Minicozzi \cite[Theorem 1.2]{colding2021singularities} proved that, under the Type I curvature assumption, the tangent flow at a singularity is unique if one tangent flow is given by $\R^{k}\times S^{n-k}$. More generally, Li--Zhang \cite[Theorem 6.2]{li2023rigidity} showed that the tangent metric soliton at a singularity is unique if one of such blow-up limits is a generalized cylinder or its quotient. Here, the notion of a tangent metric soliton refers to a blow-up limit in the $\mathbb F$-convergence sense introduced by Bamler \cite{bamler2020structure}.

Building on these rigidity results, one obtains the following uniqueness theorem for cylindrical tangent flows in the framework of Ricci flow limit spaces in \cite{fang2025RFlimit}. Here, the definition of isometry can be found in \cite[Definition 5.21]{fang2025RFlimit}.

\begin{thm}[Uniqueness of the cylindrical tangent flow] \label{thm: weakunique}
Let $(Z, d_Z, \t)$ be the completion of a closed Ricci flow $\XX=\{M^n, (g(t))_{t \in [-T,0)}\}$ with entropy bounded below by $-Y$. For any $z \in Z_0$, if a tangent flow at $z$ is isometric to $\bar{\mathcal C}^k$, then every tangent flow at $z$ is isometric to $\bar{\mathcal C}^k$.
\end{thm}

The proof of Theorem \ref{thm: weakunique} will be given later in Theorem \ref{thm:weakunique}, and is based on the aforementioned rigidity and convergence results. Here, we define the model cylindrical Ricci shrinker
\begin{align*}
\mathcal C^{k}_{-1}:=(\bar M,\bar g,\bar f)=\left(\R^{k}\times S^{n-k}, g_E \times g_{S^{n-k}}, \frac{|\vec{x}|^2}{4}+\frac{n-k}{2}+\Theta_{n-k} \right),
\end{align*}
where $g_E$ is the Euclidean metric on $\R^{k}$, and $g_{S^{n-k}}$ is the round metric on $S^{n-k}$ such that $\Ric(g_{S^{n-k}})=g_{S^{n-k}}/2$. The vector $\vec{x}=(x_1,\ldots,x_{k})$ denotes the standard coordinate function on $\R^{k}$. The constant $\Theta_{n-k}$ is the entropy of the cylinder. 

Let $\mathcal C^{k}=(\bar M, (\bar g(t))_{t<0}, (\bar f(t))_{t<0})$ denote the corresponding Ricci flow with potential such that $t=0$ is the singular time. Then, $\bar{\mathcal C}^k$ is defined as the completion of $(\bar M, (\bar g(t))_{t \in (-\infty, 0)})$ with respect to the spacetime distance $d_{\mathcal C}^*$. The base point $p^*$ is taken to be the limit of $(\bar p, t)$ as $t \nearrow 0$, where $\bar p$ is a minimum point of $\bar f$.

We call $z \in Z_0$ a \textbf{cylindrical singularity} with respect to $\bar{\mathcal C}^k$ if every tangent flow is isometric to $\bar{\mathcal C}^k$. The uniqueness result in Theorem \ref{thm: weakunique}, combined with the smooth convergence properties of Ricci flow limit spaces (see Theorem \ref{thm:intro3}), implies the following pointed Cheeger–Gromov convergence:
\begin{align}\label{eq:intro001}
\lc M, r^{-2} g(-r^2), f_z(-r^2), p_r \rc \xrightarrow[r \to 0]{\quad \text{pointed Cheeger--Gromov} \quad} (\bar M, \bar g, \bar f, \bar p),
\end{align}
where $f_z$ is the potential function at $z$ (see Definition \ref{def:chkm}), and $p_r$ is a minimum point of $f_z(-r^2)$. 

However, this convergence has two drawbacks. First, it occurs only in the Cheeger--Gromov sense, which means that the convergence is realized via a sequence of diffeomorphisms from open subsets of $\bar M$ into $M$, but these diffeomorphisms may be entirely unrelated to one another.  Second, the convergence \eqref{eq:intro001} does not yield any quantitative rate of convergence for the original Ricci flow $\XX$ as it approaches the cylindrical singularity $z$. Thus, in order to obtain more analytical information approaching $z$, one must obtain a stronger result.

To gain more analytic control and resolve these issues, we establish a stronger result—the strong uniqueness of the cylindrical tangent flow. To this end, we first modify our Ricci flow $\XX=\{M^n, (g(t))_{t \in [-T,0)}\}$ as follows (see also \cite{colding2024eigenvalue}).

Let $\phi_t$ be the family of diffeomorphisms generated by $-\na_{g(t)} f_z(t)$ with $\phi_{-T}=\mathrm{id}$. Define the \textbf{modified Ricci flow} with respect to $z$ as:
  	\begin{equation*}
  		\begin{dcases}
  			&g^z (s):= e^s \phi^*_{-e^{-s}} g(-e^{-s}),   \\
  			&f^z (s):=\phi^*_{-e^{-s}} f_z(-e^{-s}),\\
  		\end{dcases}
  	\end{equation*}
  	where $s \in [-\log T, \infty)$. The pair $(g^z(s),f^z(s))$ evolves by the following system:
  	\begin{align*}
  		\begin{dcases}
  			& \partial_s g^z(s)=g^z-2 \Ric(g^z)-2\na^2 f^z, \\
  			&	\partial_s f^z(s)=\frac{n}{2}-\scal_{g^z}-\Delta f^z.
  		\end{dcases}
  	\end{align*}
  	
We now state our first main result, which provides a precise convergence estimate for the modified flow near the cylindrical singularity.

\begin{thm}[Strong uniqueness of the cylindrical tangent flow]\label{thm:stronguni2intro}
In the setting above, suppose $z \in Z_0$ is a cylindrical singularity with respect to $\bar{\mathcal C}^k$. Then for any small $\ep>0$, there exists a large constant $\bar s$ such that for any integer $j \ge \bar s$, there exists a diffeomorphism $\psi_{j}$ from $\Omega^j:=\left\{\bar f \le (2-\ep)\log j\right\} \subset \bar M$ onto a subset of $M$ satisfying the compatibility condition
	\begin{align*}
		\psi_{j+1}=\psi_j \quad \mathrm{on}\quad \Omega^j,
	\end{align*}
and for all $s\geq j$, the following decay estimate holds on $\Omega^j$:
	\begin{align}\label{eq:intro002}
	\left[\psi_{j}^*g^z(s)-\bar g \right]_{[\ep^{-1}]}+\left[\psi_{j}^*f^z(s)-\bar f \right]_{[\ep^{-1}]}\leq C(n,Y,\ep)e^{\frac{\bar f}{2}} s^{-1+\ep}.
	\end{align}
Here, for any integer $l \ge 0$, the norm is defined by
	\begin{align*}
[\cdot]_l:=\sum_{i=0}^l \left|\na_{\bar g}^i (\cdot) \right|_{\bar g}.
\end{align*}
\end{thm}

Roughly speaking, Theorem \ref{thm:stronguni2intro} asserts that for sufficiently large $j$, one can construct a diffeomorphism $\psi_j$ from a region of size approximately $\sqrt{8\log j}$ in $\bar M$ into $M$, such that these diffeomorphisms are compatible on their common domains. Moreover, under each $\psi_j$, the modified Ricci flow exhibits a precise decay estimate as \eqref{eq:intro002}. In other words, the modified Ricci flow converges to the cylindrical model \textbf{without} modulo any diffeomorphism.

The main technical tool used in the proof of Theorem \ref{thm:stronguni2intro} is a Lojasiewicz-type inequality. In his seminal work \cite{loja1965}, Lojasiewicz proved the following fundamental result: let $f: U \to \R$ be a real-analytic function on a domain $U \subset \R^n$, then for any $x_0 \in U$ with $\na f(x_0)=0$, there exist a small neighborhood $W \subset U$ and constants $C>0$ and $\beta \in [1/2, 1)$ such that
\begin{equation}\label{loja2}
	|f(x)-f(x_0)|^\beta\leq C|\na f (x)|, \quad \forall x \in W.
\end{equation}
This inequality implies that if $x(t)$ is a negative gradient flow of $f$, and $x_0$ is the limit point of $x(t_i)$ for a sequence $t_i \to \infty$, then \eqref{loja2} can be used to show both convergence of $x(t) \to x_0$ and a quantitative decay rate depending on the exponent $\beta$ (see, for instance, \cite[Section 3]{sun2014}).

In his celebrated work \cite{simon1983asymptotics}, Leon Simon established an infinite-dimensional version of the Lojasiewicz inequality for analytic functionals on Banach spaces. As an application, he proved a fundamental result on the uniqueness of tangent cones with smooth cross-sections for minimal surfaces—an important theorem mentioned earlier. The key idea in Simon's approach was to reduce the infinite-dimensional problem to the classical finite-dimensional Lojasiewicz inequality via a Lyapunov–Schmidt reduction.

Simon’s method has since had a profound influence across many areas of analysis and geometry. Variants and extensions of his techniques have been applied in diverse settings, including global convergence results and stability analysis. For instance, see \cite{colding2014uniqueness, sun2014, haslhofer2014dynamical, sun2015kahler, kroncke2015stability, colding2014uniqueness, colding2015uniqueness, colding2019regularity, zhu2020lojasiewicz, deruelle2022lojasiewicz, CS21}, among others.

In the context of Ricci flow, the use of Lojasiewicz-type inequalities to prove the uniqueness of tangent flows is not new, as the Ricci flow can be viewed as the gradient flow of Perelman's $\boldsymbol{\mu}$-functional. In \cite[Lemma 3.1]{sun2015kahler}, Sun and Wang established a Lojasiewicz inequality for the $\boldsymbol{\mu}$-functional:
	\begin{align}\label{eq:intro003}
\abs{\boldsymbol{\mu}(g, 1)-\boldsymbol{\mu}(g_0,1)}^{\beta} \le C\|\na \boldsymbol{\mu}(g, 1)\|_{L^2},
	\end{align}
for some constant $\beta \in [1/2, 1)$, where $g_0$ is a given Ricci shrinker metric on a closed manifold and $g$ is a metric nearby. By applying inequality \eqref{eq:intro003} and following the general strategy, they proved the uniqueness of the tangent flow at a singularity, assuming that one blow-up limit is given by the compact Ricci shrinker $g_0$.

It is tempting to directly generalize inequality \eqref{eq:intro003} to non-compact Ricci shrinkers, but this cannot be done for several reasons. First, in the non-compact setting, a minimizer for Perelman's $\boldsymbol{\mu}(g, 1)$ may fail to exist, and even when it does, it may not depend analytically on the metric $g$. Second, $\boldsymbol{\mu}$-functional is inherently a global quantity, and as such, it does not effectively capture localized behavior near a given singularity. Third, when $g_0$ is a Ricci shrinker on a non-compact manifold, a perturbation $g$ that is close to $g_0$ may only resemble $g_0$ in a large but bounded region, without any control at infinity.

To overcome these difficulties, we use the \textbf{pointed $\WW$-entropy}---first introduced in \cite[Section 5]{perelman2002entropy} and studied systematically in \cite{hein2014new}---a localized version of Perelman’s $\WW$-functional. Specifically, if we write the conjugate heat kernel at $z \in Z_0$ by $\mathrm{d}\nu_{z; t}=(4\pi |t|)^{-\frac  n 2} e^{-f_z} \,\mathrm{d}V_{g(t)}$, then the pointed $\WW$-entropy at $z$ is defined by
	\begin{equation*}
		\WW_{z}(\tau):= \int_M\tau(2\Delta f_{z}-|\na f_{z}|^2+\scal)+f_{z}-n \,\mathrm{d}\nu_{z;-\tau}.
	\end{equation*}

Now, we state the Lojasiewicz inequality for the pointed $\WW$-entropy:

\begin{thm}[Lojasiewicz inequality]\label{lojaRFcyl}
Let $(Z, d_Z, \t)$ be the completion of a closed Ricci flow $\XX=\{M^n, (g(t))_{t \in [-T,0)}\}$ with entropy bounded below by $-Y$. For any cylindrical singularity $z \in Z_0$ with respect to $\bar{\mathcal C}^k$, the following property holds.

For any $\beta \in (0,3/4)$, there exists a constant $C=C(n,Y)$ such that for all sufficiently small $\tau>0$,
	\begin{align}\label{eq:intro004}
		\left|\mathcal W_{z}(\tau)-\Theta_{n-k} \right| \le C \left( \mathcal W_{z}(\tau/2)-\mathcal W_{z}(2 \tau) \right)^{\beta}.
	\end{align}
\end{thm}

Inequality \eqref{eq:intro004} is often referred to as a discrete Lojasiewicz inequality, and can be seen as a discretized variant of the classical version \eqref{loja2}. Unlike the classical setting in \cite{simon1983asymptotics}, we do not rely on Simon’s argument, as our model space is non-compact. Instead, we provide a direct proof of \eqref{eq:intro004}. Similar forms of discrete Lojasiewicz inequalities have appeared in the context of mean curvature flow, particularly in the work of Colding and Minicozzi \cite{colding2015uniqueness, colding2019regularity}, where they established analogous results for cylindrical singularities, using entropy functionals developed in \cite{colding2012generic}.

In our setting, this inequality yields a natural summability result for the entropy difference, which plays an important role in quantifying the convergence toward cylindrical models. To formulate this precisely, we introduce the following notion:

Let $\XX=\{M^n, (g(t))_{t \in [-T, 0)}\}$ be a closed Ricci flow with entropy bounded below by $-Y$ such that $(Z,d_{Z} ,\t)$ is the metric completion of $\XX_{[-0.98T, 0)}$ with respect to $d^*$. A point $z \in Z$ is called \textbf{$(k,\ep,r)$-cylindrical} if $\t(z)-\ep^{-1} r^2 \ge -0.98T$ and
	  \begin{align*}
(Z, r^{-1} d_Z, z, r^{-2}(\t-\t(z))) \quad \text{is $\ep$-close to} \quad (\bar{\mathcal C}^k ,d_{\mathcal C}, p^*,\t) \quad \text{over} \quad [-\ep^{-1}, \ep^{-1}],
  \end{align*} 
where the definition of $\ep$-closeness can be found in Definition \ref{defn:close}.

We now state a direct consequence of Theorem \ref{lojaRFcyl}, whose proof will be given in Corollary \ref{quantisummabilityW}:

\begin{cor}\label{cor:sum}
Given constants $\ep>0$, $\alpha\in (1/4,1)$ and $\delta\leq\delta(n,Y,\ep,\alpha)$, the following holds.
	
Let $\XX=\{M^n,(g(t))_{t\in [-T, 0) }\}$ be a closed Ricci flow with entropy bounded below by $-Y$. Fix $x_0^*=(x_0,t_0)\in M \times [-T, 0)$ and constants $s_2>s_1>0$. If
	\begin{align*}
		\left|\WW_{x_0^*}(s_1)-\WW_{x_0^*}(s_2)\right|<\delta,
	\end{align*}
and, for any $s\in [s_1,s_2]$, $x_0^*$ is $(k,\delta,\sqrt{s})$-cylindrical, then
	\begin{align}\label{eq:intro005}
		\sum_{s_1\leq r_j=2^{-j} \leq s_2}\left|\WW_{x_0^*}(r_j)-\WW_{x_0^*}(r_{j-1})\right|^{\alpha}<\ep.
	\end{align}
\end{cor}

In our forthcoming work \cite{FLrec25}, we extend Theorem \ref{lojaRFcyl} and Corollary \ref{cor:sum} to noncollapsed Ricci flow limit spaces. The summability inequality \eqref{eq:intro005} with $\alpha=1/2$ is a key analytic ingredient in our analysis of cylindrical singularities. Specifically, we prove in \cite{FLrec25} that, for any noncollapsed Ricci flow limit space, the subset of the singular set consisting of points at which some tangent flow is $\bar{\mathcal C}^k$ is \textbf{parabolic $k$-rectifiable} with respect to the spacetime distance introduced in \cite{fang2025RFlimit}.

\subsection*{Outline of the Proof of Theorem \ref{lojaRFcyl}}

The proof of Theorem \ref{lojaRFcyl} involves several technical challenges and key analytic ingredients. We provide an overview of the main steps below.

\textbf{(A)} Variation of the $\WW$-functional under global perturbations of the cylinder

We begin by considering the model cylinder $\mathcal C^{k}_{-1}=(\bar M,\bar g,\bar f)$, along with a global perturbation $(g, f)$, such that
	\begin{align*}
h:=g-\bar g,\quad \chi:=f-\bar f, \quad \text{with} \quad \|h\|_{C^2}+\|\chi\|_{C^2} \ll 1.
	\end{align*}
	
Following the approach in \cite{colding2021singularities} and \cite{li2023rigidity}, we apply the decomposition:
\begin{align*}
h=ug_{S^{n-k}}+\zeta,\quad \chi=\frac{n-k}{2} u+q,
\end{align*}
 where $ug_{S^{n-k}}$ denotes the projection of $h$ onto $\KK_0 g_{S^{n-k}}$, and $\KK_0$ is the space of quadratic Hermite polynomials (see Definition \ref{defncylinder}). 
 
 In Section \ref{seclojaW} (see Theorem \ref{lojaforF}), we establish the following variational inequality for the $\WW$-functional:
 	\begin{equation}\label{intrlojaW}
 		\left|\WW(g,f)-\Theta_{n-k}\right|\leq C\left(\rVert u\rVert_{L^2}^3+\rVert \mathbf{\Phi}(g,f)\rVert_{L^2}^2+\rVert\Div_{\bar f} h\rVert_{W^{1,2}}^2+|\BB(h,\chi)|^2+\left|\VV(g,f)-1\right|\right).
 	\end{equation} 
Here:
\begin{itemize}
 	\item $\WW(g, f)$ is Perelman's $\WW$-functional at scale $1$ (see Definition \ref{defnshrinkerquant}),
 	
 	\item $\displaystyle \mathbf{\Phi}(g,f):=\frac{g}{2}-\Ric(g)-\na^2f$ is the Ricci shrinker operator,
 	
 	\item $\displaystyle \VV(g,f)=(4\pi)^{-\frac{n}{2}}\int_{\bar M}e^{-f}\,\mathrm{d}V_g$ is the weighted volume,
 	
 	\item $\BB(h,\chi)$ is the center of mass vector (see \eqref{defncentermass}), and
 	
 	\item all $L^2$-norms are taken with respect to the weighted measure $\mathrm{d}V_{\bar f}=(4\pi)^{-\frac{n}{2}} e^{-\bar f}\,\mathrm{d}V_{\bar g}$.
 \end{itemize}
 
Inequality \eqref{intrlojaW} is derived via a Taylor expansion of the $\WW$-functional on the cylinder, combined with a quadratic rigidity inequality (see Proposition \ref{prop:tay1}) from \cite{colding2021singularities}.
 
 Let 
 	\begin{align*}
g(s):=\bar g+sh,\quad f(s):=\bar f+s\chi, \quad  W(s):=\WW(g(s),f(s)).
	\end{align*}
 Then the Taylor expansion yields:
 	 \begin{align}\label{intrtaylor1}
 	 	\left|\WW(g,f)- \Theta_{n-k} \right|=\left|W(1)-W(0)\right|
 	 	\leq  \left|W'(0)+\frac{1}{2}W''(0)\right|+\frac{1}{6}\sup_{s\in [0,1]}\left|W'''(s)\right|.
 	 \end{align}
By performing detailed calculations for $W'$, $W''$, and $W'''$, and carefully applying the rigidity inequality, we control each term on the right-hand side of \eqref{intrtaylor1}, which ultimately yields inequality \eqref{intrlojaW}.

\textbf{(B)} Characterization of almost cylindrical regions in the Ricci flow

Next, we consider a closed Ricci flow $\XX=\{M,(g(t))_{t \in I}\}$ with entropy bounded below by $-Y$ and a based point $x_0^*=(x_0,t_0)\in\XX$ such that $[t_0-2r^2,t_0]\subset I$. Then we consider the conjugate heat kernel measure $\nu_t=\nu_{x_0^*;t}$ with potential function $f=f_{x_0^*}$. For simplicity, we assume $t_0=0$, $r=1$, and $\tau=-t$.

For the weighted Riemannian manifold $(M, g(-1), f(-1))$, the geometry may not arise from a global perturbation of the cylinder $\mathcal C^{k}_{-1}$, even if it is locally close to one. Therefore, to detect and work with such almost cylindrical regions, we introduce a set of radius functions that quantify the size of regions in $M$ where the geometry closely approximates that of a cylinder. We define $\bar b:=2\sqrt{\bar f}$ and introduce four key radii:

\begin{defn}
	For small positive constants $\sigma$ and $\delta$, and the weighted Riemannian manifold $(M, g(-1), f(-1))$ with $\Phi=\mathbf{\Phi}(g(-1), f(-1))$:
		\begin{itemize}[leftmargin=*, label={}]
		\item \emph{($\rA$-radius)}. $\rA$ is defined as the largest number $L$ such that there exists a diffeomorphism $\varphi_A$ from $\{\bar b \le L\} \subset \bar M$ onto a subset of $M$ such that
		\begin{align}\label{eq:intro007a}
			\left[\bar g-\varphi_A^* g(-1)\right]_5+\left[\bar f-\varphi_A^* f(-1)\right]_5 \leq e^{\frac{\bar f}{4}-\frac{L^2}{16}}.
		\end{align}
		\item \emph{($\rBC$-radius)}. $\rBC$ is defined as the largest $L$ such that there exists a diffeomorphism $\varphi_B$ from $\{\bar b \le L\} \subset \bar M$ onto a subset of $M$ such that
		\begin{align}\label{eq:intro007b}
			\left[\bar g-\varphi_B^* g(-1)\right]_0+\left[\bar f-\varphi_B^* f(-1)\right]_0\leq e^{-\frac{L^2}{33}},
		\end{align}
		and
		\begin{equation}\label{eq:intro007c}
			\int_{\{\bar b\le L\}}\left|\varphi_B^*\Phi\right|^2 \,\mathrm{d}V_{\bar f}\leq e^{-\frac{L^2}{4-\sigma}}.
		\end{equation}
		Furthermore, for all $l \in [1, 10^{10} n\sigma^{-1}]$, the $C^l$-norms of $\bar g-\varphi_B^* g(-1)$ and $\bar f-\varphi_B^* f(-1)$ are bounded by $1$. 
		
		\item \emph{($\r_{C, \delta}$-radius)}. $\r_{C,\delta}$ is defined as the largest $L$ such that there exists a diffeomorphism $\varphi_C$ from $\{\bar b \le L\}$ of $\bar M$ onto a subset of $M$ such that $f \lc \varphi_C(\bar p), -1 \rc \le n$ and
	\begin{align*}
		\left[\bar g-\varphi_C^* g(-1)\right]_2\leq \delta.
	\end{align*}
	
	\item 	\emph{($\rE$-radius)}. The \textbf{entropy radius} $\rE$ is defined as
	\begin{align*}
		\exp\left(-\frac{\mathbf{r}^2_E}{4}\right):=\mathcal W_{x_0^*}(1/2)-\mathcal W_{x_0^*}(2).
	\end{align*}
	\end{itemize}
\end{defn}

The radius functions $\rA$ and $\rBC$ are closely related to the conditions $(\star_L)$ and $(\dagger_L)$ introduced in \cite[Section 5]{li2023rigidity}; see also similar conditions in \cite[Section 7]{colding2021singularities}. Unlike in \cite{li2023rigidity}, however, we incorporate an additional integral condition \eqref{eq:intro007c} in the definition of $\rBC$. In the setting of \cite{li2023rigidity}, this condition is trivially satisfied. The $\r_{C, \delta}$-radius is essentially weaker than both $\rA$ and $\rBC$, and can be viewed as an analog of the graphical radius introduced in \cite[Section 0.6]{colding2015uniqueness}. In what follows, we will compare the three radii—$\rA$, $\rBC$, and $\r_{C, \delta}$—within the entropy radius $\rE$.
 
(I) Given $\sigma>0$, we have $\rA\geq \rBC-2$ if $\rBC$ is sufficiently large.

The proof of this statement is given in Proposition \ref{prop:con}, which applies to general weighted Riemannian manifolds. In the special case where $(M, g(-1), f(-1))$ is a Ricci shrinker, the same conclusion was established in \cite[Theorem 5.2]{li2023rigidity}, building on techniques developed in \cite[Section 3]{colding2021singularities}. Broadly speaking, the argument proceeds by applying a suitable cutoff to the pair $\lc \varphi_B^* g(-1) - \bar g, \varphi_B^* f(-1) - \bar f \rc$, followed by the construction of a modified diffeomorphism that yields improved estimates as a consequence of the rigidity inequality. Although $(M, g(-1), f(-1))$ is not assumed to be a Ricci shrinker in our setting, the assumption \eqref{eq:intro007c} is sufficient to obtain the necessary estimates.

From the above argument, one can derive the following estimate (see Theorem \ref{lojawithradius}), which is based on \eqref{intrlojaW} and the remainder estimate given in Proposition \ref{prop:remainder}:
	\begin{equation}\label{eq:intro007d}
		\left|\WW_{x_0^*}(1)-\Theta_{n-k} \right|\leq C(n, Y) \exp\lc-\frac{3\rBC^2}{16} \rc.
	\end{equation}
To apply \eqref{intrlojaW}, it is necessary to consider a cutoff of the pair $\lc \varphi_B^* g(-1) - \bar g, \varphi_B^* f(-1) - \bar f \rc$. Accordingly, Proposition \ref{prop:remainder} is used to show that the contribution to the $\WW$-functional from outside a compact set of size $L$ decays like $e^{-L^2/4}$.

(II) Suppose $\r_{C, \delta} < (1 - \sigma)\rE$. If $\delta$ is sufficiently small and $\r_{C, \delta}$ is sufficiently large, then the pair $\lc \varphi_C^* g(-1), \varphi_C^* f(-1) \rc$ is close to a Ricci shrinker in the sense that
	\begin{equation}\label{eq:intro007e}
\left[\frac{\varphi_C^* g(-1)}{2}-\varphi_C^*\na^2  f(-1)-\Ric\lc \varphi_C^* g(-1) \rc\right]_{100} \le C(n,Y,\sigma) e^{\frac{\bar f}{2}-\frac{\bar L^2}{8(1-\sigma)}},
	\end{equation}
on $\{\bar b \le \bar L\}$, where $\bar L=\r_{C, \delta}-\delta^{-1}$.

Here, the constant $100$ is not essential and can be replaced by any sufficiently large constant, provided the corresponding parameters are adjusted accordingly. The proof of \eqref{eq:intro007e} will be given in Proposition \ref{pointwiseshrinkeruantity}. This estimate shows that, as long as $\r_{C, \delta}$ does not exceed $\rE$, the Ricci shrinker operator $\Phi$ admits a pointwise bound within $\r_{C, \delta}$-radius. This is somewhat surprising, since the definition of the entropy radius a priori only yields a spacetime $L^2$ control on $\Phi$.

Through a series of delicate estimates, we demonstrate that a pointwise bound is nonetheless achievable. A key step is the observation that $\varphi_C^* f(-1)$ is nearly equal to $\bar f$, up to a multiplicative factor of $1 \pm \epsilon$ (see Proposition \ref{stabilitypotential}), the proof of which relies on an optimal heat kernel estimate (see Theorem \ref{thm:upper}). Then, by introducing an auxiliary function (see Proposition \ref{betterpotential}), we establish an $L^2$ estimate on the time slice $t = -1$, which ultimately leads to the desired pointwise bound in \eqref{eq:intro007e}.

(III) Suppose $\rA < (1 - \sigma)\rE$. Then for any sufficiently large constant $D > 1$, we have $\r_{C, \delta} \ge \rA + D$, provided $\rA$ is large enough. More precisely, the map $\varphi_A$ can be extended to a diffeomorphism $\tilde \varphi_A$ defined on the region $\{ \bar b \le \rA + D \}$, which satisfies the conditions required in the definition of $\r_{C, \delta}$.

Intuitively, this means that the domain of $\varphi_A$ can be extended to a larger region, at the expense of obtaining weaker estimates. This result will be established in Theorem \ref{thm:ext1}, one of the most technically involved parts of the paper. The key idea is to analyze the geometry near the boundary of the region where $\varphi_A$ is defined and to show that the pullback metric $\varphi_A^* g(-1)$ is locally almost splitting. The approximate splitting direction is almost generated by $\varphi_A^* \nabla f(-1)$, after an appropriate rescaling.

To realize this, we employ a limiting argument based on the structure theory of Ricci flow limit spaces developed in \cite{fang2025RFlimit}. The resulting limit space is the standard $\mathbb{R} \times \mathbb{R}^{k-1} \times S^{n-k}$, where the first factor corresponds to the splitting direction. Following the gradient flow of this approximate splitting direction, we are able to extend $\varphi_A$ to a larger region while maintaining the necessary control. A similar idea was used in \cite{LW25} in the context of Ricci shrinkers to identify a splitting direction.

We remark that this extension argument differs from those used in \cite[Theorem 7.1]{colding2021singularities} and \cite[Theorem 5.3]{li2023rigidity}, both of which rely heavily on pseudolocality and the self-similarity of Ricci shrinkers. Although the pseudolocality theorem remains applicable in our setting, the absence of self-similarity prevents us from deriving spatial estimates based on estimates at later times.

(IV) Suppose $\rA < (1 - 2\sigma)\rE$. Then $\rBC \ge \rA + D$, provided that $\rA$ is sufficiently large, where $D$ is a fixed large constant.

This can be viewed as another extension step and will be established in Theorem \ref{thm:ext2}. The overall strategy follows a similar approach to that in \cite[Theorem 5.3]{li2023rigidity} (see also \cite[Theorem 7.1]{colding2021singularities}). One key difference is that the additional neck region obtained in statement (III) is smaller than that in \cite[Theorem 5.3]{li2023rigidity}, where the neck has size $\ep \rA$. Nevertheless, the same gluing argument can be carried out within this narrower neck region.

Another distinction lies in the geometric setting: \cite[Theorem 5.3]{li2023rigidity} deals with a genuine Ricci shrinker, whereas here we do not assume that $(M, g(-1), f(-1))$ is a Ricci shrinker. However, this poses no essential difficulty, as \eqref{eq:intro007e} implies that $(M, g(-1), f(-1))$ is an almost Ricci shrinker in a very strong sense—sufficient for the estimates required in our argument.

Thus, by applying statements (I) and (IV) iteratively, one obtains the estimate
\begin{align}\label{eq:intro008}
\min\{ \mathbf{r}_A, \rBC \} \ge (1 - 3\sigma) \mathbf{r}_E,
\end{align}
provided that $\rA$ or $\rBC$ is sufficiently large. When combined with \eqref{eq:intro007d}, this leads to the Lojasiewicz inequality \eqref{eq:intro004} stated in Theorem \ref{lojaRFcyl}.  \qed
\\
\\
Based on Theorem \ref{lojaRFcyl}, we sketch the proof of how it leads to a strong uniqueness result of Theorem \ref{thm:stronguni2intro}.

\subsection*{Outline of the Proof of Theorem \ref{thm:stronguni2intro}}

In the setting of Theorem \ref{thm:stronguni2intro}, let $\rA(s)$, $\rBC(s)$, and $\rE(s)$ denote the corresponding radius functions associated with the triple $(M, g^z(s), f^z(s))$. By Theorem \ref{lojaRFcyl}, we have $\rE(s) \ge \sqrt{(12 - \ep) \log s}$, provided that $s$ is sufficiently large. Then, applying \eqref{eq:intro008}, we obtain
\begin{align}\label{eq:intro009}
\min\{\rA(s), \rBC(s)\} \ge \sqrt{(12-\ep) \log s}.
	\end{align}
Let $\varphi_s$ denote the diffeomorphism that appears in the definition of $\rA(s)$. Based on estimates \eqref{eq:intro007e} and \eqref{eq:intro009}, one can show that for any $\ep > 0$, there exists a sufficiently large constant $\bar s$ such that for all $s_0 \ge \bar s$, the map $\varphi_{s_0}$ defines a diffeomorphism from the region $\left\{ \bar b \le \sqrt{(8 - \ep)\log s_0} \right\}$ onto a subset of $M$, and for all $s \ge s_0$, the following estimate holds on $\left\{ \bar b \le \sqrt{(8 - \ep)\log s_0} \right\}$:
\begin{align}\label{eq:intro010}
\left[ \varphi_{s_0}^* g^z(s) - \bar g \right]_{[\ep^{-1}]} + \left[ \varphi_{s_0}^* f^z(s) - \bar f \right]_{[\ep^{-1}]} \le C(n, Y, \ep)  e^{\frac{\bar f}{2}} s^{-1 + \ep}.
\end{align}

Note that the estimate \eqref{eq:intro010} remains valid if we replace $s_0$ by $s$, but the key point here is that it also holds with a fixed diffeomorphism $\varphi_{s_0}$, independent of $s \ge s_0$. The full details will be presented in Theorem \ref{thm:stronguni}. Based on \eqref{eq:intro010}, one can construct the sequence of diffeomorphisms $\psi_j$ in Theorem \ref{thm:stronguni2intro} inductively, ensuring they are compatible on overlapping regions. \qed
\\
\\
The main results of this paper—including the strong uniqueness theorem (Theorem \ref{thm:stronguni2intro}) and the Lojasiewicz inequality (Theorem \ref{lojaRFcyl})—can be extended to the setting of quotient cylinders (see Section \ref{sec:lojaquo}). While the overall structure of the proofs remains similar, it is necessary to modify the definitions of $\rA$, $\rBC$, and $\r_{C, \delta}$ to suit the quotient cylinder context; see Definitions \ref{defnradiiquo} and \ref{defnrCdq}.

More precisely, rather than considering a diffeomorphism from the model space $\bar M/\Gamma$—where $\Gamma$ is a finite group acting freely on $\bar M$—into $M$, we instead consider a smooth covering map of degree $|\Gamma|$ from a domain in $\bar M$ onto a subset of $M$. This map is constructed so that the pullbacks of $(g(-1), f(-1))$ satisfy estimates analogous to \eqref{eq:intro007a} or \eqref{eq:intro007b}. A similar formulation was used in \cite[Theorem 5.5]{li2023rigidity}.

Following this modification, we can establish a Lojasiewicz inequality for quotient cylindrical singularities (see Theorem \ref{thm:loquotient}). Furthermore, the strong uniqueness of the associated tangent flow can also be proved (see Theorem \ref{thm:stronguni2quo}). Likewise, an analog of Corollary \ref{cor:sum} can be obtained in this setting (see Corollary \ref{quantisummabilityWq}).

In addition, one may consider ancient solutions to Ricci flow with bounded entropy, where a tangent flow at infinity is isometric to $\bar{\mathcal{C}}^k$ or its finite quotient. In this context, similar strong uniqueness theorems can be established by working with the appropriately modified Ricci flow (see \eqref{eq:mcf1}) and applying the Lojasiewicz inequality. See Theorems \ref{thm:stronguni-infty} and \ref{thm:stronguni-inftyquo} for the precise statements.

\subsection*{Organization of the Paper}

The paper is organized as follows.

\begin{itemize}
\item Section \ref{secprel} introduces the necessary definitions and basic properties of weighted Riemannian manifolds. We also present fundamental conventions and results for closed Ricci flows that will be used throughout the paper. In addition, we review key results on Ricci flow limit spaces from \cite{fang2025RFlimit}, introduce our model spaces, and prove Theorem \ref{thm: weakunique}.

\item Section \ref{sec:locesti} establishes the remainder estimate for integrals.

\item Section \ref{seclojaW} derives the variational inequality for the $\WW$-functional near the cylinder.

\item Section \ref{secradiicontraction} defines various radius functions and proves statement (I) as well as the estimate \eqref{eq:intro007d}.

\item Section \ref{seclojaRF}, which forms the core of the paper, contains the proofs of statements (II), (III), and (IV), thereby completing the proof of Theorem \ref{lojaRFcyl}.

\item Section \ref{sec:scyf} proves Theorem \ref{thm:stronguni2intro}, building on the results of Theorem \ref{lojaRFcyl}.

\item Section \ref{sec:lojaquo} extends the above results to the setting of quotient cylinders.

\item Finally, we include a list of notations for reference.
\end{itemize}

\textbf{Acknowledgements}: Hanbing Fang would like to thank his advisor, Prof. Xiuxiong Chen, for his encouragement and support. Hanbing Fang is supported by the Simons Foundation. Yu Li is supported by National Key R\&D Program of China 2025YFA1018200, NSFC-12522105, YSBR-001 and research funds from University of Science and Technology of China and Chinese Academy of Sciences.

\section{Preliminaries}\label{secprel}

\subsection*{Weighted Riemannian manifolds}

An $n$-dimensional \textbf{weighted Riemannian manifold} $(M^n, g, f)$ is a complete Riemannian manifold $(M^n, g)$ coupled with a smooth function $f: M \to \R$. For a weighted Riemannian manifold $(M^n, g, f)$, we define
\begin{align*}
	\VV(g,f):=(4\pi)^{-\frac{n}{2}}\int_Me^{-f}\,\mathrm{d}V_g.
\end{align*}\index{$\VV(g, f)$}
Here, $\mathrm{d}V_g$ denotes the volume form of $(M,g)$. For simplicity, we set $\mathrm{d}V_f:=(4\pi)^{-\frac{n}{2}}e^{-f}\,\mathrm{d}V_g$\index{$\mathrm{d}V_f$}. We say $(M^n,g,f)$ is \textbf{normalized} if it satisfies the following normalization:
\begin{align}\label{normalziation1}
	\VV(g,f)=1.
\end{align}

\begin{exmp}[Model space: weighted cylinders]\label{exmp:model}
For any $n \ge 3$ and $m\in \{2,\ldots,n\}$, we define 
\begin{align*}
\mathcal C^{n-m}_{-1}:=(\bar M,\bar g,\bar f)=\left(\R^{n-m}\times S^{m}, g_E \times g_{S^m}, \frac{|\vec{x}|^2}{4}+\frac{m}{2}+\Theta_m \right),
\end{align*}\index{$\CC^{n-m}_{-1}$}\index{$(\bar M,\bar g,\bar f)$}
where $g_E$ is the Euclidean metric on $\R^{n-m}$, and $g_{S^m}$ is the round metric on $S^{m}$ such that $\Ric(g_{S^m})=g_{S^m}/2$. The vector $\vec{x}=(x_1,\ldots,x_{n-m})$\index{$\vec{x}$} denotes the standard coordinate function on $\R^{n-m}$. The constant $\Theta_m$ is chosen to ensure that the normalization condition \eqref{normalziation1} is satisfied. More explicitly, it follows from a direct calculation that
\begin{align}\label{eq:entropyexplicit}
\Theta_m=\log\lc \frac{2^{1-\frac m 2} (m-1)^{\frac m 2} e^{-\frac m 2} \sqrt{\pi}}{\Gamma\lc \frac{m+1}{2}\rc} \rc.
\end{align}
\end{exmp}\index{$\Theta_m$}

\begin{defn}\label{defnshrinkerquant}\index{$\mathbf{\Phi}(g,f)$}\index{$\WW(g,f)$}
	For a weighted Riemannian manifold $(M^n,g,f)$, we define:
	\begin{align*}
		\mathbf{\Phi}(g,f):=&\frac{g}{2}-\Ric(g)-\na^2 f, \\
		\WW(g,f):=& \int_{M} 2\Delta f-|\nabla{f}|^2+\scal+f-n \, \mathrm{d}V_f.
	\end{align*}
\end{defn}

It is clear that $\mathbf{\Phi}(g, f)=0$ implies that the weighted Riemannian manifold $(M^n,g,f)$ is given by a Ricci shrinker. On the other hand, $\WW(g,f)$ is nothing but Perelman's $\WW$-functional $\WW(g, f, 1)$ (see \cite{perelman2002entropy}). For the cylinder $\mathcal C^{n-m}_{-1}=(\bar M,\bar g,\bar f)$ as in Example \ref{exmp:model}, we have
	\begin{align*}
		\mathbf{\Phi}(\bar g,\bar f)=0 \quad \text{and} \quad \WW(\bar g,\bar f)=\Theta_m.
	\end{align*}
In other words, $\Theta_m$ denotes the entropy of $(\bar M,\bar g,\bar f)$ as a Ricci shrinker. For more information regarding the entropy of Ricci shrinkers, we refer readers to \cite[Section 5]{li2020heat}.

For later applications, we have the following definition.

\begin{defn}\label{defnoperators}
Given a weighted Riemannian manifold $( M^n, g, f)$, let $T^{r, s} M$ denote the space of $(r, s)$-tensors on $M$, and let $S^2(M)$ denote the space of symmetric $(0, 2)$-tensors. We then define the following:
	\begin{enumerate}[label=\textnormal{(\roman{*})}]
	\item The $L^2$-norm for $(r,s)$-tensors is defined with respect to $\mathrm{d}V_f$, i.e., for any $T\in L^2(T^{r, s} M)$, $\rVert T\rVert_{L^2}^2=\int_{ M}|T|^2\,\mathrm{d}V_f$. The Sobolev spaces $W^{k,2}(T^{r,s} M)$ are defined with respect to $\mathrm{d}V_f$ as well.\index{$\|\cdot\|_{L^2}$}
	
		\item The weighted divergence operator $\Div_{ f}:C^\infty(T^{r,s} M)\to C^\infty (T^{r,s-1} M)$ is defined as 
		\begin{align*}
			(\Div_f T)^{i_1,\cdots,i_r}_{j_1,\cdots,j_{s-1}}:=\na_iT^{i_1,\cdots, i_r}_{i,j_1,\cdots, j_{s-1}}-T^{i_1,\cdots, i_r}_{i,j_1,\cdots, j_{s-1}}f_i,
		\end{align*}\index{$\Div_f$}
		for any $T\in C^\infty (T^{r,s} M)$. In particular, when $f$ is constant, then $\Div_f$ reduces to the standard divergence operator, which we denote by $\Div$\index{$\Div$}.
        We denote by $\Div_f^*$\index{$\Div_f^*$} the $L^2$-adjoint of $\Div_f$.
		
		\item For a tensor field $T$ on $M$, we set
		\begin{align}\label{defnnormtensor}
			[T]_l:=\sum_{k=0}^l \left|\na^k T\right|.
		\end{align}\index{$[\cdot]_l$}
		For any two tensor fields $T_1,T_2$ on $M$, we use the notation $T_1*T_2$ to denote any tensor field that is a real linear combination of terms obtained from $T_1\otimes T_2$ by contracting indices using the metric—i.e., by converting any number of $T^* M$ components to $TM$ components or vice versa.		
		
		\item The weighted Laplacian $\Delta_f$\index{$\Delta_f$} is defined as $\Delta_f:=\Delta-\la\na \cdot,\na f\ra$ on $C^\infty(T^{r,s} M)$. Denote $\LL$ to be the stability operator $\Delta_f+2\Rm$ on $C^{\infty}(S^2(M))$, i.e., for any such tensor $h \in C^{\infty}(S^2(M))$,
		\begin{align*}
			\LL (h)_{ij}:=\Delta_f h_{ij}+2\Rm_{likj}h_{lk}.
		\end{align*}\index{$\LL$}
	\end{enumerate}
\end{defn}

For weighted cylinders, we define the following space.

\begin{defn}\label{defncylinder}
Let $\mathcal C^{n-m}_{-1}=(\bar M,\bar g,\bar f)$ denote the weighted cylinder as described in Example \ref{exmp:model}. We define $\KK_0$ to be the linear space of quadratic Hermite polynomials of the form
		\begin{align*}
v=a_{ij}x_ix_j-2\Tr( a)
		\end{align*}
where $a=(a_{ij})$ is a symmetric matrix.
\end{defn}\index{$\KK_0$}

On the cylinder $\mathcal C^{n-m}_{-1}$, the $L^2$-kernel of the stability operator $\LL$ in $\ker(\Div_{\bar f})$ is $\KK_0 g_{S^m}$; see \cite[Section 6]{colding2021singularities} and \cite[Proposition 4.2]{li2023rigidity}.

We conclude this subsection with the following weighted Bianchi identity, which is well-known to experts.

\begin{lem}\label{weightedBianchilem}
For any weighted Riemannian manifold $( M^n, g, f)$ and constant $\tau>0$, the following identity holds:
	\begin{equation*}
		\nabla \lc \tau(2\Delta f-|\nabla f|^2+\scal)+f-n \rc=2 \Div_f \lc \tau \Ric+\nabla^2(\tau f)-\frac{g}{2} \rc.
	\end{equation*}
\end{lem}
\begin{proof}
By the traced Bianchi identity $\di(\Ric)=\nabla \scal/2$ and the identity $\di(\nabla^2 f)=\nabla \Delta f+\Ric(\nabla f)$, we have
	\begin{align*}
2 \Div_f \lc \tau \Ric+\nabla^2(\tau f)-\frac{g}{2} \rc &=2 \di(\tau \Ric+\tau\nabla^2 f)-2(\tau \Ric+\tau\nabla^2 f-\frac{g}{2})(\nabla f)\\
		&=\tau \nabla \scal+2\tau \nabla\Delta f-2\tau\nabla^2 f(\nabla f)+\nabla f=\nabla \lc \tau(2\Delta f-|\nabla f|^2+\scal)+f-n \rc.
	\end{align*}
\end{proof}

\subsection*{Basic conventions and results for Ricci flows}\label{secprelRF}

Throughout this paper, we consider a closed Ricci flow solution $\XX=\{M^n,(g(t))_{t\in I}\}$, where $M$ is an $n$-dimensional closed manifold, $I \subset \R$ is a closed interval, and $(g(t))_{t\in I}$ is a family of smooth metrics on $M$ satisfying for all $t\in I$ the Ricci flow equation:
\begin{equation*}
	\partial_tg(t)=-2\Ric(g(t)).
\end{equation*}

Following notations in \cite{fang2025RFlimit}, we use $x^*\in\XX$ to denote a spacetime point $x^* \in M \times I$, and define $\t(x^*)$ to be its time component. We denote by $d_t$ the distance function and by $\mathrm{d}V_{g(t)}$ the volume form induced by the metric $g(t)$. For $x^*=(x,t)\in \XX$, we write $B_t(x,r)$ for the geodesic ball centered at $x$ with radius $r$ at time $t$. The Riemannian curvature, Ricci curvature, and scalar curvature are denoted by $\Rm$, $\Ric$, and $\scal$, respectively, with the time parameter omitted when there is no ambiguity. Throughout this paper, we use $\Psi(\ep)$ to denote a function satisfying $\Psi(\ep) \to 0$ if $\ep \to 0$.\index{$\Psi(\ep)$}

For the smooth closed Ricci flow $\XX$, we set $K(x,t;y,s)$ to be the heat kernel\index{$K(x,t;y,s)$}, which is determined by:
  \begin{align*}
  \begin{cases}
    &\square K(\cdot,\cdot;y,s) =0,\\
    &\square^{*} K(x,t;\cdot,\cdot) =0,   \\
    &\lim_{t\searrow s} K(\cdot,t;y,s)=\delta_{y},\\
    &\lim_{s\nearrow t} K(x,t;\cdot,s)=\delta_{x}.
        \end{cases}
  \end{align*}
where $\square:=\partial_t-\Delta$ and $\square^*:=-\partial_t-\Delta+\scal$. 

\begin{defn}\label{def:chkm}
	The \textbf{conjugate heat kernel measure} $\nu_{x^*;s}$\index{$\nu_{x^*;s}$} based at $x^*=(x, t)$ is defined as
	\begin{align*}
		\mathrm{d}\nu_{x^*;s}=\mathrm{d}\nu_{x,t;s}:=K(x,t;\cdot,s)\,\mathrm{d}V_{g(s)}.
	\end{align*}		
	It is clear that $\nu_{x^*;s}$ is a probability measure on $M$. If we set 
	\begin{align*}
		\mathrm{d}\nu_{x^*;s}=(4\pi(t-s))^{-n/2}e^{-f_{x^*}(\cdot,s)}\,\mathrm{d}V_{g(s)},
	\end{align*}	
	then the function $f_{x^*}$\index{$f_{x^*}$} is called the \textbf{potential function} at $x^*$ which satisfies:
	\begin{equation}\label{evolutionoff}
		-\partial_s f_{x^*}=\Delta f_{x^*}-|\nabla f_{x^*}|^2+\scal-\frac{n}{2(t-s)}.
	\end{equation}	
\end{defn}

Next, we recall the definitions of the Nash entropy and the $\WW$-entropy based at a spacetime point $x^*$.

\begin{defn}\label{defnentropy}
	The \textbf{Nash entropy} based at $x^*\in \XX$ is defined by
	\begin{equation*}
		\NN_{x^*}(\tau):= \int_M f_{x^*}\,\mathrm{d}\nu_{x^*;\t(x^*)-\tau}-\frac{n}{2},
	\end{equation*}\index{$\NN_{x^*}(\tau)$}
	for $\tau>0$ with $\t(x^*)-\tau \in I$, where $f_{x^*}$ is the potential function at $x^*$. 
	Moreover, the $\WW$-entropy based at $(x,t)$ is defined by
	\begin{equation*}
		\WW_{x^*}(\tau):= \int_M\tau(2\Delta f_{x^*}-|\na f_{x^*}|^2+\scal)+f_{x^*}-n \,\mathrm{d}\nu_{x^*;\t(x^*)-\tau}.
	\end{equation*}\index{$\WW_{x^*}(\tau)$}
\end{defn}

 The following proposition gives basic properties of Nash entropy (see \cite[Section 5]{bamler2020entropy}):
\begin{prop}\label{propNashentropy}
	For any $x^* \in \XX$ with $\t(x^*)-\tau \in I$ and $\scal(\cdot, \t(x^*)-\tau) \ge R_{\min}$, we have the following inequalities:
	\begin{enumerate}[label=\textnormal{(\roman{*})}]
		\item $\displaystyle -\frac{n}{2\tau}+R_{\min} \leq \frac{\dd}{\dd\tau}\NN_{x^*}(\tau)\leq 0$;
		\item $\displaystyle \frac{\dd}{\dd\tau}\lc\tau \NN_{x^*}(\tau)\rc=\WW_{x^*}(\tau)\leq 0$;
		\item $\displaystyle \frac{\dd^2}{\dd\tau^2}\lc\tau\NN_{x^*}(\tau)\rc=-2\tau\int_M \abs{\Ric+\nabla^2 f_{x^*}-\frac{1}{2\tau}g}^2\,\mathrm{d}\nu_{x^*;\t(x^*)-\tau}\leq 0$.
	\end{enumerate}	
\end{prop}

As in \cite{fang2025RFlimit}, we have the following definition.

\begin{defn}\label{def:entropybound}
	A closed Ricci flow $\XX=\{M^n,(g(t))_{t\in I}\}$ is said to have entropy bounded below by $-Y$ at $x^*\in \XX$ if 
	\begin{align}\label{entropybd-Y}
		\inf_{\tau>0}\NN_{x^*}(\tau)\geq -Y,
	\end{align}
	where the infimum is taken over all $\tau>0$ whenever the Nash entropy $\NN_{x^*}(\tau)$ is well-defined. 
	
	Moreover, we say the Ricci flow $\XX$ has entropy bounded below by $-Y$ if \eqref{entropybd-Y} holds for all $x^*\in\XX$.
\end{defn} 

Next, we recall the following definition.

\begin{defn}\label{defncurvatureradius}
	For $x^*=(x,t)\in\XX$, the curvature radius $r_{\Rm}$ \index{$r_{\Rm}$}is defined as the supremum over all $r>0$ such that $|\Rm|\leq r^{-2}$ on the parabolic ball $B_t(x,r)\times [t-r^2,t+r^2] \cap I$. 
\end{defn}

Now we recall the following definition of $H$-center from \cite[Definition 3.10]{bamler2020entropy}.

\begin{defn} \label{defh-center}
	A point $(z,t) \in \XX$ is called an \textbf{$H$-center} of $x_0^* \in \XX$ for a constant $H>0$ if $t \in I$, $t<\t(x_0^*)$ and 
	\begin{align*}
		\Var_t(\delta_z,\nu_{x_0^*;t})\leq H(\t(x_0^*)-t).
	\end{align*}
	Here, $\Var_t$ denotes the variance between two probability measures with respect to $g(t)$; see \emph{\cite[Definition 2.1]{fang2025RFlimit}}. Note that by \emph{\cite[Corollary 3.8]{bamler2020entropy}}, an $H_n$-center, where $H_n:=(n-1)\pi^2/4+2$\index{$H_n$}, must exist for any $t<\t(x_0^*)$.
\end{defn}

We have the following result from \cite[Propositions 3.12, 3.13]{bamler2020entropy}.

\begin{prop}\label{existenceHncenter}
Any two $H$-centers $(z_1, t)$ and $(z_2, t)$ of $x_0^*$ satisfy $d_t(z_1, z_2) \le 2 \sqrt{H(\t(x_0^*)-t)}$. Moreover, if $(z,t)$ is an $H$-center of $x_0^* \in \XX$, then for any $L>0$, we have
	\begin{align*}
		\nu_{x^*_0;t}\lc B_t \lc z,\sqrt{LH (\t(x_0^*)-t)}\rc\rc\geq 1-L^{-1}.
	\end{align*}
\end{prop}

The following estimates are proved in \cite[Theorem 2.15]{fang2025RFlimit}, which have improved the corresponding results in \cite{bamler2020entropy} and will play important roles later.

\begin{thm}\label{thm:upper}
	Let $\XX=\{M^n,(g(t))_{t\in I}\}$ be a closed Ricci flow with $[t, t_0] \subset I$. Then for any $\epsilon>0,L>0$ and $(x_0,t_0) \in \XX$, the following statements hold.
		\begin{enumerate}[label=\textnormal{(\roman{*})}]
	\item We have
	\begin{align*}
		\nu_{x_0,t_0;t}\lc M\setminus B_t(z,L)\rc\leq C(n,\ep) \exp\lc-\frac{L^2}{(4+\epsilon)(t_0-t)} \rc. 
	\end{align*}
	
\item If $\scal(\cdot, t) \ge R_{\min}$, then for any $(y, t) \in \XX$,
	\begin{equation*}
		K(x_0,t_0;y,t)\leq \frac{C\lc n,R_{\min}(t_0-t),\ep\rc}{(t_0-t)^{n/2}}\exp\lc -\frac{d_t^2(z,y)}{(4+\epsilon)(t_0-t)}-\NN_{x_0,t_0}(t_0-t) \rc.
	\end{equation*}
\end{enumerate}	
Here, $(z,t)$ is any $H_n$-center of $(x_0,t_0)$.
\end{thm}

We also need the following two-sided pseudolocality from \cite[Theorem 10.1]{perelman2002entropy} and \cite[Theorem 2.47]{bamler2020structure}:

\begin{thm}[Two-sided pseudolocality theorem]\label{thm:twoside}
Let $\flow$ be a closed Ricci flow. For any $\alpha > 0$, there is an $\ep (n, \alpha) > 0$ such that the following holds.

Given $x_0^*=(x_0, t_0) \in \XX$ and $r > 0$ with $[t_0-r^2,t_0] \subset I$, if $|B_{t_0}(x_0,r)|_{t_0} \geq \alpha r^n$ and $|\Rm| \leq (\alpha r)^{-2}$ on $B_{t_0}(x_0,r)$, then
\begin{align*}
r_{\Rm}(x_0^*) \ge \ep r.
\end{align*}
\end{thm}

\subsection*{Ricci flow limit spaces}

In this subsection, we review the construction of noncollapsed Ricci flow limit spaces and their key properties, as developed in \cite{fang2025RFlimit}.

As in \cite{fang2025RFlimit}, we consider the moduli space $\MM(n, Y, T)$\index{$\MM(n,Y,T)$} of closed Ricci flows defined as follows:

\begin{defn}[Moduli space]
	For fixed constants $T \in (0, +\infty]$ and $Y>0$, the moduli space $\MM(n, Y, T)$ consists of all $n$-dimensional closed Ricci flows $\XX=\{M^n,(g(t))_{t \in \III^{++}}\}$ satisfying
	\begin{enumerate}[label=\textnormal{(\roman{*})}]
		\item $\XX$ is defined on $\III^{++}:=[-T, 0]$.
		
		\item $\XX$ has entropy bounded below by $-Y$ (cf. Definition \ref{def:entropybound}).
	\end{enumerate}
\end{defn}

In addition, we set 
	\begin{align*}
\III^+:=[-0.99T,0], \quad \III:=[-0.98 T,0], \quad \III^-:=(-0.98 T,0].
	\end{align*}
As noted in \cite{fang2025RFlimit}, these intervals can generally be chosen as $[-(1-\sigma)T, 0]$, $[-(1-2\sigma)T,0]$, and $(-(1-2\sigma)T,0]$, respectively, where $\sigma>0$ is an arbitrarily small parameter. For simplicity, we fix $\sigma=1/100$ in the present setting.

For any $\XX \in \MM(n, Y, T)$, we have the following definition of the spacetime distance on $M \times \III^+$:
\begin{defn}\label{defnd*distance}
	For any $x^*=(x,t), y^*=(y,s) \in M \times \III^+$ with $s \le t$, we define
	\begin{align*}
		d^*(x^*,y^*):=\inf_{r \in [\sqrt{t-s}, \sqrt{t+0.99 T})} \left\{r \mid d_{W_1}^{t-r^2} (\nu_{x^*;t-r^2},\nu_{y^*;t-r^2}) \le \ep_0 r \right\}.
	\end{align*}\index{$d^*$}
	If no such $r$ exists, we define $d^*(x^*,y^*):=\ep_0^{-1} d_{W_1}^{-0.99 T} (\nu_{x^*;-0.99 T},\nu_{y^*;-0.99 T})$. 
\end{defn}
Here, $\ep_0 \in (0, 1]$ is a small constant depending on $n$ and $Y$ (see \cite[Definition 3.3]{fang2025RFlimit}). By \cite[Lemma 3.7]{fang2025RFlimit}, $d^*$ defines a distance function on $M \times \III^+$, which coincides with the standard topology on $M \times \III^+$ (see \cite[Corollary 3.11]{fang2025RFlimit}). 

The following weak compactness theorem is proved in \cite[Theorem 1.3]{fang2025RFlimit}.

\begin{thm}[Weak compactness]\label{thm:intro1}
	Given any sequence $\XX^i=\{M_i^n,(g_i(t))_{t  \in \III^{++}}\} \in \MM(n, Y, T)$ with base points $p_i^* \in M_i \times \III$ \emph{(}when $T=+\infty$, we additionally assume $\limsup_{i \to \infty} \t_i(p_i^*)>-\infty$\emph{)}, by taking a subsequence if necessary, we obtain the pointed Gromov--Hausdorff convergence
	\begin{align*}
		(M_i \times \III, d^*_i, p_i^*,\t_i) \xrightarrow[i \to \infty]{\quad \mathrm{pGH} \quad} (Z, d_Z, p_{\infty},\t),
	\end{align*}
	where $d^*_i$ denotes the restriction of the $d^*$-distance on $M_i \times \III$, and $\t_i$ is the standard time-function on $M_i \times \III$. The limit space $(Z, d_Z,\t)$ is a complete, separable, locally compact metric space coupled with a $2$-H\"older continuous time-function $\t:Z \to \III$.
\end{thm}

The limit space $(Z, d_Z,\t)$ is referred to as a \textbf{noncollapsed Ricci flow limit space} over $\III$. Locally, the limit space is closely related to the $\F$-limit developed by Bamler in \cite{bamler2023compactness}. More precisely, for any $z \in Z$, we take a sequence of points $z_i^* \in M_i \times \III$ converging to $z$ in the Gromov--Hausdorff sense. By the theory of $\F$-convergence (see \cite{bamler2023compactness}), there exists a correspondence $\CF$ such that 
\begin{equation}\label{thm:intro2}
	(\XX^i, (\nu_{z_i^*;t})_{t \in [-T, \t^i(z_i^*)]}) \xrightarrow[i \to \infty]{\quad \IF, \CF\quad} (\XX^z, (\nu_{z;t})_{t \in  [-T, \t(z)]})
\end{equation}
such that the metric flow $\XX^z$ is future continuous for all $t \in [-T, \t(z)]$, except possibly at $t=-0.99 T$, where we require that the convergence \eqref{thm:intro2} is uniform. The metric flow $\XX^z$\index{$\XX^z$} is referred to as the \textbf{metric flow associated with $z$} with time-function denoted by $\t^z$. On $\XX^z_{\III^+}$, one can define a spacetime function $d_z^*$ as Definition \ref{defnd*distance}. In general, $d_z^*$ is only a pseudo-distance on $\XX^z_{\III^+}$. However, by passing to the corresponding quotient space $\widetilde{\XX^{z}_\III}$, one obtains an isometric embedding into the limit space $Z$ (cf. \cite[Theorem 1.4]{fang2025RFlimit}):

\begin{thm}\label{thm:iden}
	For any $z \in Z$, there exists an isometric embedding
	\begin{align*}
		\iota_z: (\widetilde{\XX^{z}_\III}, d^*_z) \longrightarrow (Z, d_Z)
	\end{align*}
	such that $\iota_z(z)=z$ and $\t \circ \iota_z=\t^z$, where $\t^z$ is the time-function on $\widetilde{\XX^{z}_\III}$. Moreover, for any $y_i^* \in \XX^i_\III$ and $y_{\infty} \in \XX^{z}_\III$, $y_i^*$ converge to $y_{\infty}$ within $\CF$ if and only if $y_i^* \to \iota_z(\tilde y_{\infty})$ in the Gromov--Hausdorff sense, where $\tilde y_{\infty}$ is the quotient of $y_{\infty}$ from $\XX^{z}_\III$ to $\widetilde{\XX^{z}_\III}$.	
\end{thm}

The limit space $Z$ contains a regular part $\RR$, which is a dense open subset of $Z_{\III^-}$ (see \cite[Corollary 5.7]{fang2025RFlimit}) and carries the structure of a Ricci flow spacetime $(\RR, \t, \partial_\t, g^Z)$. On this regular part, the convergence described in Theorem \ref{thm:intro1} is smooth, in the following sense (cf. \cite[Theorem 1.5]{fang2025RFlimit}):

\begin{thm}[Smooth convergence]\label{thm:intro3}
	There exists an increasing sequence $U_1 \subset U_2 \subset \ldots \subset \RR$ of open subsets with $\bigcup_{i=1}^\infty U_i = \RR$, open subsets $V_i \subset M_i \times \III$, time-preserving diffeomorphisms $\phi_i : U_i \to V_i$ and a sequence $\ep_i \to 0$ such that the following holds:
	\begin{enumerate}[label=\textnormal{(\alph{*})}]
		\item We have
		\begin{align*}
			\Vert \phi_i^* g^i - g^Z \Vert_{C^{[\ep_i^{-1}]} ( U_i)} & \leq \ep_i, \\
			\Vert \phi_i^* \partial_{\t_i} - \partial_{\t} \Vert_{C^{[\ep_i^{-1}]} ( U_i)} &\leq \ep_i,
		\end{align*}
		where $g^i$ is the spacetime metric induced by $g_i(t)$, and $\partial_{\t_i}$ is the standard time vector field induced by $\t_i$.
		
		\item Let $y \in \RR$ and $y_i^* \in M_i \times \III$. Then $y_i^* \to y$ in the Gromov--Hausdorff sense if and only if $y_i^* \in V_i$ for large $i$ and $\phi_i^{-1}(y_i^*) \to y$ in $\RR$.
		
		\item For $U_i^{(2)}=\{(x,y) \in U_i \times U_i \mid \t(x)> \t(y)+\ep_i\}$, $V_i^{(2)}=\{(x^*,y^*) \in V_i \times V_i \mid \t_i(x^*)> \t_i(y^*)+\ep_i\}$ and $\phi_i^{(2)}:=(\phi_i, \phi_i): U_i^{(2)} \to V_i^{(2)}$, we have
		\begin{align*}
			\Vert  (\phi_i^{(2)})^* K^i-K_Z \Vert_{C^{[\ep_i^{-1}]} ( U_i^{(2)})} \le \ep_i,
		\end{align*}	
		where $K^i$ and $K_Z$\index{$K_Z$} denote the heat kernels on $(M_i \times \III, g_i(t))$ and $(\RR, g^Z)$, respectively.
		
		\item If $z_i^* \in M_i \times \III$ converge to $z \in Z$ in Gromov--Hausdorff sense, then
\begin{align*}
K^i(z_i^*;\phi_i(\cdot)) \xrightarrow[i \to \infty]{C^{\infty}_\mathrm{loc}} K_Z(z;\cdot) \quad \text{on} \quad \RR_{(-\infty,\t(z))}.
\end{align*}	

		\item For each $t \in \III$, there are at most countable connected components of the time-slice $\RR_t$.
	\end{enumerate}
\end{thm}

For each $z \in Z$, we can assign a conjugate heat kernel measure $\nu_{z;s}:=K_Z(z;\cdot) \,\mathrm{d}V_{g^Z_s}$\index{$\nu_{z;s}$} based at $z$ for $s \le \t(z)$, which is a probability measure on $\RR_s$. All these probability measures together satisfy the reproduction formula (cf. \cite[Equation (5.5)]{fang2025RFlimit}).

Locally, one can define the tangent flow as follows.

\begin{defn}[Tangent flow] \label{def:tf}
	For any $z \in Z_{\III^-}$, a \textbf{tangent flow} at $z$ is a pointed Gromov--Hausdorff limit of $(Z, r_j^{-1} d_Z, z, r_j^{-2}(\t-\t(z)))$ for a sequence $r_j  \searrow 0$.
\end{defn}

It can be proved (cf. \cite[Section 7]{fang2025RFlimit}) that any tangent flow is a noncollapsed Ricci flow limit space, and the regular part admits the structure of Ricci shrinker (cf. \cite[Theorem 1.9]{fang2025RFlimit}).

On $Z_{\III^-}$, we have the following regular-singular decomposition:
\begin{align*}
	Z_{\III^-}=\RR_{\III^-} \sqcup \MS,
\end{align*}
where $\RR_{\III^-}$ denotes the restriction of $\RR$ on $\III^-$. It can be proved (see \cite[Theorem 7.15]{fang2025RFlimit}) that a point $z$ is a regular point if and only if any of its tangent flow is isometric to $(\R^{n} \times \R,d^*_{E, \ep_0},(\vec 0^n, 0),\t)$ or $(\R^{n} \times \R_{-},d^*_{E, \ep_0}, (\vec 0^n,0),\t)$, where $d^*_{E, \ep_0}$ denotes the induced $d^*$-distance on $\R^{n} \times \R$ using the spacetime distance constant $\ep_0$. Equivalently, $z$ is a regular point if and only if $\NN_z(0) \ge -\ep_n$ (see \cite[Proposition 7.7]{fang2025RFlimit}). It can be proved (see \cite[Theorem 1.13]{fang2025RFlimit}) that the Minkowski dimension of $\mathcal S$ with respect to $d_Z$ is at most $n-2$.

\subsection*{Cylindrical and almost cylindrical points}\label{sec:cylpts}

Let $\mathcal C^k_{-1}$ be as defined in Example \ref{exmp:model}. We define $\mathcal C^k=(\bar M, (\bar g(t))_{t<0}, (\bar f(t))_{t<0})$\index{$\mathcal C^k$}\index{$(\bar M, (\bar g(t))_{t<0}, (\bar f(t))_{t<0})$} to be the associated Ricci flow, where $t=0$ corresponds to the singular time, and the potential function is given by
\begin{align*}
\bar f(t):=\frac{|\vec{x}|^2}{4|t|}+\frac{n-k}{4}+\Theta_{n-k},
\end{align*}
where $\Theta_{n-k}$ is the same constant defined in \eqref{eq:entropyexplicit}.

As shown in \cite[Section 9]{fang2025RFlimit}, suppose that $\XX=\{M^n, (g(t))_{t \in [-T,0)}\}$ is a closed Ricci flow with entropy bounded below by $-Y$, where $0$ is the first singular time. We consider the $d^*$-distance on $\XX_{[-0.99T, 0)}$, defined in Definition \ref{defnd*distance}, using a fixed constant $\ep_0=\ep_0(n, Y)>0$. We then define
	\begin{align*}
(Z,d_{Z} ,\t)
	\end{align*}
to be the metric completion of $\XX_{[-0.98T, 0)}$ with respect to $d^*$. By construction, $(Z_{[-0.98T, 0)}, d_Z)=(\XX_{[-0.98T, 0)}, d^*)$; that is, the completion adds only the points in $Z_0$. Notice that $(Z,d_{Z} ,\t)$, which is called the \textbf{completion} of $\XX$, is a noncollapsed Ricci flow limit space over $[-0.98T, 0]$. 

Although the above construction is stated for closed Ricci flows, the same conclusion holds for any complete Ricci flow $\XX=\{M^n, (g(t))_{t \in [-T,0)}\}$ with entropy bounded below by $-Y$, provided the curvature is bounded on every compact time interval within $[-T, 0)$; see \cite[Section 11]{fang2025RFlimit}.

In particular, we set the completion of $(\bar M, (\bar g(t))_{t \in (-\infty, 0)})$ to be $\bar{\mathcal C}^k$\index{$\bar{\CC}^{n-m}$}, equipped with the spacetime distance $d_{\mathcal C}^*$\index{$d_{\mathcal C}^*$}. It is straightforward to verify that the metric completion adds only the singular set $\R^k \times \{0\}$ to $ \mathcal C^k$. We then define the base point $p^*$ as the limit of $(\bar p, t)$ as $t \nearrow 0$ with respect to $d_{\mathcal C}^*$, where $\bar p \in \bar M$ is a minimum point of $\bar f(-1)$. It is clear that $p^*$ is independent of the choice of $\bar p$. Moreover, for any $t<0$,
	\begin{align*}
\nu_{p^*;t}=(4\pi |t|)^{-\frac n 2} e^{-\bar f(t)} \,\mathrm{d}V_{\bar g(t)}.
	\end{align*}

In general, let $\XX=\{M^n, (g(t))_{t \in [-T,0)}\}$ be a closed Ricci flow with entropy bounded below by $-Y$, where $0$ is the first singular time. Suppose $(Z, d_Z, \t)$ is the completion of $\XX$. Fix a point $z \in Z_0$. Then the potential function $f_z$ is smooth for $t<0$. Moreover,
	\begin{align*}
\nu_{z;t}=(4\pi |t|)^{-\frac n 2} e^{-f_z} \,\mathrm{d}V_{g(t)}.
	\end{align*}

\begin{defn}[Modified Ricci flow]\label{def:mrf}
For any $z \in Z_0$, let $\phi_t$ be the family of diffeomorphisms generated by $-\na_{g(t)} f_z(t)$ with $\phi_{-T}=\mathrm{id}$. Then we define 
  \begin{equation*}
  \begin{dcases}
    &g^z (s):= e^s \phi^*_{-e^{-s}} g(-e^{-s}),   \\
    &f^z (s):=\phi^*_{-e^{-s}} f_z(-e^{-s}).\\
        \end{dcases}
  \end{equation*}
$(M, g^z(s),f^z(s))$\index{$(M, g^z(s),f^z(s))$} is called the \textbf{modified Ricci flow with respect to $z$}. It is clear that $(g^z(s),f^z(s))$ satisfies
\begin{align}\label{equationMRF}
  \begin{dcases}
& \partial_s g^z(s)=g^z-2 \Ric(g^z)-2\na^2 f^z, \\
&	\partial_s f^z(s)=\frac{n}{2}-\scal_{g^z}-\Delta f^z.
        \end{dcases}
\end{align}
\end{defn}

Next, we consider a general noncollapsed Ricci flow limit space $(Z, d_Z, \t)$ over $\III$, obtained as the limit of a sequence in $\mathcal M(n, T, Y)$. Then, we have the following definition.

\begin{defn}\label{def:ccc}
A point $z \in Z_{\III^-}$ is called \textbf{a cylindrical point with respect to $\bar{\mathcal C}^k$} if a tangent flow at $z$ (see Definition \ref{def:tf}) is isometric to $\bar{\mathcal C}^k$ for some $k$.
\end{defn}

The following theorem regarding the uniqueness of cylindrical tangent flows was proved in \cite[Theorem 6.2]{li2023rigidity}. Although \cite[Theorem 6.2]{li2023rigidity} is stated for a singular point at the first singular time, the same proof works for any singular point in a Ricci flow limit space. We sketch the proof for readers' convenience.

\begin{thm}[Uniqueness of the cylindrical tangent flow] \label{thm:weakunique}
Let $(Z, d_Z, \t)$ be a noncollapsed Ricci flow limit space over $\III$, obtained as the limit of a sequence in $\mathcal M(n, T, Y)$. For any $z \in Z_{\III^-}$, if a tangent flow at $z$ is isometric to $\bar{\mathcal C}^k$, then any tangent flow at $z$ is also $\bar{\mathcal C}^k$.
\end{thm}

\begin{proof}
We consider a metric flow $\XX^z$ associated with $z$. By our assumption, it implies that a tangent metric flow (see \cite[Definition 4.13]{fang2025RFlimit}) at $z$ is given by a cylinder $\mathcal C^k$. By the same proof of \cite[Theorem 6.2]{li2023rigidity}, it implies that any tangent metric flow at $z$ is $\mathcal C^k$. Now, we take a sequence $r_j \to 0$ such that
	\begin{align*}
(Z, r_j^{-1} d_Z, z, r_j^{-2}(\t-\t(z)))\xrightarrow[j \to \infty]{\quad \mathrm{pGH} \quad} (Z', d_{Z'}, z',\t')
	\end{align*}
	in the Gromov--Hausdorff sense.
	
By \cite[Theorem 4.17 (2)]{fang2025RFlimit}, and after passing to a further subsequence if necessary, the rescalings of $(\XX^z, \nu_{z;t})$ by $r_i^{-1}$, $\F$-converge to a tangent metric flow at $z$, which must be given by $\mathcal C^k$ as mentioned above. Thus, we conclude that $\XX^{z'}$, which is a metric flow associated with $z'$, is given by $\mathcal C^k$. 

From \cite[Theorem 7.19]{fang2025RFlimit}, we know that 
	\begin{align*}
Z'_{(-\infty, 0)}=\RR'_{(-\infty, 0)}=\mathcal C^k.
	\end{align*}
Consequently, it follows from \cite[Theorem 7.25]{fang2025RFlimit} that $Z'_{(0, +\infty)}=\emptyset$ and hence $Z'=\bar{\mathcal C}^k$.

In sum, the proof is complete.
\end{proof}

\begin{rem} 
It is clear that Theorem \ref{thm:weakunique} also holds in the case that one tangent flow is isometric to $\R^{k} \times N$ equipped with the standard metric, where $N$ is an $H$-stable Einstein manifold with the obstruction of order $3$.
\end{rem}

Next, we recall the following definition from \cite[Definition 5.37]{fang2025RFlimit}.

\begin{defn}[$\ep$-close]\label{defn:close}
Suppose $(Z, d_Z, z, \t)$ and $(Z', d_{Z'}, z',\t')$ are two pointed noncollapsed Ricci flow limit spaces, with regular parts given by the Ricci flow spacetimes $(\RR, \t, \partial_\t, g^Z)$ and $(\RR', \t', \partial_{\t'}, g^{Z'})$, respectively, such that $J$ is a time interval.

We say that $(Z, d_Z, z, \t)$ is \textbf{$\ep$-close} to $(Z', d_{Z'}, z',\t')$ \textbf{over $J$} if there exists an open set $U \subset \RR'_J$ and a smooth embedding $\phi: U \to \RR_J$ satisfying the following properties.
\begin{enumerate}[label=\textnormal{(\alph{*})}]
\item $\phi$ is time-preserving.

\item $U \subset B^*_{Z'}(z', \ep^{-1}) \bigcap \RR'_J$ and $U$ is an $\ep$-net of $B^*_{Z'}(z', \ep^{-1}) \bigcap Z'_J$ with respect to $d_{Z'}$.

\item For any $x, y \in U$, we have
	\begin{align*}
\abs{d_Z(\phi(x), \phi(y))-d_{Z'}(x, y)} \le \ep.
	\end{align*}
	
\item The $\ep$-neighborhood of $\phi(U)$ with respect to $d_Z$ contains $B^*_{Z}(z, \ep^{-1}-\ep) \bigcap Z_J$.

\item There exists $x_0 \in U$ such that $d_{Z'}(x_0, z') \le \ep$ and $d_{Z}(\phi(x_0), z) \le \ep$.

\item On $U$, the following estimates hold:
  \begin{align*}
  	\rVert \phi^* g^Z-g^{Z'}\rVert_{C^{[\ep^{-1}]}(U)}+\rVert \phi^* \partial_\t-\partial_{\t'} \rVert_{C^{[\ep^{-1}]}(U)} \le \ep.
  \end{align*} 
\end{enumerate}
\end{defn}

Next, we define

\begin{defn}\label{def:almost0}
Let $(Z, d_Z, \t)$ be a noncollapsed Ricci flow limit space over $\III$, obtained as the limit of a sequence in $\mathcal M(n, T, Y)$. A point $z \in Z_{\III^-}$ is called \textbf{$(k,\ep,r)$-cylindrical} if $\t(z)-\ep^{-1} r^2 \in \III^-$ and
	  \begin{align*}
(Z, r^{-1} d_Z, z, r^{-2}(\t-\t(z))) \quad \text{is $\ep$-close to} \quad (\bar{\mathcal C}^k ,d_{\mathcal C}, p^*,\t) \quad \text{over} \quad [-\ep^{-1}, \ep^{-1}].
  \end{align*} 
\end{defn}

The following lemma follows immediately from smooth convergence and Definition \ref{def:almost0}.

\begin{lem} \label{lem:smoothcyl}
Let $(Z, d_Z, \t)$ be a noncollapsed Ricci flow limit space over $\III$, obtained as the limit of a sequence in $\mathcal M(n, T, Y)$. For any $\delta>0$, if $\ep \le \ep(n, Y, \delta)$ and $z \in Z_{\III^-}$ is $(k,\ep,r)$-cylindrical, then there exists a diffeomorphism from $\{\bar b \le \delta^{-1}\} \subset \bar M$, where $\bar b=2\sqrt{\bar f(-1)}$, onto a subset of $\RR$ such that
\begin{align}\label{smoothcylindrical}
\| r^{-2} \varphi^* g^Z_{r^2t}-\bar g(t)\|_{C^{[\delta^{-1}]}}+\| \varphi^* f_z(r^2 t)-\bar f(t)\|_{C^{[\delta^{-1}]}} \le \delta,
\end{align}
on $\{\bar b \le \delta^{-1}\} \times [-\delta^{-1}, -\delta]$.
\end{lem}

\section{Remainder estimates for integral quantities}\label{sec:locesti}

Throughout this section, let $\flow$ be a closed Ricci flow. We fix a spacetime point $x_0^*=(x_0,t_0)\in \XX$ and assume $\XX$ has entropy bounded below by $-Y$ at $x_0^*$ (see Definition \ref{def:entropybound}). Moreover, we set
\begin{align*}
\mathrm{d}\nu_t=\mathrm{d}\nu_{x^*_0;t}=(4\pi\tau)^{-n/2}e^{-f}\,\mathrm{d}V_{g(t)},
\end{align*}
where $\tau=t_0-t$ and $f=f_{x_0^*}$. For simplicity, we define
  \begin{equation*} 
  \begin{dcases}
    &w:= \tau(2\Delta f-|\nabla f|^2+\scal)+f-n,   \\
    &\TT:=\tau \Ric+\nabla^2(\tau f)-\frac{g}{2}.\\
        \end{dcases}
  \end{equation*}
Recall that Perelman's differential Harnack inequality (see \cite[Section 9]{perelman2002entropy}) states that $w \le 0$. For the next proposition, we obtain the remainder estimates for the Nash entropy and $\WW$-entropy.

\begin{prop}\label{prop:remainder}
Suppose $[t,t_0]\subset I$ and $\scal \ge R_{\min}$ on $M \times [t, t_0]$. For any $\epsilon>0$, there exists a constant $C_\ep=C(n,Y,R_{\min}(t_0-t),\ep)$ such that 
	\begin{equation*}
		\int_{M\setminus B_t( z,L)}\tau\lc|\nabla f|^2+|\scal|+|\Delta f|\rc+\abs{f}\,\mathrm{d}\nu_{t}\leq C_\ep \exp\left(-\frac{L^2}{(4+\epsilon)(t_0-t)}\right)
	\end{equation*}
	for any $L >0$, where $(z,t)$ is any $H_n$-center of $x_0^*$ (see Definition \ref{defh-center}). 
\end{prop}

\begin{proof}
Without loss of generality, we assume $t_0-t=1$.

From \cite[Proposition 6.5]{bamler2020structure}, it follows that for any $\alpha\in [0,\frac{1}{2}]$,
	\begin{equation*}
		\int_{M}e^{\alpha f}\,\mathrm{d}\nu_t\leq e^{(n-R_{\min})\alpha}.
	\end{equation*}
	By Theorem \ref{thm:upper} (ii), it holds that $f\geq -C(n,Y)$, and thus, for any $p\geq 1$,
	\begin{align}\label{localizationentropy2}
		\int_M |f|^p\,\mathrm{d}\nu_t\leq C(n,Y,p) \int _M e^{f/2}\,\mathrm{d}\nu_t\leq C(n,Y,R_{\min},p). 
	\end{align}
	By H\"older's inequality, \eqref{localizationentropy2} and Theorem \ref{thm:upper} (i), we obtain that for $p^{-1}+q^{-1}=1$,
	\begin{align*}
		\int_{M\setminus B_t(z,L)}|f|\,\mathrm{d}\nu_t&\leq \lc\int_{M\setminus B_t(z,L)}|f|^p\,\mathrm{d}\nu_t\rc^{1/p}\lc\int_{M\setminus B_t(z,L)}\,\mathrm{d}\nu_t\rc^{1/q}\\
		&\leq C(n,Y,p,R_{\min},\ep)\exp\left(-\frac{L^2}{q(4+\ep)}\right).
	\end{align*}
	If we choose $q$ to be close to $1$ and modify constants, we obtain
	\begin{equation}\label{eq:local0}
		\int_{M\setminus B_t( z,L)}\abs{f}\,\mathrm{d}\nu_{t}\leq C_\ep \exp\left(-\frac{L^2}{4+\epsilon}\right).
	\end{equation}	
	
Next, we choose a cut-off function $\eta:\R\to [0,1]$ such that $\eta(s)=0$ if $s\leq L$ and $\eta(s)=1$ if $s\geq L+1$. Moreover, $|\eta'|^2 \leq 100 \eta$. For $x\in M$, we also denote $\psi(x):=\eta(d_t(x,z))$.
Using integration by parts, we obtain:
	\begin{align*}
		\int_{M}(|\nabla f|^2+\scal)\psi \,\mathrm{d}\nu_t&=  \int_{ M}(2\Delta f-|\nabla f|^2+\scal)\psi \,\mathrm{d}\nu_t+2\int_{M}\la\na f,\nabla\psi\ra \,\mathrm{d}\nu_t\\ 
		&\leq \int_{M}(n-f) \psi\,\mathrm{d}\nu_t+\frac{1}{2}\int_M|\na f|^2\psi \,\mathrm{d}\nu_t+2\int_M \frac{|\na \psi|^2}{\psi}\,\mathrm{d}\nu_t\\
		&\leq \int_{M\setminus B_t(z,L)}n+|f|\,\mathrm{d}\nu_t+\frac{1}{2}\int_M|\na f|^2\psi \,\mathrm{d}\nu_t+200\int_{M\setminus B_t(z,L)}\,\mathrm{d}\nu_t\\
		&\leq C_\ep \exp\left(-\frac{L^2}{4+\epsilon}\right)+\frac{1}{2}\int_M|\na f|^2\psi \,\mathrm{d}\nu_t,
	\end{align*}
where in the second inequality we used the fact that $w \le 0$, and in the last inequality we applied \eqref{eq:local0} and Theorem \ref{thm:upper} (i). Thus, we obtain
	\begin{align} \label{eq:local1}
		\int_{M\setminus B_t(z,L)}|\nabla f|^2+|\scal|\,\mathrm{d}\nu_t\leq C_\ep \exp\left(-\frac{L^2}{4+\epsilon}\right).
	\end{align}
	By $w \le 0$ again and integration by parts, we have
	\begin{align*}
		2\int_M |\Delta f|\psi \,\mathrm{d}\nu_t&\leq \int_M \lc|w|+|\na f|^2+|\scal|+|f|+n\rc\psi \,\mathrm{d}\nu_t\\
		&=\int_M \lc-w+|\na f|^2+|\scal|+|f|+n\rc\psi \,\mathrm{d}\nu_t\\
		&\leq \int_M \left(-2\Delta f+2|\na f|^2+2|\scal|+2|f|+2n\right)\psi \,\mathrm{d}\nu_t\\
		&\leq 2\int_M (|\na f|^2+|\scal|+|f|+n)\psi \,\mathrm{d}\nu_t +\int_{M\setminus B_t(z,L)} \frac{|\na \psi|^2}{\psi}\mathrm{d}\nu_t \leq C_\ep \exp\left(-\frac{L^2}{4+\epsilon}\right),
	\end{align*}
	where we used \eqref{eq:local0}, \eqref{eq:local1} and Theorem \ref{thm:upper} (i).
	
In sum, the proof is complete.
\end{proof}

\section{Variations of the \texorpdfstring{$\mathcal{W}$}{W}-functional near cylinders}\label{seclojaW}

In this section, we aim to derive the variations of the $\WW$-functional (as in Definition \ref{defnshrinkerquant}) of a weighted Riemannian manifold near a weighted cylinder.

Throughout this section, we fix $n \ge 3$, $m \in \{2, \ldots, n-1\}$, and consider the cylinder\index{$\CC^{n-m}$}\index{$(\bar M,\bar g,\bar f)$}
\begin{align}\label{eq:model}
\mathcal C^{n-m}_{-1}=(\bar M,\bar g,\bar f)=\left(\R^{n-m}\times S^{m}, g_E \times g_{S^m}, \frac{|\vec{x}|^2}{4}+\frac{m}{2}+\Theta_m \right),
\end{align}
as in Example \ref{exmp:model}. In the following, the Sobolev spaces are defined with respect to the weighted measure $\mathrm{d}V_{\bar f}$ (cf. Definition \ref{defnoperators}). The constant $C$ appearing in this section depends only on $n$, and may vary from line to line.

Consider a compactly supported pair $(h,\chi) \in C^2(S^2(\bar M)) \times C^2(\bar M)$ such that 
\begin{enumerate}[label=\textnormal{(\roman{*})}]
	\item $(\bar M, \bar g+h,\bar f+\chi)$ is a weighted Riemannian manifold;
	
\item 	$\displaystyle \|h\|_{C^2}+\|\chi\|_{C^2}:=\delta \ll 1$.
\end{enumerate}	

As in \cite{colding2021singularities} and \cite{li2023rigidity}, we have the following decomposition:
\begin{align}\label{eq:decompose}
h=ug_{S^m}+\zeta,\quad \chi=\frac{m}{2} u+q,
\end{align}
 where $ug_{S^m}$ is the projection onto $\KK_0 g_{S^m}$ and $\KK_0$ is the space of quadratic Hermite polynomials (see Definition \ref{defncylinder}). For simplicity, we set
\begin{align*}
\alpha:=\|u\|_{L^2}.
\end{align*} 
  
For $u \in \KK_0$, the following estimate holds (see \cite[Lemma 6.4]{colding2021singularities}):
\begin{lem}\label{estimateJacobifield}
	There exists constant $C=C(n)$ such that the following holds for $u\in \KK_0$ on $\R^{n-m}$:
	\begin{equation*}
		|u|+|\vec{x}||\nabla u|+(1+|\vec{x}|^2)| \na^2u|\leq C (1+|\vec{x}|^2) \alpha.
	\end{equation*}
\end{lem}

We also need the following estimate from \cite[Lemma 3.6]{li2023rigidity}:
\begin{lem}\label{concentrationinequ}
	For any $k\geq 0$, there exists a constant $C_k=C_k(n)$ such that for any tensor $T\in W^{k,2}(T^{r,s}\bar M)$,
	\begin{align*}
		\int_{\bar M}\bar f^k|T|^2\,\mathrm{d}V_{\bar f}\leq C_k\rVert T\rVert_{W^{k,2}}^2.
	\end{align*}
\end{lem}

Next, we recall that the center of mass vector $\BB(h,\chi)=\left(\BB_1(h,\chi),\cdots,\BB_{n-m}(h,\chi)\right)$\index{$\BB(h,\chi)$} is defined as:
\begin{align}\label{defncentermass}
	\BB_i(h,\chi):=\int_{\bar M}\bigg\la\partial_{x_i},\bar \na \left(\frac{1}{2}\Tr_{\bar g}(h)-\chi\right) \bigg\ra \,\mathrm{d}V_{\bar f}=\frac{1}{2}\int_{\bar M} x_i\left(\frac{1}{2}\Tr_{\bar g}(h)-\chi\right) \,\mathrm{d}V_{\bar f}.
\end{align}

The following two rigidity inequalities are obtained in \cite[Theorem 9.1]{colding2021singularities} (see also \cite[Proposition 4.9, Theorem 4.10]{li2023rigidity}), provided that $\delta$ is sufficiently small. Again, the norms of the Sobolev spaces are defined with respect to $\mathrm{d}V_{\bar f}$.
\begin{prop} \label{prop:tay1} 
	There exists a constant $C=C(n)$ such that 
	\begin{align*} 
		\|\zeta\|_{W^{2,2}}+\|\bar \na q\|_{W^{1,2}} \le C \left( \|\mathbf{\Phi}(\bar g+h, \bar f+\chi)\|_{L^{2}}+\|\Div_{\bar f} h\|_{W^{1,2}}+|\mathcal B(h,\chi)| +\rVert u\rVert_{L^2}^2\right).
	\end{align*}
\end{prop} 

\begin{prop} \label{prop:tay2}
	For any $\ep>0$,
	there exist positive constants $C=C(n),C_{\ep}=C(n,\ep)$ such that 
	\begin{align*}
		\alpha^2 \le C \|(1+|\vec{x}|)^2\mathbf{\Phi}(\bar g+h, \bar f+\chi)\|_{L^1} +C_{\ep} \left( \|\mathbf{\Phi}(\bar g+h, \bar f+\chi)\|^{2-\ep}_{L^2}+\|\Div_{\bar f} h\|^{2-\ep}_{W^{1,2}}+|\mathcal B(h,\chi)|^{2-\ep} \right).
	\end{align*}
\end{prop} 

For simplicity, we define 
	\begin{align*} 
\mathfrak{X}:=\|\mathbf{\Phi}(\bar g+h, \bar f+\chi)\|_{L^2}+\|\Div_{\bar f} h\|_{W^{1,2}}+|\mathcal B(h,\chi)|.
	\end{align*}\index{$\mathfrak{X}$}

Since $\|h\|_{C^2}+\|\chi\|_{C^2}=\delta$, we can easily get $\alpha+\mathfrak X \leq C\delta\ll 1$ and $|\zeta|+|q|\leq C\delta(1+|\vec x|^2)$. Moreover, by H\"older's inequality, we have
	\begin{align*}
\|(1+|\vec{x}|)^2\mathbf{\Phi}(\bar g+h, \bar f+\chi)\|_{L^1} \le C \|\mathbf{\Phi}(\bar g+h, \bar f+\chi)\|_{L^2} \le C \mathfrak X.
	\end{align*}
Thus, it follows from Propositions \ref{prop:tay1} and \ref{prop:tay2} that

\begin{cor} \label{cor:tay3}
We have
	\begin{align*}
\|\zeta\|_{W^{2,2}}+\|\bar \na q\|_{W^{1,2}} \le C \mathfrak{X} \quad \text{and} \quad \alpha^2 \le C \mathfrak X.
	\end{align*}
\end{cor}

Now, we consider the families for $s \in [0, 1]$:
\begin{align*}
	g(s) = \bar g+s h, \quad \text{and} \quad  f(s) = \bar f+s \chi.
\end{align*}
In addition, for the weighted Riemannian manifold $(\bar M,g(s),f(s))$, we define
\begin{align*}
  \begin{dcases}
  		& \mathrm{d}V_s:=\mathrm{d}V_{f(s)},\\
		&V(s)=\int_{\bar M} 1 \,\mathrm{d}V_{s},\\
		&Q(s)=\frac{1}{2} \Tr_{g(s)} h-\chi,\\
		&w(s):= 2\Delta_{g(s)} f(s)-|\nabla_{g(s)} f(s)|^2+\scal\left(g(s)\right)+f(s)-n,\\
		&W(s)=\WW(g(s),f(s))= \int_{\bar M}  w(s) \,\mathrm{d}V_s,\\
		&\Phi(s)=\mathbf{\Phi}(g(s), f(s))=\frac{g(s)}{2}-\Ric(g(s))-\na^2_{g(s)} f(s).
\end{dcases}
\end{align*}
For the rest of this section, all geometric quantities such as $\Ric$, $\na$, $\na^2$, and $\Div_f$ are defined with respect to the pair $(g(s), f(s))$, though we omit the subscript $g(s)$ and the parameter $s$ for simplicity. The norms $\|\cdot\|_{L^1}$, $\|\cdot\|_{L^2}$, $|\cdot|$ and $[\cdot]_l$ (see \eqref{defnnormtensor}) are taken with respect to $\mathrm{d}V_{s}$ and $g(s)$, respectively. Notice that since $\|h\|_{C^2}+\|\chi\|_{C^2}$ is sufficiently small, the volume form $\mathrm{d}V_s$ is uniformly comparable to $\mathrm{d}V_0=\mathrm{d}V_{\bar f}$, and the norms $[\cdot]_l$ are uniformly comparable to the corresponding norms defined with respect to $\bar g$ for $l \le 2$.

First, we recall the following result; see \cite[Lemma 2.3]{cao2012second} for instance.

\begin{lem} \label{lem:first1}
We have
\begin{align*}
	\Phi'(0)=\frac{1}{2}\LL h+\Div_{\bar f}^*\Div_{\bar f}h+\na^2 Q(0).
\end{align*}
\end{lem}

Next, we estimate:

\begin{lem}\label{estideriPhi}
	There exist a constant $C=C(n)$ such that the following inequalities hold for any $s\in [0,1]$:
	\begin{enumerate}[label=\textnormal{(\roman{*})}]
		\item 
		$\displaystyle \left|\Phi'(s)\right|\leq C\left([h]_2+[\na\chi]_1+(1+|\vec{x}|)[h]_1\right)$,
		\item 
		$\displaystyle \left|\Phi''(s)\right|\leq C\left( [h]_1^2+[h]_0[h]_2+(1+|\vec{x}|)[h]_0[h]_1+[h]_1[\na\chi]_0\right)$.
	\end{enumerate}
\end{lem}
\begin{proof}
	Recall the following local expression of $\Ric,\na^2 f$,
	\begin{align*}
		\Ric_{ij}=\partial_l\Gamma^l_{ij}-\partial_i\Gamma^l_{lj}+\Gamma^a_{ij}\Gamma^{l}_{la}-\Gamma^a_{lj}\Gamma^l_{ia},\quad \na_{ij}^2 f=\partial_i\partial_j f-\Gamma^k_{ij}\partial_k f.
	\end{align*}
	Set $C^{k}_{ij}:=\left(\Gamma^k_{ij}\right)_s$ and $D^k_{ij}:=\left(\Gamma^k_{ij}\right)_{ss}$. Since $\Gamma^k_{ij}=\frac{1}{2}g^{kl}\left(\partial_ig_{jl}+\partial_jg_{il}-\partial_lg_{ij}\right)$, we have for $s\in [0,1]$, 
	\begin{align}\label{Cijkestimate}
		C^k_{ij}=&\frac{1}{2}g^{kl}\left(\partial_i(g_s)_{jl}+\partial_j(g_s)_{il}-\partial_l(g_s)_{ij}\right)+\frac{1}{2}(g^{kl})_s\left(\partial_ig_{jl}+\partial_jg_{il}-\partial_lg_{ij}\right)\nonumber\\
		=&\frac{1}{2}g^{kl}\left(\partial_ih_{jl}+\partial_jh_{il}-\partial_lh_{ij}\right)-\frac{1}{2}h^{kl}\left(\partial_ig_{jl}+\partial_jg_{il}-\partial_lg_{ij}\right)
	\end{align}
and
	\begin{align}\label{Dijkestimate}
		D^k_{ij}=&\frac{1}{2}(g^{kl})_s\left(\partial_ih_{jl}+\partial_jh_{il}-\partial_lh_{ij}\right)-\frac{1}{2}(h^{kl})_s\left(\partial_ig_{jl}+\partial_jg_{il}-\partial_lg_{ij}\right)\nonumber\\
		&-\frac{1}{2}h^{kl}\left(\partial_i(g_s)_{jl}+\partial_j(g_s)_{il}-\partial_l(g_s)_{ij}\right)\nonumber\\
		=&-\frac{1}{2}h^{kl}\left(\partial_ih_{jl}+\partial_jh_{il}-\partial_lh_{ij}\right)+h^{kj}h_{ji}g^{il} \left(\partial_ig_{jl}+\partial_jg_{il}-\partial_lg_{ij}\right)\nonumber\\
		&-\frac{1}{2}h^{kl} \left(\partial_ih_{jl}+\partial_jh_{il}-\partial_lh_{ij}\right).
	\end{align}
In particular, we have
	\begin{align}\label{Cijkestimate1}
\abs{C^k_{ij}} \le C[h]_1 \quad \text{and} \quad \abs{D^k_{ij}} \le C[h]_0 [h]_1.
	\end{align}

	Thus, by \eqref{Cijkestimate}, \eqref{Dijkestimate} and \eqref{Cijkestimate1}, we obtain
	\begin{align}\label{Ricestimate2}
		\left|(\Ric_{ij})_s\right|&=\left|\partial_l C^l_{ij}-\partial_iC^l_{lj}+C^{a}_{ij}\Gamma^l_{la}+\Gamma^{a}_{ij}C^l_{la}-C^a_{lj}\Gamma^l_{ia}-\Gamma^a_{lj}C^l_{ia}\right|\leq C[h]_2
	\end{align}
	and
	\begin{align}\label{Ricestimate1}
		\left|(\Ric_{ij})_{ss}\right|&=\left|\partial_lD^l_{ij}-\partial_iD^l_{lj}+D^a_{ij}\Gamma^{l}_{la}+\Gamma^a_{ij}D^l_{la}+2C^a_{ij}C^l_{la}-D^a_{lj}\Gamma^l_{ia}-\Gamma^a_{lj}D^l_{ia}-2C^a_{lj}C^l_{ia}\right|\nonumber\\
		&\leq C\left( [h]_1^2+[h]_0[h]_2\right).
	\end{align}
	
	Similarly, we have
	\begin{align}\label{na2festimate2}
		\left|(f_{ij})_s\right|&=\left|\partial_i\partial_j f_s-C^k_{ij}\partial_k f-\Gamma^k_{ij}\partial_kf_s\right|\leq C\left([\na\chi]_1+(1+|\vec{x}|)[h]_1\right),
	\end{align}
	and
	\begin{align}\label{na2festimate}
		\left|(f_{ij})_{ss}\right|&=\left|\partial_i\partial_j f_{ss}-D^k_{ij}\partial_k f-C^k_{ij}\partial_k f_s-C^k_{ij}\partial_k f_s-\Gamma^k_{ij}\partial_kf_{ss}\right|\nonumber\\
		&=\left|D^k_{ij}\partial_k f+2C^k_{ij}\partial_k f_s\right|\leq C\left((1+|\vec{x}|)[h]_0[h]_1+[h]_1[\na\chi]_0\right).
	\end{align}
	Here, we have used the fact that $|\partial_k \bar f|\leq C |\vec{x}|$. Note that
	\begin{align*}
		\Phi'(s)=\frac{g_{s}}{2}-\Ric_{s}-(\na^2 f)_{s},\quad \Phi''(s)=\frac{g_{ss}}{2}-\Ric_{ss}-(\na^2 f)_{ss}.
	\end{align*}
By \eqref{Ricestimate2}, \eqref{Ricestimate1}, \eqref{na2festimate2} and \eqref{na2festimate}, it follows that
	\begin{align*}
		|\Phi'(s)|&\leq C\left([h]_2+[\na\chi]_1+(1+|\vec{x}|)[h]_1\right),\\
		|\Phi''(s)|&\leq C\left( [h]_1^2+[h]_0[h]_2+(1+|\vec{x}|)[h]_0[h]_1+[h]_1[\na\chi]_0\right).
	\end{align*}
	This completes the proof.
\end{proof}

\begin{cor}\label{keycal1}
	There exists a constant $C=C(n)>0$ such that the following integral estimates hold for any $s\in [0,1]$:
	\begin{enumerate}[label=\textnormal{(\roman{*})}]
		\item $\rVert\Phi''*h\rVert_{L^1}\leq C\left(\alpha^3+\alpha^2 \mathfrak{X}+\delta\mathfrak{X}^2\right)$,
		\item $\rVert\Phi'*h*h\rVert_{L^1}\leq C\left(\alpha^3+\alpha^2 \mathfrak{X}+\delta\mathfrak{X}^2\right)$,
		\item $\rVert\Phi*h*h\rVert_{L^1}\leq C\left(\alpha^3+\alpha^2 \mathfrak{X}+\delta\mathfrak{X}^2\right)$.
	\end{enumerate}
\end{cor}
\begin{proof}
	
	By Lemma \ref{estideriPhi}, we have
	\begin{align}\label{phiintegral1}
		\rVert\Phi''*h\rVert_{L^1}&\leq C \int_{\bar M}  [h]_1^2[h]_0+[h]_0^2[h]_2+(1+|\vec{x}|)[h]_0^2[h]_1+[h]_0[h]_1[\na\chi]_0 \,\mathrm{d}V_s.
	\end{align}
	
	Recall that $h=ug_{S^m}+\zeta$, $\chi=mu/2+q$ and $\delta=\|h\|_{C^2}+\|\chi\|_{C^2}$. By Corollary \ref{cor:tay3}, we have
	\begin{align}\label{phiintegral1xa}
		\rVert\zeta\rVert_{W^{2,2}}+\rVert\na q\rVert_{W^{1,2}}\leq C\mathfrak{X}.
	\end{align}
	

We first estimate
	\begin{align*}
[h]_1^2[h]_0 \le & C ([u]_1^2+[\zeta]_1^2) [h]_0\le C \delta [\zeta]_1^2+C[u]_1^2 [h]_0 \\
\le &C \delta [\zeta]_1^2+C[u]_1^2 ([u]_0+[\zeta]_0) \\
\le& C \delta [\zeta]_1^2 +C \alpha^3 (1+|\vec{x}|)^6+C \alpha^2 (1+|\vec{x}|)^4 [\zeta]_0,
	\end{align*}
	where we used Lemma \ref{estimateJacobifield}. From \eqref{phiintegral1xa} and H\"older's inequality, it is clear that
		\begin{align}
		\label{phiintegral11}\int_{\bar M}[h]_1^2[h]_0 \,\mathrm{d}V_s \leq C\left(\alpha^3+\alpha^2 \mathfrak{X}+\delta\mathfrak{X}^2\right).
	\end{align}
	
Similarly, we have
	\begin{align*}
[h]_0^2[h]_2 \le & C \delta [\zeta]_0^2 +C [u]_0^2[h]_2 \\
 \le & C \delta [\zeta]_0^2 +C [u]_0^2[u]_2+C [u]_0^2[\zeta]_2 \\
\le& C \delta [\zeta]_0^2 +C \alpha^3 (1+|\vec{x}|)^6+C \alpha^2 (1+|\vec{x}|)^4 [\zeta]_2
	\end{align*}
and hence
\begin{align}
		\label{phiintegral22}\int_{\bar M} [h]_0^2[h]_2 \,\mathrm{d}V_s \leq C\left(\alpha^3+\alpha^2 \mathfrak{X}+\delta\mathfrak{X}^2\right).
	\end{align}	

Moreover, we have
	\begin{align*}
(1+|\vec{x}|)[h]_0^2[h]_1 \le & C \delta (1+|\vec{x}|) [\zeta]_0^2 +C (1+|\vec{x}|) [u]_0^2[h]_1 \\
\le &C \delta (1+|\vec{x}|) [\zeta]_0^2 +C (1+|\vec{x}|) [u]_0^2[u]_1+C (1+|\vec{x}|) [u]_0^2[\zeta]_1 \\
\le & C \delta (1+|\vec{x}|) [\zeta]_0^2+C (1+|\vec{x}|)^7 \alpha^3+C (1+|\vec{x}|)^{5} \alpha^2[\zeta]_1.
	\end{align*}
Since by Lemma \ref{concentrationinequ}, 
	\begin{align*}
\int_{\bar M} (1+|\vec{x}|) [\zeta]_0^2 \,\mathrm{d}V_s \le C \int_{\bar M}\bar f [\zeta]_0^2 \,\mathrm{d}V_s \le C \int_{\bar M} [\zeta]_1^2 \,\mathrm{d}V_s,
	\end{align*}
we obtain
\begin{align}
		\label{phiintegral23}\int_{\bar M} (1+|\vec{x}|)[h]_0^2[h]_1 \,\mathrm{d}V_s \leq C\left(\alpha^3+\alpha^2 \mathfrak{X}+\delta\mathfrak{X}^2\right).
	\end{align}	

In addition, we compute
		\begin{align*}
[h]_0[h]_1[\na\chi]_0 \le & [h]^2_1 [\na\chi]_0 \le C([u]_1^2+[\zeta]_1^2) [\na\chi]_0 \\
\le & C[u]_1^2 [\na\chi]_0+C \delta [\zeta]_1^2 \\
\le & C[u]_1^2 ([\na u]_0+[\na q]_0)+C \delta [\zeta]_1^2 \\
\le & C\alpha^3 (1+|\vec{x}|)^5+C\alpha^2(1+|\vec{x}|)^4[\na q]_0+C \delta [\zeta]_1^2,
	\end{align*}
	which implies
\begin{align}
		\label{phiintegral33}\int_{\bar M} [h]_0[h]_1[\na\chi]_0 \,\mathrm{d}V_s \leq C\left(\alpha^3+\alpha^2 \mathfrak{X}+\delta\mathfrak{X}^2\right).
	\end{align}	

	 Plugging all these estimates \eqref{phiintegral11}, \eqref{phiintegral22}, \eqref{phiintegral23} and \eqref{phiintegral33} into \eqref{phiintegral1}, we get
	\begin{align*}
		\rVert\Phi''*h\rVert_{L^1}\leq C\left(\alpha^3+\alpha^2 \mathfrak{X}+\delta\mathfrak{X}^2\right).
	\end{align*}
	
	For $\rVert\Phi'(s)*h*h\rVert_{L^1} $, it follows from Lemma \ref{estideriPhi} that
	\begin{align}\label{phiintegral3}
		\rVert\Phi'(s)*h*h\rVert_{L^1}\leq C\int_{\bar M} [h]_0^2[h]_2+[h]_0^2[\na\chi]_1+(1+|\vec{x}|)[h]_0^2[h]_1 \,\mathrm{d}V_s.
	\end{align}
	
	Notice that
		\begin{align*}
[h]_0^2[\na\chi]_1 \le & C([u]_0^2+[\zeta]_0^2) [\na\chi]_1\le  C[u]_0^2 [\na\chi]_1+C \delta [\zeta]_0^2 \\
\le & C[u]_0^2 ([\na u]_1+[\na q]_1)+C \delta [\zeta]_0^2 \\
\le & C\alpha^3 (1+|\vec{x}|)^5+C\alpha^2(1+|\vec{x}|)^4[\na q]_1+C \delta [\zeta]_0^2,
	\end{align*}	
which implies	
	\begin{align*}
\int_{\bar M} [h]_0^2[\na\chi]_1 \,\mathrm{d}V_s \leq C\left(\alpha^3+\alpha^2 \mathfrak{X}+\delta\mathfrak{X}^2\right).
	\end{align*}	

Combining this with \eqref{phiintegral22}, \eqref{phiintegral23} and \eqref{phiintegral3}, we obtain
	\begin{align*}
		\rVert\Phi'(s)*h*h\rVert_{L^1}\leq C\left(\alpha^3+\alpha^2 \mathfrak{X}+\delta\mathfrak{X}^2\right).
	\end{align*}
	
Since $\Phi(0)=0$, we have $|\Phi(s)| \leq C\left([h]_2+[\na\chi]_1+(1+|\vec{x}|)[h]_1\right)$ by integration, which implies
	\begin{align*}
		\rVert \Phi* h*h\rVert_{L^1}\leq C\int_{\bar M} [h]_0^2[h]_2+[h]_0^2[\na\chi]_1+(1+|\vec{x}|)[h]_0^2[h]_1 \,\mathrm{d}V_s\leq C\left(\alpha^3+\alpha^2 \mathfrak{X}+\delta\mathfrak{X}^2\right).
	\end{align*}
\end{proof}

Next, we consider the derivatives of $w(s)=2\Delta f-|\nabla f|^2+\scal+f-n$ with respect to $s$. The following variational formula is well-known to experts.

\begin{lem} \label{lem:variaw1}
We have
\begin{align}\label{mu1stderi}
	w'(s)=-2\left(\Delta_f+\frac{1}{2}\right)Q+\la\Phi,h\ra+\Div_f\Div_f h.
\end{align}
\end{lem}

\begin{proof}
From direct calculations, we have
\begin{align*}
\scal'=&-\la \Ric, h \ra+\Div\,\Div \, h-\Delta (\Tr\, h),\\ 
\lc|\na f|^2\rc'=&2\la \na f, \na \chi \ra-h(\na f, \na f),\\
\lc \Delta f \rc'=&-\la \Div\,h, \na f \ra-\la h, \na^2 f \ra+\frac{1}{2} \la \na\Tr\, h, \na f \ra+\Delta\chi,
\end{align*}
where $\la \cdot, \cdot \ra$ denotes the inner product induced by $g(s)$. Thus, we have
\begin{align*}
w'=-2\left(\Delta_f+\frac{1}{2}\right)Q+\la\Phi,h\ra+\Div\,\Div\, h-2\la \Div\,h, \na f \ra-\la h, \na^2 f \ra+h(\na f, \na f).
\end{align*}
Since 
\begin{align*}
\Div_f \,\Div_f h=\Div\,\Div\, h-2\la \Div\,h, \na f \ra-\la h, \na^2 f \ra+h(\na f, \na f),
\end{align*}
we obtain \eqref{mu1stderi}.
\end{proof}

\begin{lem}\label{keycal2}
	The following pointwise bounds hold for $C=C(n)$ and any $s\in [0,1]$:
	\begin{enumerate}[label=\textnormal{(\roman{*})}]
		\item $|w'(s)|\leq C\left([h]_2+[\nabla \chi]_1+[Q]_0+[\Phi]_0[h]_0+(1+|\vec{x}|)([h]_1+[\na\chi]_0)+(1+|\vec{x}|^2)[h]_0\right)$,
		\item $|w''(s)|\leq C\left( [h]_2^2+[h]_1[\na\chi]_1+[\na\chi]_0^2+(1+|\vec{x}|)([h]_0[h]_1+[h]_0[\na\chi]_0)+(1+|\vec{x}|^2)[h]_0^2\right)$.
	\end{enumerate}
\end{lem}
\begin{proof}
	(i): This follows from \eqref{mu1stderi} directly.
	
	(ii): We calculate the second derivative of $w$ by
	\begin{align*}
		w''(s)=-2(\Delta_f)'Q-2\left(\Delta_f+\frac{1}{2}\right)Q'+\la\Phi',h\ra-2\Phi_{ij}h_{jl}h^{il}+\big(\Div_f \Div_f\big)'h.
	\end{align*}
	
	Note that $Q'=-|h|_{g(s)}^2/2$ and by Lemma \ref{estideriPhi}, $|\Phi'(s)|\leq C\left([h]_2+[\na\chi]_1+(1+|\vec{x}|)[h]_1\right)$. Thus, we have
	\begin{align}\label{muesti1}
		\left|-2\left(\Delta_f+\frac{1}{2}\right)Q'+\la\Phi',h\ra-2\Phi_{ij}h_{jl}h^{il}\right|\leq C\left( [h]_2^2+[h]_1[\na\chi]_1+(1+|\vec{x}|)[h]_0[h]_1\right).
	\end{align}
	
	Under local coordinates, we have
	\begin{align*}
		(\Div_f h)_i=g^{jk}\lc \partial_kh_{ij}-\Gamma_{ki}^l h_{lj}-\Gamma_{kj}^l h_{il} \rc-g^{jk}h_{ij}\partial_k f,
	\end{align*}
	then it follows that
	\begin{align*}
		(\Div_f h)'_i=-h^{jk}\lc \partial_kh_{ij}-\Gamma_{ki}^l h_{lj}-\Gamma_{kj}^l h_{il} \rc-g^{jk}(C_{ki}^lh_{lj}+C_{kj}^lh_{il})+h^{jk}h_{ij}\partial_k f-g^{jk}h_{ij} \partial_k \chi,
	\end{align*}
	where $C_{ij}^k=(\Gamma_{ij}^k)_s$, as in the proof of Lemma \ref{estideriPhi}. By the expression \eqref{Cijkestimate}, this implies
	\begin{align*}
		\left|\Div_f(\Div_fh)'\right|\leq C \lc [h]_2^2+[h]_1[\na \chi]_1+(1+|\vec{x}|)([h]_0[h]_1+[h]_0[\na \chi]_0)+(1+|\vec{x}|^2)[h]_0^2 \rc.
	\end{align*}
	Therefore,
	\begin{align}\label{deldelh}
		\left|(\Div_f\Div_f h)'\right|&=\left|\Div_f(\Div_fh)'+(\Div_f)'(\Div_fh)\right|\nonumber\\
		&\leq C\left( [h]_2^2+[h]_1[\na\chi]_1+(1+|\vec{x}|) ([h]_0[h]_1+[h]_0[\na\chi]_0)+(1+|\vec{x}|^2)[h]_0^2\right).
	\end{align}
	
	Similarly, we have
	\begin{align*}
		(\Delta _f)'Q=-h^{ij}\partial_i\partial_j Q+h^{ij}\Gamma^k_{ij}\partial_k Q-g^{ij}C^k_{ij}\partial_kQ+h^{ij}\partial_if\partial_jQ-g^{ij}\partial_i \chi \partial_j Q
	\end{align*}
	and hence
	\begin{align}\label{dellaplacian}
		\left|(\Delta_f)'Q \right|\leq C\left( [h]_1[\na Q]_1+[\na\chi]_0[\na Q]_0+(1+|\vec{x}|)[h]_0[\na Q]_0\right).
	\end{align}
	
	Combining \eqref{muesti1}, \eqref{deldelh}, \eqref{dellaplacian} and noting that $[\na Q]_0\leq [h]_1+[\na\chi]_0$ and $[\na Q]_1\leq [h]_2+[\na\chi]_1$, we obtain 
	\begin{align*}
		\left|w''(s)\right|\leq C\left( [h]_2^2+[h]_1[\na\chi]_1+[\na\chi]_0^2+(1+|\vec{x}|)([h]_0[h]_1+[h]_0[\na\chi]_0)+(1+|\vec{x}|^2)[h]_0^2\right).
	\end{align*}
	This completes the proof of (ii).
\end{proof}

\begin{lem}\label{keycal3}
	There exists a constant $C=C(n)$ such that
	\begin{equation*}
		\int_{\bar M} Q^2(s)\,\mathrm{d}V_s\leq C\left(\alpha^4+\delta^2\mathfrak{X}^2+\alpha^2 \mathfrak{X}^2\right)+2\int_{\bar M} Q^2(0)\,\mathrm{d}V_0.
	\end{equation*}
\end{lem}
\begin{proof}
	Set $\eta(s):= \int_{\bar M} Q^2(s) \,\mathrm{d}V_s$. Since $|Q|\leq C\delta$ and $Q'=-|h|^2/2$, we calculate
	\begin{align}\label{Qineq1}
		\eta'(s)=\int_{\bar M} 2Q(s)Q'(s) \,\mathrm{d}V_s+\int_{\bar M} Q^3(s) \,\mathrm{d}V_s\leq -\int_{\bar M} Q(s)|h|^2 \,\mathrm{d}V_s+C\delta \eta(s).
	\end{align}
	By the Cauchy--Schwarz inequality, we have
	\begin{align}\label{Qineq2x}
		\int_{\bar M} |Q(s)||h|^2 \,\mathrm{d}V_s\leq \left(\int_{\bar M}|h|^4 \,\mathrm{d}V_s\right)^{1/2}\eta^{1/2}(s).
	\end{align}
	Note that
		\begin{align*}
|h|^4 \le C(u^4+|\zeta|^4) \le Cu^4+C|\zeta|^2(|h|^2+u^2) \le Cu^4+C \delta^2 |\zeta|^2+C \alpha^2 (1+|\vec{x}|)^4 |\zeta|^2,
	\end{align*}
	which, by using Lemma \ref{concentrationinequ} and Corollary \ref{cor:tay3}, implies
			\begin{align*}
\int_{\bar M} |h|^4 \,\mathrm{d}V_s \le C \lc \alpha^4+\delta^2 \mathfrak{X}^2+ \alpha^2 \mathfrak{X}^2 \rc.
	\end{align*}
	Thus, it follows from \eqref{Qineq2x} that
		\begin{align}\label{Qineq2}
		\int_{\bar M} |Q(s)||h|^2 \,\mathrm{d}V_s\leq C\left(\alpha^2+\delta\mathfrak{X}+\alpha \mathfrak{X} \right)\eta^{1/2}(s).
	\end{align}

	Combining \eqref{Qineq1} and \eqref{Qineq2}, we obtain
	\begin{align*}
		\eta'(s)\leq C\left(\alpha^2+\delta\mathfrak{X}+\alpha \mathfrak{X}\right)\eta^{1/2}(s)+C\delta\eta.
	\end{align*}
	If we set $\psi(s)=(e^{-C\delta s}\eta)^{1/2}$, then
	$$\psi '\leq \frac{C}{2}e^{-\frac{C\delta s}{2}}\left(\alpha^2+\delta\mathfrak{X}+\alpha \mathfrak{X}\right).$$
	Therefore, by integration, we obtain
	$$\psi(s)\leq C\left(\alpha^2+\delta\mathfrak{X}+\alpha \mathfrak{X}\right)+\psi(0)=C\left(\alpha^2+\delta\mathfrak{X}+\alpha \mathfrak{X}\right)+\eta^{1/2}(0).$$
	
Since $\delta$ is sufficiently small, we have
	$$\eta(s)\leq C\left(\alpha^4+\delta^2\mathfrak{X}^2+\alpha^2 \mathfrak{X}^2 \right)+2\eta(0).$$
	This finishes the proof.
\end{proof}

\begin{prop}\label{keycal4}
	There exists a constant $C=C(n)$ such that
	\begin{equation*}
		\int_{\bar M} Q^2(s)\,\mathrm{d}V_s\leq C\left(\mathfrak{X}^2+\left|V(1)-V(0)\right|^2\right).
	\end{equation*}
\end{prop}

\begin{proof}
For the function $V(s)=\int_{\bar M} \,\mathrm{d}V_s$, a direct calculation shows that
	$$V'=\int_{\bar M} Q\,\mathrm{d}V_s,\quad V''=\int_{\bar M} Q'+Q^2\,\mathrm{d}V_s,\quad V'''=\int_{\bar M}Q''+3Q'Q+Q^3\,\mathrm{d}V_s.$$
	Thus,
	$$V'(0)= \int_{\bar M} Q(0)\,\mathrm{d}V_0, \quad V''(0)=-\frac{1}{2}\int_{\bar M} |h|^2\,\mathrm{d}V_0+\int_{\bar M} Q^2(0)\,\mathrm{d}V_0.$$
	
	Using the fact that $Q'=-|h|^2/2$, we have $Q''=h*h*h$, and hence for any $s\in [0,1]$,
	\begin{align} \label{eq:excom000}
		\abs{V'''}\leq C\int_{\bar M} |h|^3+|h|^2|Q|+\delta Q^2\,\mathrm{d}V_s.
	\end{align}

Notice that 
	\begin{align}\label{exeq:comput1}
|h|^3 \le C(|u|^3+|\zeta|^3) \le C|u|^3+C|\zeta|^2(|h|+|u|) \le C|u|^3+C \delta |\zeta|^2+C \alpha (1+|\vec{x}|)^2 |\zeta|^2,
	\end{align}
	which, by using Lemma \ref{concentrationinequ} and Corollary \ref{cor:tay3}, implies
			\begin{align} \label{eq:excom001}
\int_{\bar M} |h|^3 \,\mathrm{d}V_s \le C \lc \alpha^3+\delta \mathfrak{X}^2+ \alpha \mathfrak{X}^2 \rc.
	\end{align}
	
In addition, it follows from \eqref{Qineq2} and Lemma \ref{keycal3} that
				\begin{align} \label{eq:excom002}
		\int_{\bar M} |h|^2|Q(s)| \,\mathrm{d}V_s\leq C\left(\alpha^4+\delta^2\mathfrak{X}^2+\alpha^2 \mathfrak{X}^2\right)+C\int_{\bar M} Q^2(0)\,\mathrm{d}V_0.
	\end{align}
	
Combining \eqref{eq:excom000}, \eqref{eq:excom001}, \eqref{eq:excom002} and Lemma \ref{keycal3}, we obtain
	\begin{align} \label{eq:excom003}
		\abs{V'''}\leq C \lc \alpha^3+\delta \mathfrak{X}^2+ \alpha \mathfrak{X}^2 \rc+C\int_{\bar M} Q^2(0)\,\mathrm{d}V_0.
	\end{align}

	By the Taylor expansion of $V$, we have
	\begin{equation*}
		\left|V(1)-V(0)-V'(0)-\frac{1}{2}V''(0)\right|\leq \frac{1}{6}\sup_{s\in [0,1]}|V'''(s)|.
	\end{equation*}
By using \eqref{eq:excom003}, this implies
	\begin{align}\label{taylorofV2}
		&\left|\int_{\bar M} Q(0)\,\mathrm{d}V_0-\frac{1}{4}\int_{\bar M} |h|^2\,\mathrm{d}V_0+\frac{1}{2}\int_{\bar M} Q^2(0)\,\mathrm{d}V_s\right| \notag \\
		\leq& C \lc \alpha^3+\delta \mathfrak{X}^2+ \alpha \mathfrak{X}^2 \rc+C\int_{\bar M} Q^2(0)\,\mathrm{d}V_0+\left|V(1)-V(0)\right|.
	\end{align}
	
	Since $\int_{\bar M}|h|^2\,\mathrm{d}V_0\leq C(\alpha^2+\mathfrak{X}^2)$, it follows that
	\begin{align}\label{equ:taylorV1}
		\left|\int_{\bar M} Q(0) \,\mathrm{d}V_0\right|\leq C\left(\alpha^2+\mathfrak{X}^2\right)+C\int_{\bar M} Q^2(0)\,\mathrm{d}V_0+\left|V(1)-V(0)\right|.
	\end{align}
	
	On the other hand, by the Poincar\'e inequality on the cylinder, we have
	\begin{align}\label{equ:taylorV2}
		0\leq\int_{\bar M} Q^2(0)\,\mathrm{d}V_0-\left(\int Q(0) \,\mathrm{d}V_0\right)^2\leq C\int_{\bar M} \left|\nabla Q(0)\right|^2 \,\mathrm{d}V_0\leq C\mathfrak{X}^2,
	\end{align}
where the last inequality follows from 
	\begin{align}\label{exeq:comput2}
\int_{\bar M} |\nabla Q(0)|^2 \,\mathrm{d}V_0\leq C \lc \rVert \zeta\rVert_{W^{1,2}}^2+\rVert\na q\rVert_{L^2}^2 \rc \leq C\mathfrak{X}^2.
	\end{align}

	Combining \eqref{equ:taylorV1} with \eqref{equ:taylorV2}, we get
	\begin{align*}
		\left|\int_{\bar M} Q(0) \,\mathrm{d}V_0\right|\leq C\left(\alpha^2+\mathfrak{X}^2\right)+C\left(\int_{\bar M} Q(0)\mathrm{d}V_0\right)^2+\left|V(1)-V(0)\right|.
	\end{align*}
	Since $\left|\int_{\bar M} Q(0) \,\mathrm{d}V_0\right|\leq C\delta$, we conclude that if $\delta\ll 1$,
	\begin{equation*}
		\left|\int_{\bar M} Q(0) \,\mathrm{d}V_0\right|\leq C\left(\alpha^2+\mathfrak{X}^2\right)+\left|V(1)-V(0)\right|.
	\end{equation*}
	Therefore, by Corollary \ref{cor:tay3} and \eqref{equ:taylorV2}, we have
	\begin{align}\label{Qesti}
		\int_{\bar M} Q^2(0) \,\mathrm{d}V_0\leq C\mathfrak{X}^2+\left(\alpha^2+\mathfrak{X}^2+\left|V(1)-V(0)\right|\right)^2\leq C\mathfrak{X}^2+C\left|V(1)-V(0)\right|^2.
	\end{align}
	Then Proposition \ref{keycal4} follows from the combination of Lemma \ref{keycal3} and \eqref{Qesti}.
\end{proof}

\begin{lem}\label{keycalxetra2}
There exists a constant $C=C(n)$ such that for any $s\in [0,1]$,
	\begin{equation*}
		\int_{\bar M} (1+|\vec{x}|)^2 Q^2(s)\,\mathrm{d}V_s\leq C \lc  \delta \alpha^3+\mathfrak X^2 +|V(1)-V(0)|^2 \rc.
	\end{equation*}
\end{lem}

\begin{proof}
We set $\eta(s)=\int_{\bar M} (1+|\vec{x}|)^2 Q^2(s) \,\mathrm{d}V_s$. By Lemma \ref{concentrationinequ}, we obtain
	\begin{align}\label{equ:naQestix2}
\eta(0) =\int_{\bar M} (1+|\vec{x}|)^2 Q(0)^2 \,\mathrm{d}V_0 \le C \lc \|Q(0)\|_{L^2}^2+\|\na Q(0)\|^2_{L^2} \rc  \le C \lc \mathfrak{X}^2+|V(1)-V(0)|^2 \rc,
	\end{align}
where we used Proposition \ref{keycal4} and \eqref{exeq:comput2}. By a direct calculation, as in the proof of Lemma \ref{keycal3}, we have
	\begin{align}\label{equ:naQestix3}
\eta'(s) \le \left(\int_{\bar M} (1+|\vec{x}|)^2 |h|^4 \,\mathrm{d}V_s\right)^{1/2}\eta^{1/2}(s)+C \delta \eta(s).
	\end{align}
From \eqref{exeq:comput1}, we have
	\begin{align*}
\int_{\bar M} (1+|\vec{x}|)^2 |h|^4 \,\mathrm{d}V_s \le C \delta \int_{\bar M} (1+|\vec{x}|)^2 \lc  |u|^3+ \delta |\zeta|^2+ \alpha (1+|\vec{x}|)^2 |\zeta|^2 \rc \,\mathrm{d}V_s,
	\end{align*}	
	which, by using Lemma \ref{concentrationinequ} and Corollary \ref{cor:tay3}, implies
			\begin{align} \label{equ:naQestix4}
\int_{\bar M} (1+|\vec{x}|)^2 |h|^4 \,\mathrm{d}V_s \le C \delta \lc \alpha^3+\delta \mathfrak X^2+\alpha  \mathfrak X^2  \rc.
	\end{align}	
	
Combining \eqref{equ:naQestix2}, \eqref{equ:naQestix3}, and \eqref{equ:naQestix4}, we deduce, by the same argument as in the proof of Lemma \ref{keycal3}, that
	\begin{align*}
\eta(s) \le C \lc  \delta \alpha^3+\mathfrak X^2 +|V(1)-V(0)|^2 \rc.
	\end{align*}	
\end{proof}

\begin{lem}\label{keycal6}
There exists a constant $C=C(n)$ such that for any $s\in [0,1]$,
	\begin{align*}
\rVert \Phi' * h* Q\rVert_{L^1}+\rVert \Phi * h * Q\rVert_{L^1} \le C \lc \alpha^2 \mathfrak{X}+\alpha^2\left|V(1)-V(0)\right|+\delta\mathfrak{X}^2 \rc.
	\end{align*}
\end{lem}
\begin{proof}
	By Lemma \ref{estideriPhi}, we have $|\Phi(s)|+|\Phi'(s)|\leq C\left([h]_2+[\na\chi]_1+(1+|\vec{x}|)[h]_1\right)$. Thus,
	\begin{align*}
		\rVert \Phi * h * Q\rVert_{L^1}+\rVert \Phi'  * h * Q\rVert_{L^1}\leq C\int_{\bar M}\left( [h]_2^2+[h]_0[\na\chi]_1+(1+|\vec{x}|)[h]_1^2\right)|Q|\,\mathrm{d}V_s.
	\end{align*}
	
	Now we estimate these three terms separately. For the first term, by Lemma \ref{estimateJacobifield} and Proposition \ref{keycal4}, 
	\begin{align}\label{inteesti1}
		\int_{\bar M}[h]_2^2|Q| \,\mathrm{d}V_s&\leq C\alpha^2\int_{\bar M} \left(1+|\vec{x}|^4\right)|Q| \,\mathrm{d}V_s+C\int_{\bar M}[\zeta]_2^2|Q| \,\mathrm{d}V_s\nonumber\\
		&\leq C\alpha^2\rVert Q\rVert_{L^2}+C \delta \|\zeta\|^2_{W^{2, 2}}\nonumber\\
		&\leq C \lc \alpha^2 \mathfrak{X}+\alpha^2\left|V(1)-V(0)\right|+\delta\mathfrak{X}^2 \rc.
	\end{align}
	Similarly, for the second term, it holds that
	\begin{align}\label{inteesti2}
		\int_{\bar M}[h]_0[\na\chi]_1|Q|\,\mathrm{d}V_s\leq& C\alpha^2\int_{\bar M} \left(1+|\vec{x}|^4\right)|Q| \,\mathrm{d}V_s+C\int_{\bar M}\left([\zeta]_0^2+[\na q]_1^2\right)|Q| \,\mathrm{d}V_s\nonumber\\
\le & C \lc \alpha^2 \mathfrak{X}+\alpha^2\left|V(1)-V(0)\right|+\delta\mathfrak{X}^2 \rc.
	\end{align}
	For the third term, by the Cauchy--Schwarz inequality and Lemma \ref{concentrationinequ}, we have
	\begin{align}\label{inteesti3}
		\int_{\bar M} \left(1+|\vec{x}|\right)[h]_1^2|Q| \,\mathrm{d}V_s&\leq C\alpha^2\int_{\bar M}(1+|\vec{x}|)^5|Q|\,\mathrm{d}V_s+C\int_{\bar M}\left(1+|\vec{x}|\right)[\zeta]_1^2|Q| \,\mathrm{d}V_s\nonumber\\
		&\leq C\alpha^2\rVert Q\rVert_{L^2}+C\delta \int_{\bar M}\left(1+|\vec{x}|\right)[\zeta]_1^2 \,\mathrm{d}V_s\nonumber\\
		&\leq C\alpha^2\rVert Q\rVert_{L^2}+C\delta \|\zeta\|_{W^{2, 2}}^2\nonumber\\
&\le  C \lc \alpha^2 \mathfrak{X}+\alpha^2\left|V(1)-V(0)\right|+\delta\mathfrak{X}^2 \rc.
	\end{align}
	Combining \eqref{inteesti1}, \eqref{inteesti2} and \eqref{inteesti3}, we complete the proof.
\end{proof}

\begin{prop}\label{3rdderiative}
	There exists a constant $C=C(n)$ such that
	\begin{equation*}
		\left|W'''(s)\right|\leq C\left(\alpha^3+\alpha^2\mathfrak{X}+\delta\mathfrak{X}^2+\alpha^2\left|V(1)-V(0)\right|+\delta\left|V(1)-V(0)\right|^2\right).
	\end{equation*}
\end{prop}
\begin{proof}
	Note that $(\mathrm{d}V_s)'=Q\,\mathrm{d}V_s$, we thus have $W'(s)=\int w'+w Q\,\mathrm{d}V_s$ and by Lemma \ref{lem:variaw1}, we can rewrite it as
	\begin{align} \label{eq:first1}
		W'(s)=\int_{\bar M}\la \Phi(s),h\ra \,\mathrm{d}V_s+\int_{\bar M} Q(w-1)\,\mathrm{d}V_s.
	\end{align}
	Then we have
	\begin{align}
		W''(s)=&\int_{\bar M}\la \Phi',h\ra \,\mathrm{d}V_s+\int_{\bar M} -2\Phi_{ij}h_{jl}h^{il} \,\mathrm{d}V_s+\int_{\bar M} \la \Phi,h\ra Q \,\mathrm{d}V_s\quad (=:\mathrm{I}+\mathrm{II}+\mathrm{III}) \notag \\
		&+\int_{\bar M} Q'(w-1) \,\mathrm{d}V_s+\int_{\bar M} Q w' \,\mathrm{d}V_s+\int_{\bar M} Q^2(w-1)\,\mathrm{d}V_s\quad (=:\mathrm{IV}+\mathrm{V}+\mathrm{VI}) \label{eq:second2}.
	\end{align}
	
	We now estimate the third derivative term by term using Corollary \ref{keycal1}, Lemma \ref{keycal2}, and Lemma \ref{keycal6}.
	First, we have
	\begin{align*}
		\abs{\mathrm{I}'}&=\abs{\int_{\bar M} \la \Phi'',h\ra+\Phi'*h*h+\la\Phi',h\ra Q  \,\mathrm{d}V_s}\\
		&\leq \rVert\Phi''*h\rVert_{L^1}+ \rVert\Phi'*h*h\rVert_{L^1}+\rVert \la\Phi', h\ra Q\rVert_{L^1}\\
		&\leq C\left(\alpha^3+\alpha^2\mathfrak{X}+\delta\mathfrak{X}^2+\alpha^2\left|V(1)-V(0)\right|\right),
	\end{align*}
	where in the last inequality, we have used Lemma \ref{keycal1} and Lemma \ref{keycal6}. Similarly,
	\begin{align*}
		\abs{\mathrm{II}'}=&\abs{\int_{\bar M} \Phi'*h*h+\Phi*h*h*h+\Phi'*h*hQ \,\mathrm{d}V_s}\\
		\leq & C\left(\alpha^3+\alpha^2\mathfrak{X}+\delta\mathfrak{X}^2+\alpha^2\left|V(1)-V(0)\right|\right),
	\end{align*}
	where in the last inequality, we have used Lemma \ref{keycal1} and Lemma \ref{keycal6}.
	
In the same way, we have
	\begin{align*}
		\abs{\mathrm{III}'}&=\abs{\int_{\bar M}\la\Phi',h\ra Q+\Phi*h*h Q-\frac{1}{2}\la \Phi,h\ra |h|^2+\la \Phi,h\ra Q^2 \,\mathrm{d}V_s}\\
		&\leq C\left(\alpha^3+\alpha^2\mathfrak{X}+\delta\mathfrak{X}^2+\alpha^2\left|V(1)-V(0)\right|\right).
	\end{align*}
	
	Note that $Q'=-|h|^2/2$ and $Q''=h*h*h$. It follows from Lemma \ref{keycal2} that
	\begin{align*}
		\abs{\mathrm{IV}'}&=\abs{\int_{\bar M}  Q''(w-1)+Q'w'+Q'Q(w-1) \,\mathrm{d}V_s}\\
		&\leq C\int_{\bar M} [h]_0^2([h]_2+[\nabla \chi]_1+[Q]_0+[\Phi]_0[h]_0+(1+|\vec{x}|)([h]_1+[\na\chi]_0)+(1+|\vec{x}|^2)[h]_0) \,\mathrm{d}V_s.
	\end{align*}
	By similar argument as in the proof of Lemma \ref{keycal1} and Lemma \ref{keycal6}, we have
	\begin{align*}
		\int_{\bar M} [h]_0^2\left([h]_2+[\na\chi]_1+[Q]_0+[\Phi]_0[h]_0\right) \,\mathrm{d}V_s\leq C\left(\alpha^3+\alpha^2\mathfrak{X}+\delta\mathfrak{X}^2+\alpha^2\left|V(1)-V(0)\right|\right),
	\end{align*}
	and 
	\begin{align*}
		\int_{\bar M} (1+|\vec{x}|)[h]_0^2([h]_1+[\na\chi]_0) +(1+|\vec{x}|^2)[h]_0^3\,\mathrm{d}V_s\leq C\left(\alpha^3+\alpha^2\mathfrak{X}+\delta\mathfrak{X}^2\right).
	\end{align*}
	Thus, it follows that
	\begin{align*}
		\abs{\mathrm{IV}'}\leq C\left(\alpha^3+\alpha^2\mathfrak{X}+\delta\mathfrak{X}^2+\alpha^2\left|V(1)-V(0)\right|\right).
	\end{align*}
	
	For the fifth term, by Lemma \ref{keycal2} and a similar argument as above, we obtain
	\begin{align*}
		\abs{\mathrm{V}'}=&\abs{\int_{\bar M} Q'w'+Qw''+Q^2 w' \,\mathrm{d}V_s}\\
		\leq& C\int_{\bar M}[h]_0^2\left([h]_2+[\nabla \chi]_1+[Q]_0+[\Phi]_0[h]_0+(1+|\vec{x}|)([h]_1+[\na\chi]_0)+(1+|\vec{x}|^2)[h]_0\right)\\
		&+[Q]_0\left( [h]_2^2+[h]_1[\na\chi]_1+[\na\chi]_0^2+(1+|\vec{x}|)([h]_0[h]_1+[h]_0[\na\chi]_0)+(1+|\vec{x}|^2)[h]_0^2\right)+ Q^2 |w'|\,\mathrm{d}V_s.
	\end{align*}
Using the same argument as in the proof of Lemma \ref{keycal1} and Lemma \ref{keycal6}, we obtain
	\begin{align*}
		\abs{\mathrm{V}'} \le C\left(\alpha^3+\alpha^2\mathfrak{X}+\delta\mathfrak{X}^2+\alpha^2\left|V(1)-V(0)\right|\right)+C\int_{\bar M} Q^2 |w'|\,\mathrm{d}V_s.
	\end{align*}
	
The last integral can be estimated by Lemmas \ref{keycal2} and \eqref{keycalxetra2} as
	\begin{align*}
\int_{\bar M} Q^2 |w'|\,\mathrm{d}V_s \le C \delta \int_{\bar M} (1+|\vec{x}|)^2 Q^2 \,\mathrm{d}V_s \le C \delta \lc  \delta \alpha^3+\mathfrak X^2 +|V(1)-V(0)|^2 \rc.
	\end{align*}
Thus, we have
 	\begin{align*}
		\abs{\mathrm{V}'} \le& C\left(\alpha^3+\alpha^2\mathfrak{X}+\delta\mathfrak{X}^2+\alpha^2\left|V(1)-V(0)\right|+\delta\left|V(1)-V(0)\right|^2\right).
	\end{align*}	
	
Finally, the last term can be similarly estimated as
	\begin{align*}
		\abs{\mathrm{VI}'}&=\abs{\int_{\bar M} 2QQ'(w-1)+Q^2 w'+Q^3(w-1) \,\mathrm{d}V_s}\\
		&\leq C\int_{\bar M} |Q||h|^2+|w'| Q^2+\delta Q^2 \,\mathrm{d}V_s\\
		&\leq C\left(\delta^2 \alpha^3+\alpha^2\mathfrak{X}+\delta\mathfrak{X}^2+\alpha^2\left|V(1)-V(0)\right|+\delta\left|V(1)-V(0)\right|^2\right).
	\end{align*}
	This finishes the proof of Proposition \ref{3rdderiative}.
\end{proof}

Applying the Taylor expansion to $W$, we have
\begin{equation}\label{taylorofF}
	\left|W(1)-W(0)-W'(0)-\frac{1}{2}W''(0)\right|\leq \frac{1}{6}\sup_{s\in[0,1]}|W'''(s)|.
\end{equation}

It is clear from \eqref{eq:first1}, \eqref{eq:second2} and Lemma \ref{lem:first1} for $s=0$ that
	\begin{align} \label{eq:first11}
W'(0)=& (W(0)-1) \int_{\bar M} Q(0) \,\mathrm{d}V_0
	\end{align}
and
\begin{align}
	W''(0)=&\frac{1}{2}\int_{\bar M} \la\LL h+2\Div_{\bar f}^*\,\Div_{\bar f} h+2\nabla^2 Q(0),h\ra \,\mathrm{d}V_0+\int_{\bar M} 2|\nabla Q(0)|^2-Q^2(0)+Q(0)\Div_{\bar f}\,\Div_{\bar f} h \,\mathrm{d}V_0\nonumber\\
	&+(W(0)-1)\int_{\bar M}(Q(0)^2-\frac{1}{2}|h|^2) \,\mathrm{d}V_0. \label{eq:second22}
\end{align}

Next, we prove

\begin{lem}\label{keycal5}
	We can find a constant $C=C(n)$ such that
	\begin{align*}
		\left|W'(0)+\frac{1}{2}W''(0)\right|\leq C\left(\alpha^3+\mathfrak{X}^2+\left|V(1)-V(0)\right|\right).
	\end{align*}
\end{lem}
\begin{proof}
	By the Taylor expansion of $V$ in \eqref{taylorofV2}  (note that here we can improve the estimate in \eqref{taylorofV2} by \eqref{Qesti}), we have
	\begin{align*}
		\left|\int_{\bar M} Q(0)\,\mathrm{d}V_0-\frac{1}{4}\int_{\bar M} |h|^2 \,\mathrm{d}V_0+\frac{1}{2}\int_{\bar M} Q^2(0)\,\mathrm{d}V_0\right|\leq C\left(\alpha^3+\mathfrak{X}^2+\left|V(1)-V(0)\right|\right).
	\end{align*}
	Thus, it follows from \eqref{eq:first11} and \eqref{eq:second22} that
	\begin{align*}
		\left|W'(0)+\frac{1}{2}W''(0)\right|\leq& \left|\frac{1}{2}\int_{\bar M} \la\LL h+2\Div_{\bar f}^*\,\Div_{\bar f} h+2\nabla^2 Q(0),h\ra \,\mathrm{d}V_0\right|\\
		&+\left|\int_{\bar M} \left(2|\nabla Q(0)|^2-Q^2(0)+Q(0)\Div_{\bar f}\,\Div_{\bar f} h\right) \,\mathrm{d}V_0\right|\\
		&+C\left(\alpha^3+\mathfrak{X}^2+\left|V(1)-V(0)\right|\right).
	\end{align*}
	
By Proposition \ref{keycal4} and \eqref{exeq:comput2}, we have
	\begin{align*}
\int_{\bar M} | Q(0)|^2 \,\mathrm{d}V_0\leq C\left(\mathfrak{X}^2+\left|V(1)-V(0)\right|^2\right) \quad \text{and} \quad \int_{\bar M} |\nabla Q(0)|^2 \,\mathrm{d}V_0\leq C\mathfrak{X}^2.
	\end{align*}

Using the integration by parts, which is justified since $h$ is compactly supported, we obtain
	\begin{align*}
		\left|\int_{\bar M} Q(0)\Div_{\bar f}\,\Div_{\bar f} h \,\mathrm{d}V_0\right|=\left|\int_{\bar M}\la \na Q(0),\Div_{\bar f} h\ra \,\mathrm{d}V_0\right|\leq \rVert \na Q(0)\rVert_{L^2}\rVert\Div_{\bar f} h\rVert_{L^2}\leq C\mathfrak{X}^2,
	\end{align*}
where we used $\rVert \Div_{\bar f}h\rVert_{L^2}\leq \mathfrak X$ from the definition of $\mathfrak X$.

Since $\LL(ug_{S^m})=0$, we have
	\begin{align*}
		\left|\int_{\bar M} \la \LL h,h\ra \,\mathrm{d}V_0\right|=\left|\int_{\bar M} \la \LL \zeta, \zeta\ra \,\mathrm{d}V_0\right|\leq C\|\zeta\|_{W^{1,2}}^2 \le C\mathfrak{X}^2.
	\end{align*}
Moreover, using integration by parts, we obtain
	\begin{align*}
		&\left|\int_{\bar M} \la\Div_{\bar f}^*\,\Div_{\bar f}h,h\ra \,\mathrm{d}V_0\right|=\left|\int_{\bar M} \la \Div_{\bar f}h,\Div_{\bar f}h\ra \,\mathrm{d}V_0\right|\leq C\mathfrak{X}^2,\\
		&\left|\int_{\bar M} \la \nabla^2 Q(0),h\ra \,\mathrm{d}V_0\right|=\left|\int_{\bar M} \la \nabla Q(0),\Div_{\bar f}h\ra \,\mathrm{d}V_0\right|\leq C\mathfrak{X}^2.
	\end{align*}
	Since $\left|V(1)-V(0)\right| \ll 1$, we can combine all the above inequalities to conclude that
	$$\left|W'(0)+\frac{1}{2}W''(0)\right|\leq C\left(\alpha^3+\mathfrak{X}^2+\left|V(1)-V(0)\right|\right).$$
\end{proof}

The main result of this section is the following estimate for $\WW(g, f)$.

\begin{thm}\label{lojaforF}
For the weighted Riemannian manifold $(\bar M, g, f)$ such that $(h:=g-\bar g, \chi:=f-\bar f) \in C^2(S^2(\bar M)) \times C^2(\bar M)$ is compactly supported and
	\begin{align*}
\|h\|_{C^2}+\|\chi\|_{C^2} \le \delta_n,
	\end{align*}
for a small constant $\delta_n$ depending only $n$, then we have
	\begin{align*}
		\left|\WW(g,f)-\Theta_m \right|\leq C\left(\rVert u\rVert_{L^2}^3+\rVert \mathbf{\Phi}(g,f)\rVert_{L^2}^2+\rVert\Div_{\bar f} h\rVert_{W^{1,2}}^2+|\BB(h,\chi)|^2+\left|\VV(g,f)-1\right|\right),
	\end{align*}
where $ug_{S^m}$ is the projection of $h$ onto $\KK_0 g_{S^m}$.
\end{thm}
\begin{proof}
	Note that $\WW(g,f)=W(1)$, $\WW(\bar g,\bar f)=W(0)=\Theta_m$\index{$\Theta_m$}, $\VV(g,f)=V(1)$ and $\VV(\bar g,\bar f)=1$, then the conclusion follows from the combination of Proposition \ref{3rdderiative}, Lemma \ref{keycal5} and \eqref{taylorofF}, noting that $\alpha^2 \mathfrak{X} \le 2\alpha^3/3+\mathfrak{X}^3/3$.
\end{proof}

\section{Radius functions on weighted Riemannian manifolds}\label{secradiicontraction}

In this section, we introduce several notions of radius functions that serve to characterize regions that are nearly cylindrical on a weighted Riemannian manifold. Our main result, Theorem \ref{lojawithradius}, provides an estimate for the deviation of the pointed $\mathcal{W}$-entropy from that of the standard cylinder, expressed in terms of a suitable curvature radius.

\subsection*{Definitions of radius functions}

Throughout this subsection, the weighted Riemannian manifold $\left(M,g,f\right)$ is always assumed to be normalized (see \eqref{normalziation1}) and satisfies, for some constant $C_V>0$,
\begin{align}\label{normalization22}
	\int_{M\setminus B(p,L)} 1 \,\mathrm{d}V_f \leq C_V e^{-\frac{L^2}{15}}, \quad\forall L>0,
\end{align}
where $p$ is a fixed minimum point of $f$. As in Section \ref{seclojaW}, we fix a weighted cylinder $\left(\bar M,\bar g,\bar f\right)$ as in \eqref{eq:model}, set $\bar b=2\sqrt{\bar f}$\index{$\bar b$}, and fix a base point $\bar p \in \bar M$, which is a minimum point of $\bar f$.

\begin{defn}\label{defnradii}
	For $\sigma\in (0,1/10)$ and a weighted Riemannian manifold $(M,g,f)$ with $\Phi=\mathbf{\Phi}(g, f)=\dfrac{g}{2}-\Ric(g)-\nabla^2 f$, we define
	\begin{enumerate}[label=\textnormal{(\Alph{*})}]
		\item \emph{($\rA$-radius)}\index{$\rA$} $\rA$ as the largest number $L$ such that there exists a diffeomorphism $\varphi_A$ from $\{\bar b \le L\} \subset \bar M$ onto a subset of $M$ such that
		\begin{align*}
			\left[\bar g-\varphi_A^* g\right]_5+\left[\bar f-\varphi_A^* f\right]_5 \leq e^{\frac{\bar f}{4}-\frac{L^2}{16}},
		\end{align*}
		\item \emph{($\rBC$-radius)}\index{$\rBC$} $\rBC$ as the largest number $L$ such that there exists a diffeomorphism $\varphi_B$ from $\{\bar b \le L\} \subset \bar M$ onto a subset of $M$ such that
		\begin{align*}
			\left[\bar g-\varphi_B^* g\right]_0+\left[\bar f-\varphi_B^* f\right]_0\leq e^{-\frac{L^2}{33}},
		\end{align*}
		and
		\begin{equation}\label{equ:integralphi}
			\int_{\{\bar b\le L\}}\left|\varphi_B^*\Phi\right|^2 \,\mathrm{d}V_{\bar f}\leq e^{-\frac{L^2}{4-\sigma}}.
		\end{equation}
Moreover, for all $k \in [1, 10^{10} n\sigma^{-1}]$, the $C^k$-norms of $\bar g-\varphi_B^* g$ and $\bar f-\varphi_B^* f$ are bounded by $1$. 
	\end{enumerate}
Here and for the remainder of the paper, all norms $[\cdot]_l$ are taken with respect to $\bar g$, unless explicitly stated otherwise.
\end{defn}

In the following, we will frequently use the interpolation inequalities from \cite[Section 3.2]{kryelliptic} and \cite[Lemma B.1]{colding2015uniqueness}. Although originally stated for Euclidean space, these inequalities remain valid on Riemannian manifolds with bounded geometry. In particular, they apply to $(\bar M, \bar g)$. Let $a_{l,n}:=\frac{l}{l+n}$ denote the constants introduced in \cite[Lemma B.1]{colding2015uniqueness}. Unless otherwise specified, all Sobolev norms considered in this section are weighted norms with respect to $\mathrm{d}V_{\bar f}$.

\subsection*{Comparison of radius functions I}

The following proposition is similar to \cite[Theorem 5.2]{li2023rigidity}:
\begin{prop}\label{prop:con}
	With assumption \eqref{normalization22}, there exists a constant $L_1=L_1(n,C_V, \sigma)>0$ such that if $\rBC\geq L_1$, then
	\begin{equation*}
		\rA\geq \rBC-2.
	\end{equation*}
\end{prop}
\begin{proof}
For simplicity, we set $L=\rBC$ and let $\varphi_B$ denote the diffeomorphism corresponding to $\rBC$.  

Throughout the proof, we denote by $C$ constants depending only on $n$, and by $C'$ constants depending only on $n,C_V$. All these constants can be different line by line.
	
Choose a cut-off function $\eta$ that is supported in $\{\bar b<L-1/2\}$ and is identically $1$ on $\{\bar b\leq L-1\}$. By Definition \ref{defnradii} (B) and the multiplicative interpolation (see \cite[Theorem 3.2.1]{kryelliptic}) on each ball of size $1/4$, we have on $\{\bar b\leq L-1/4\}$,
	\begin{align}
		\label{highorder1}\rVert \bar g-\varphi_B^* g\rVert_{C^{k, \frac 1 2}}+\rVert\bar f-\varphi_B^* f\rVert_{C^{k, \frac 1 2}}&\leq C e^{-\frac{L^2}{34}},
	\end{align}
for any $k \in [1, 10^8 n\sigma^{-1}]$, provided that $L$ is sufficiently large.

As in the proof of \cite[Theorem 5.2]{li2023rigidity} (see \cite[Equations (5.15), (5.16)]{li2023rigidity}), there exists a diffeomorphism $\varphi_A$, which is a modification of $\varphi_B$ by some diffeomorphism, from $\{\bar b \le L-1/2\}$ to a subset of $M$ such that for $h:=\eta \big(\varphi_A^{*}g-\bar{g}\big),\chi:=\eta\left(\varphi_A^*f-\bar{f}\right)$,
	\begin{align}
\sup_{\bar b \le L-1/2}\lc\rVert h\rVert_{C^{k, \frac 1 2}}+\rVert\chi\rVert_{C^{k, \frac 1 2}}\rc \leq e^{-\frac{L^2}{34}} \quad \text{and} \quad
 \sup_{\bar b \le L-6}\rVert \Div_{\bar f} h\rVert_{C^{2, \frac 1 2}}+\left|\BB(h,\chi)\right| \leq e^{-\frac{(L-6)^2}{8}-\frac{L^2}{34}}, 		\label{improvedgauge}
	\end{align}
for any $k \in [1, 10^6 n\sigma^{-1}]$, provided that $L$ is sufficiently large.
		
	In particular, by Lemma \ref{estideriPhi} (i), we see that for $k \in [1, 10^6 n\sigma^{-1}]$,
		\begin{equation}\label{highorderx1}
\rVert\mathbf{\Phi}(\bar g+h, \bar f+\chi)\rVert_{C^{k-2}}\leq CL e^{-\frac{L^2}{34}}.
	\end{equation}

	Now, we write as in \eqref{eq:decompose}:
		\begin{align*}
h=ug_{S^m}+\zeta \quad \text{and} \quad \chi=\frac{m}{2}u+q
	\end{align*}
	such that $ug_{S^m}$ is the projection of $h$ to $\KK_0 g_{S^m}$. Then, it follows from Proposition \ref{prop:tay2} that
	\begin{align}\label{contractionineq1}
		\alpha^2:= \rVert u\rVert_{L^2}^2\leq C\rVert (1+|\vec{x}|^2)\mathbf{\Phi}(\bar g+h, \bar f+\chi)\rVert_{L^1}+C_\ep \left(\rVert\mathbf{\Phi}(\bar g+h, \bar f+\chi)\rVert_{L^2}^{2-\ep}+|\BB(h,\chi)|^{2-\ep}+\rVert\Div_{\bar f}h\rVert_{W^{1,2}}^{2-\ep}\right).
	\end{align}
	Since $\mathbf{\Phi}(\bar g+h, \bar f+\chi)= \varphi_A^* \Phi$ on $\{\bar b \le L-1\}$, we have
			\begin{align*}
\int_{\{\bar b<L-1\}}|\mathbf{\Phi}(\bar g+h, \bar f+\chi)|^2 \,\mathrm{d}V_{\bar f}= \int_{\{\bar b<L-1\}}|\varphi_A^* \Phi|^2 \,\mathrm{d}V_{\bar f} \le  C \int_{\{\bar b<L-1\}}|\varphi_A^* \Phi|^2 \,\mathrm{d}V_{\varphi_A^* f} =  C \int_{\varphi_A(\{\bar b<L-1\})}|\Phi|^2 \,\mathrm{d}V_{f},
	\end{align*}
	where we used \eqref{improvedgauge} for the inequality. Since $\varphi_A$ is equal to $\varphi_B$ up to an exponential error (see \cite[Equations (5.9)]{li2023rigidity}), it follows from \eqref{equ:integralphi} and \eqref{highorderx1} that
			\begin{align*}
\int_{\{\bar b<L\}}|\mathbf{\Phi}(\bar g+h, \bar f+\chi)|^2 \,\mathrm{d}V_{\bar f}\leq & C\int_{\{\bar b<L\}} |\varphi_B^*\Phi|^2 \,\mathrm{d}V_{\bar f}+\int_{\{L-1 \le \bar b<L\}}|\mathbf{\Phi}(\bar g+h, \bar f+\chi)|^2 \,\mathrm{d}V_{\bar f} \\
\le & C\left(e^{-\frac{L^2}{4-\sigma}}+L^{n+2} e^{-\frac{(L-1)^2}{4}-\frac{L^2}{17}}\right) \le Ce^{-\frac{L^2}{4-\sigma}}. \label{eq:extrainte}
	\end{align*}

Moreover, by H\"older's inequality,
	\begin{align*}
		\rVert(1+|\vec{x}|^2)\mathbf{\Phi}(\bar g+h, \bar f+\chi)\rVert_{L^1(\{\bar b< L\})}&\leq C\rVert \mathbf{\Phi}(\bar g+h, \bar f+\chi)\rVert_{L^2(\{\bar b<L\})}\leq Ce^{-\frac{L^2}{2(4-\sigma)}}.
	\end{align*}
	Using \eqref{improvedgauge}, \eqref{contractionineq1} and choosing $\ep$ to be small, we obtain
	\begin{equation}\label{contreq2}
		\alpha\leq Ce^{-\frac{L^2}{4(4-\sigma)}}.
	\end{equation}
	Moreover, by Proposition \ref{prop:tay1}, it follows that
	\begin{align}\label{contreq1}
		\rVert \zeta\rVert_{W^{2,2}}+\rVert \nabla q\rVert_{W^{1,2}}\leq C\left(\rVert\mathbf{\Phi}(\bar g+h, \bar f+\chi)\rVert_{L^2}+\rVert\Div_{\bar f}h\rVert_{W^{1,2}}+|\BB(h,\chi)|+\alpha^2\right)\leq Ce^{-\frac{L^2}{2(4-\sigma)}}.
	\end{align}
	
For any $x\in \{\bar b<L-1\}$, applying the interpolation (see \cite[Lemma B.1]{colding2015uniqueness}) to the function $U:=([\zeta]_2^2+[\nabla q]_1^2)e^{-\bar f}$ on $B(x,1/2)$, we get that for $k=10^6 n\sigma^{-1}-2$,
	\begin{align}\label{interploation1}
		\rVert U\rVert_{L^\infty(B(x,1/4))}
		\leq& C \rVert U\rVert_{L^1(B(x,1/2))}+C \rVert U\rVert_{L^1(B(x,1/2))}^{a_{k,n}}\rVert\nabla ^kU\rVert_{L^\infty(B(x,1/2))}^{1-a_{k,n}}\nonumber\\
		\leq& C e^{-\frac{L^2}{4-\sigma}}+CL^k e^{-a_{k,n}\frac{L^2}{4-\sigma}-(1-a_{k,n})\frac{L^2}{17}} 
		\leq Ce^{-\frac{L^2}{4-0.9\sigma}}.
	\end{align}
 Here, in the second inequality, we have used \eqref{highorder1} and \eqref{contreq1}. Note that in \eqref{interploation1}, the Sobolev norms are unweighted. 
	
	It follows from \eqref{contreq2} and \eqref{interploation1} that on $\{\bar b<L-1\}$,
	\begin{equation}\label{improvedestimate1}
		[\zeta]_2+[\nabla q]_1\leq Ce^{\frac{\bar f}{2}-\frac{L^2}{8-1.8\sigma}},\quad [u]_2\leq C \bar f \alpha\leq C \bar f e^{-\frac{L^2}{4(4-\sigma)}}.   
	\end{equation}
	
	To get the estimate for $q$ itself, we use the normalization condition and the assumption \eqref{normalization22}. In fact, by \eqref{normalization22}, we have
		\begin{align*}
\int_{\{b>L-2\}} e^{-f} \,\mathrm{d}V_{g} \leq C' e^{-\frac{L^2}{4(4-\sigma)}},
	\end{align*}
	thus, it follows from the normalization condition that
	\begin{align}\label{contractionineq2}
		\left| \int_{\{\bar b<L-1\}} e^{-\bar f} \,\mathrm{d}V_{\bar g}-\int_{\{\bar b<L-1\}}e^{-\varphi_A^*f} \,\mathrm{d}V_{\varphi_A^*g}\right|\leq C' e^{-\frac{L^2}{4(4-\sigma)}}.
	\end{align}
	Since $\left|\,\mathrm{d}V_{\bar g}-\,\mathrm{d}V_{\varphi_A^*g}\right|\leq C[ h]_0\,\mathrm{d}V_{\bar g}\leq C\left([\zeta]_0+[u]_0\right)\,\mathrm{d}V_{\bar g}$, it follows that
	\begin{align}
		&\left|\int_{\{\bar b<L-1\}}e^{-\varphi_A^*f} \,\mathrm{d}V_{\bar g}-\int_{\{\bar b<L-1\}}e^{-\varphi_A^*f} \,\mathrm{d}V_{\varphi_A^*g}\right| \notag\\
		\leq & \left|\int_{\{\bar b<L-1\}}e^{-\chi-\bar f} \,\lc \mathrm{d}V_{\bar g}-\mathrm{d}V_{\varphi_A^*g} \rc \right| \notag\\
		\leq & C\left(\rVert\zeta\rVert_{L^1\left(\{\bar b\leq L-1\}\right)}+\rVert u\rVert_{L^1\left(\{\bar b\leq L-1\}\right)}\right) \le Ce^{-\frac{L^2}{4(4-\sigma)}}.\label{contractionineq3}
	\end{align}
	
	On the other hand, by definition, 
	\begin{align*}
		\int_{\{\bar b<L-1\}}e^{-\bar f} \,\mathrm{d}V_{\bar g}-\int_{\{\bar b<L-1\}}e^{-\varphi_A^*f} \,\mathrm{d}V_{\bar g}=\int_{\{\bar b<L-1\}}(e^{-\chi}-1) e^{-\bar f}\,\mathrm{d}V_{\bar g},
	\end{align*}
	we thus obtain from \eqref{contractionineq2} and \eqref{contractionineq3} that
	\begin{equation}\label{contractionineq4a}
\left|\int_{\{\bar b<L-1\}}(e^{-\chi}-1) e^{-\bar f}\,\mathrm{d}V_{\bar g} \right|\leq C'e^{-\frac{L^2}{4(4-\sigma)}}.
	\end{equation}

By \eqref{improvedgauge}, we have
	\begin{equation}\label{contractionineq4b}
\left|\int_{\{L/2 \le \bar b<L-1\}}(e^{-\chi}-1) e^{-\bar f}\,\mathrm{d}V_{\bar g} \right|\leq C L^n e^{-\frac{L^2}{34}-\frac{L^2}{16}}.
	\end{equation}

Combinining \eqref{contractionineq4a} and \eqref{contractionineq4b}, we obtain
	\begin{align*}
\left|\int_{\{\bar b<L/2\}}(e^{-\chi}-1) e^{-\bar f}\,\mathrm{d}V_{\bar g} \right| \leq C'e^{-\frac{L^2}{4(4-\sigma)}}.
	\end{align*}
Therefore, there exists a point $z_0 \in\{\bar b<L/2\}$ such that $|\chi(z_0)| \le C' e^{-\frac{L^2}{4(4-\sigma)}}$. By \eqref{improvedestimate1}, this implies
	\begin{equation}\label{contreq3}
		\left|q(z_0)\right|\leq |\chi(z_0)|+C |u(z_0)| \leq  C' L^2 e^{-\frac{L^2}{4(4-\sigma)}}.
	\end{equation}
	Combining \eqref{contreq3} with \eqref{improvedestimate1} again, it follows that on $\{\bar b<L-1\}$, 
	\begin{equation}\label{contractionineq5}
		|q|\leq C\bar b e^{\frac{\bar f}{2}-\frac{L^2}{8-1.8\sigma}}+C' L^2 e^{-\frac{L^2}{4(4-\sigma)}}.
	\end{equation}
Indeed, \eqref{contractionineq5} is obviously true on $\{L/2 \le \bar b<L-1\}$. On the other hand, if $z \in \{\bar b<L/2\}$, it follows from \eqref{improvedestimate1} and \eqref{contreq3} that
	\begin{equation*}
		|q(z)|\leq C L e^{\frac{L^2}{32}-\frac{L^2}{8-1.8\sigma}}+|q(z_0)| \le C' L^2 e^{-\frac{L^2}{4(4-\sigma)}}.
	\end{equation*}

By \eqref{improvedestimate1} and \eqref{contractionineq5}, we can conclude that on $\{\bar b<L-1\}$,                    
	\begin{equation}\label{improvedestimate3}
		[h]_2+[\chi]_2\leq C\bar b e^{\frac{\bar f}{2}-\frac{L^2}{8-1.8\sigma}}+C \bar f e^{-\frac{L^2}{4(4-\sigma)}}+C' L^2 e^{-\frac{L^2}{4(4-\sigma)}}\leq C'e^{\frac{\bar f}{4}-\frac{L^2}{16-3\sigma}}, 
	\end{equation}
	where in the last inequality, we have used the following facts: 
	\begin{align*}
		\bar f \leq C e^{\frac{\bar f}{4}} \quad \text{and} \quad  \bar b e^{\frac{\bar f}{4}-\frac{L^2}{8-1.8\sigma}+\frac{L^2}{16-3\sigma}}\leq CL e^{\frac{(L-1)^2}{16}-\frac{L^2}{16}}\leq C.
	\end{align*}
	
	Using \eqref{highorder1}, \eqref{improvedestimate3} and another interpolation at the scale of $L^{-1}$, we conclude that on $\{\bar b \le L-2\}$
		\begin{equation}\label{improvedestimate3a}
		[h]_5+[\chi]_5\leq e^{\frac{\bar f}{4}-\frac{(L-2)^2}{16}}.
	\end{equation}

In summary, the diffeomorphism $\varphi_A$ together with the estimates in \eqref{improvedestimate3a} ensures that the conditions in the definition of $\rA$ are satisfied. 
\end{proof}

\begin{rem}\label{rem:constan1}
From the proof of Proposition \ref{prop:con}, it is clear that the constant $15$ in \eqref{normalization22} can be replaced with any positive constant smaller than $16-5\sigma$.
\end{rem}

The proof of Proposition \ref{prop:con} also yields the following result:

\begin{thm}\label{lojawithradius}
There exist constants $L_2= L_2(n,Y, \sigma)>0$ and $C=C(n,Y)>0$ such that the following holds.
	Let $\XX=\{M,(g(t))_{t\in I}\}$ be a closed Ricci flow with entropy bounded below by $-Y$. Assume $x_0^*=(x_0,t_0)\in \XX$ and $[t_0-2r^2,t_0]\subset I$. If the weighted Riemannian manifold $\left(M,r^{-2}g(t_0-r^2),f_{x_0^*}(t_0-r^2)\right)$ satisfies $\rBC\geq L_2 $, then
	\begin{equation*}
		\left|\WW_{x_0^*}(r^2)-\Theta_m\right|\leq C \exp\lc-\frac{3\rBC^2}{16} \rc,
	\end{equation*}
	where $\WW_{x_0^*}$ is the pointed $\WW$-entropy at $x_0^*$; see Definition \ref{defnentropy}.
\end{thm}
\begin{proof}
	Without loss of generality, we assume $t_0=0$, $r=1$ and choose a small parameter $\ep\ll 1$ to be determined later. Also, we set $L=\rBC$, $g_0=g(-1)$ and $f_0=f_{x_0^*}(-1)$ for simplicity. 
	
	By Definition \ref{defnradii}, we can find a diffeomorphism $\varphi_B$ from $\{\bar b \le L\}$ onto a subset of $M$ such that
	\begin{align*}
		\left[\bar g-\varphi_B^* g_0\right]_0+\left[\bar f-\varphi_B^* f_0\right]_0\leq e^{-\frac{L^2}{33}},
	\end{align*}
	and
	\begin{equation*}
\int_{\{\bar b<L\}}\left|\varphi_B^*\Phi\right|^2 \,\mathrm{d}V_{\bar f} \le e^{-\frac{L^2}{4-\sigma}},
	\end{equation*}
	where $\Phi=g_0/2-\Ric\left(g_0\right)-\na^2 f_0$. Furthermore, all $C^k$-norms of $\bar g-\varphi_B^* g_0$ and $\bar f-\varphi_B^* f_0$ are bounded by $1$ for $k \in [1, 10^{10} n \sigma^{-1}]$. By Theorem \ref{thm:upper} (i), the assumption \eqref{normalization22} holds for $C_V=C_V(n,Y)$. Now we choose $L_2=L_1(n, C_V)$, where $L_1$ is the same constant as in Proposition \ref{prop:con} so that it applies to the weighted Riemannian manifold $\left(M, g(-1),f(-1)\right)$.
	
	Choose a cut-off function $\eta$ on $\bar M$ such that $\eta=0$ outside $\{\bar b<L-1\}$ and $\eta =1$ on $\{\bar b<L-2\}$. Let $g_1=\bar g+\eta (\varphi_B^*g_0-\bar g), f_1=\bar f+\eta (\varphi_B^*f_0-\bar f)$. By Proposition \ref{prop:remainder} and construction, we have
	\begin{equation}\label{2ndloja2}
		\left|\WW(g_0,f_0)-\WW(g_1,f_1)\right|\leq C(n,Y,\ep) e^{-\frac{(1-\ep)L^2}{4}}.
	\end{equation}
	
	By the proof of Proposition \ref{prop:con} (see \eqref{improvedgauge}), we can find a new diffeomorphism $\varphi_A$ such that for $h=\eta \left(\varphi_A^*g_1-\bar g\right), \chi=\eta \left(\varphi_A^*f_1-\bar f\right)$, we have
	\begin{align*}
		\sup_{\bar b \le L-6}\rVert \Div_{\bar f} h\rVert_{C^{2, \frac 1 2}}+\left|\BB(h,\chi)\right|&\leq e^{-\frac{(L-6)^2}{8}-\frac{L^2}{34}}.
	\end{align*}
	
	Set $g_2=\bar g+h,f_2=\bar f+\chi$ and write $h=ug_{S^m}+\zeta, \chi=mu/2+q$, where $ug_{S^m}$ is the projection of $h$ to $\KK_0 g_{S^m}$. Note that $h$ and $\chi$ are both supported on $\{\bar b<L\}$, then by Theorem \ref{lojaforF}, 
	\begin{align}\label{2ndloja4}
		\left|\WW(g_2,f_2)-\Theta_m\right|\leq C(n)\left(\alpha^3+\mathfrak{X}^2+\left|\VV(g_2,f_2)-1\right|\right),
	\end{align}
	where $\alpha=\rVert u\rVert_{L^2}$ and $\mathfrak{X}=\|\mathbf{\Phi}(g_2,f_2)\|_{L^2}+\|\Div_{\bar f} h\|_{W^{1,2}}+|\mathcal B(h,\chi)|$. By Theorem \ref{thm:upper} (i) again, we know that
	\begin{align}\label{voldiff}
		\left|\VV(g_2,f_2)-1\right|\leq C(n,Y,\ep)e^{-\frac{(1-\ep)L^2}{4}}.
	\end{align}
	
	As in the proof of Proposition \ref{prop:con} (see \eqref{contreq2} and \eqref{contreq1}), we conclude that
	\begin{align}\label{2ndloja7}
		\mathfrak{X}+\alpha^2 \le C(n) e^{-\frac{L^2}{8}}.
	\end{align}
Plugging \eqref{voldiff} and \eqref{2ndloja7} into \eqref{2ndloja4}, we obtain
	\begin{align}\label{2ndloja8}
		\left|\WW(g_2,f_2)-\Theta_m\right|\leq C(n)(e^{-\frac{3L^2}{16}}+e^{-\frac{L^2}{4}})+C(n,Y,\ep)e^{-\frac{(1-\ep)L^2}{4}} \leq C(n,Y) e^{-\frac{3L^2}{16}},
	\end{align}
	provided that $\ep$ is chosen to be small. 
	
	By constructions of $\varphi_A$ and $\varphi_B$, we have
	\begin{align}\label{2ndloja9}
		\left|\WW(g_2,f_2)-\WW(g_1,f_1)\right|\leq C(n,Y,\ep)e^{-\frac{(1-\ep)L^2}{4}}.
	\end{align}
	Combining \eqref{2ndloja2}, \eqref{2ndloja8} and \eqref{2ndloja9}, we finally conclude by choosing a small $\ep>0$ that
	\begin{align*}
		\left|\WW(g_0,f_0)-\Theta_m\right|\leq C(n,Y) e^{-\frac{3L^2}{16}}.
	\end{align*}
	Since $L=\rBC$ and $\WW(g_0, f_0)=\WW_{x_0^*}(1)$, this completes the proof.
\end{proof}



\section{Lojasiewicz inequalities for cylindrical geometries}\label{seclojaRF}

In this section, we consider a closed Ricci flow $\XX=\{M^n,(g(t))_{t\in I} \}$ with entropy bounded below by $-Y$ such that $[-10,0]\subset I$. Throughout, we fix a spacetime point $x_0^*=(x_0,0)$, define $\tau=-t$, and set
\begin{align*}
  \begin{dcases}
  & \mathrm{d}\nu_t=\mathrm{d}\nu_{x^*_0;t}=(4\pi \tau)^{-\frac n 2}e^{-f} \,\mathrm{d}V_{g(t)}, \\
		& \Phi= \frac{g}{2}-\tau \left( \Ric+\na^2 f\right),\\
		& F=\tau f.
        \end{dcases}
\end{align*}
Moreover, we define $b=2\sqrt{\left|f(-1)\right|}$.

As before, the model space we consider is the weighted cylinder:
\begin{align*}
\mathcal C^{n-m}_{-1}=(\bar M,\bar g(-1),\bar f(-1))=\left(\R^{n-m}\times S^{m}, g_E \times g_{S^m}, \frac{|\vec{x}|^2}{4}+\frac{m}{2}+\Theta_m \right),
\end{align*}
where $m \in \{2, \ldots, n-1\}$. We set $(\bar M, \bar g(t), \bar f(t))$ to be the induced Ricci flow such that $t=0$ is the singular time coupled with $\bar f(t):=|\vec{x}|^2/4\tau+m/2+\Theta_m$. We set $\bar F=|\vec{x}|^2/4$, $\bar b=2\sqrt{\bar f(-1)}$, and fix $\bar p$ to be a minimum point of $\bar b$.

\begin{defn}[$\r_{C, \delta}$-radius]\label{defnrCd}\index{$\rCd$}
	For the weighted Riemannian manifold $\left(M,g(-1),f(-1)\right)$,
	$\r_{C,\delta}$ is defined as the largest $L$ such that there exists a diffeomorphism $\varphi_C$ from $\{\bar b \le L\}$ of $\bar M$ onto a subset of $M$ such that $f \lc \varphi_C(\bar p), -1 \rc \le n$ and
	\begin{align*}
		\left[\bar g-\varphi_C^* g(-1)\right]_2\leq \delta.
	\end{align*}
\end{defn}

\begin{prop}\label{stabilitymetric}
	For any small $\ep>0$, there exists $\bar \delta=\bar \delta(n,\ep)>0$ such that if $\delta \le \bar \delta$ and $\mathbf{r}_{C,\delta} \ge \bar \delta^{-2}$, then
	\begin{align*}
		\left[\varphi_C^*g(t)-\bar g(t)\right]_{[\ep^{-1}]} \le \ep
	\end{align*}
	on $\left\{\bar b \le \mathbf{r}_{C,\delta}-\bar \delta^{-1}\right\} \times [-9 , -\ep]$, where the norm $[\cdot]_{[\ep^{-1}]}$ is taken with respect to $\bar g(t)$.
\end{prop}
\begin{proof}
For simplicity, we set $L=\mathbf{r}_{C,\delta}$ and $g_0(t)=\varphi_C^*g(t)$, which is regarded as a Ricci flow on $\{\bar b \le L\} \times [-10, 0]$.

It follows from the two-sided pseudolocality theorem (see Theorem \ref{thm:twoside}) and Shi’s local estimates (see \cite{shi89}) that for a fixed $\ep>0$, there exist constants $\theta=\theta(n, \ep) \in (0, \ep)$ such that if
	\begin{align*}
		\left[g_0(t)-\bar g(t)\right]_2\leq \theta
	\end{align*}
	on $\{\bar b \le r\}$, for some $r \in (1, L)$ and some $t \in [-9, -\ep]$, then
	\begin{align}\label{eq:rad01}
		\left|\na^k\Rm(g_0)\right| \le C(n,k , \ep)
	\end{align}
holds on the region $\left\{\bar b \le r-1 \right\} \times [t-10\theta, t+2\theta]$.

We first claim that 
\begin{align*}
\left[g_0(t)-\bar g(t)\right]_{[\theta^{-1}]} \le \theta
	\end{align*}
on $\left\{\bar b \le L-\bar \delta^{-1}/2 \right\} \times [-1-9 \theta , -1+\theta]$, provided that $\bar \delta$ is sufficiently small. Suppose, for contradiction, that the conclusion fails. Then, there exist sequences $L_i \to +\infty$ and $\ep_i \to 0$ with $L_i \ge \ep_i^{-2}$, and a sequence of Ricci flows $g_i(t)$ such that
		\begin{align} \label{eq:rad12}
		\left[g_i(-1)-\bar g(-1)\right]_{2} \le \ep_i
	\end{align}
on $\{\bar b \le L_i\}$. By \eqref{eq:rad01}, we obtain for large $i$
			\begin{align} \label{eq:rad11}
\left|\na^k\Rm(g_i(t))\right| \le C(n,k , \ep)
	\end{align}
on the region $\left\{\bar b \le L_i-1 \right\} \times [t-10\theta, t+2\theta]$. By our assumption, there exist $q_i$ with $\bar b(q_i) \le L_i-\ep_i^{-1}/2$ and $t_i \in [-1-9\theta, -1+\theta]$ such that
		\begin{align} \label{eq:rad13}
\left[g_i(t_i)-\bar g(t_i)\right]_{[\theta^{-1}]}(q_i) > \theta.
	\end{align}

It follows from \eqref{eq:rad12} and \eqref{eq:rad11} that, after taking a subsequence, the following convergence holds for $t \in (t-10\theta, t+2\theta)$:
	\begin{align*}
		\left(\bar M, g_i(t), q_i\right) \xrightarrow[i\to\infty]{\quad C_{\loc}^{\infty}\quad}\left(\bar M, g_{\infty}(t), q_{\infty}\right),
	\end{align*}
where the smooth convergence is in the fixed coordinate system given by $\bar M$. Moreover, $g_{\infty}(-1)=\bar g(-1)$. By the two-sided uniqueness of the Ricci flow (see \cite[Theorem 1.1]{chen2006uniqueness} and \cite[Theorem 1]{kotschwar2010backwards}), this implies that $g_{\infty}(t)=\bar g(t)$ for any $t \in [-1-9\theta,-1+\theta]$. However, this contradicts \eqref{eq:rad13} if $i$ is sufficiently large.

By the same argument with adjusting constants and induction, we can prove that for any $j \in \mathbb N$ with $j \le 
(1-\ep)/\theta$, we have 
\begin{align*}
\left[g_0(t)-\bar g(t)\right]_{[\theta^{-1}]} \le \theta
	\end{align*}
on $\left\{\bar b \le L-\bar \delta^{-1} \lc 2^{-1}+\cdots+2^{-j} \rc \right\} \times [-1-9j\theta , -1+j\theta]$, provided that $\bar \delta$ is sufficiently small. 
	
Since $\theta<\ep$, this completes the proof.
\end{proof}

\begin{prop}\label{stabilitypotential}
Under the same assumptions of Proposition \ref{stabilitymetric}, we have
	\begin{equation}\label{stabpotential}
		(1-100\ep)\bar F(x)-C(n,Y,\ep)\leq F(\varphi_C(x),t)\leq (1+100\ep)\bar F(x)+C(n,Y,\ep)
	\end{equation}
for $(x,t)\in \left\{\bar b\leq \r_{C,\delta}-\bar \delta^{-1}\right\}\times [-8,-1/8]$.
\end{prop}
\begin{proof}
We set $p=\varphi_C(\bar p)$. By Definition \ref{defnrCd}, since $f \lc p, -1 \rc \le n$, it follows from Theorem \ref{thm:upper} (ii) that $(p, -1)$ is an $H(n,Y)$-center of $x_0^*$. 

	\begin{claim}\label{extrahcenter1}
For any $t \in [-9, -\ep]$, $(p, t)$ is an $H$-center of $x_0^*$, where $H=H(n, Y, |t|)>0$ is a constant.
	\end{claim}
	
For $t \in [-9, -1)$, it follows from Proposition \ref{stabilitymetric} that $|\Rm| \le C(n)$ along ${p} \times [t, -1]$. Hence, by \cite[Proposition 2.21(i)]{fang2025RFlimit}, we conclude that $(p, t)$ is an $H(n, Y)$-center of $(p, -1)$.
By the monotonicity property (see \cite[Proposition 2.12]{fang2025RFlimit}), we have
	\begin{align*}
C(n, Y) \ge d_{W_1}^{-1}(\nu_{-1}, \delta_{p}) \ge d_{W_1}^{t}(\nu_{t}, \nu_{p, -1;t}) \ge d_{W_1}^{t}(\nu_{t}, \delta_{p})-d_{W_1}^{t}(\delta_{p}, \nu_{p, -1;t}) \ge d_{W_1}^{t}(\nu_{t}, \delta_{p})-C(n, Y),
	\end{align*}
from which it follows that $(p, t)$ is an $H(n, Y)$-center of $x_0^*$.

For $t \in (-1, -\ep]$, by the reproduction formula we have
	\begin{align*}
		K(x_0, 0;p, -1)=\int_M K(w,t;p, -1)  \,\mathrm{d}\nu_{t}(w).
	\end{align*}
Since $K(x_0, 0; p, -1) \ge c(n) > 0$ by assumption, it follows from Theorem \ref{thm:upper}, Proposition \ref{stabilitymetric}, and \cite[Proposition 9.16(b)]{bamler2023compactness} that any $H_n$-center $(z, t)$ of $x_0^*$ must satisfy
	\begin{align*}
d_t(p, z) \le C(n, Y, |t|),
	\end{align*}
which completes the proof of Claim \ref{extrahcenter1}.

In particular, $(p,t)$ is an $H(n, Y)$-center of $x_0^*$ for any $t\in [-8,-1/8]$. By Theorem \ref{thm:upper} (ii), we conclude that
	\begin{align}\label{potential1}
		K(x_0, 0; x, t) \le C(n,Y,\ep) \tau^{-\frac n 2} \exp\left(-\frac{d_t^2(p,x)}{4(1+\ep)\tau}\right)
	\end{align}
for any $(x, t) \in M \times [-8, -1/8]$. Since $K(x_0,0;\cdot,t)=(4\pi\tau)^{-n/2}e^{-f}$, it follows from \eqref{potential1} that for any $(x, t) \in M \times [-8, -1/8]$,
	\begin{align*}
f(x, t) \ge \frac{d_t^2(p, x)}{4(1+\ep) \tau}-C(n, Y, \ep),
	\end{align*}
which implies on $M \times [-8, -1/8]$,
	\begin{align}\label{potential102}
F(x, t) \ge \frac{d_t^2(p, x)}{4(1+\ep)}-C(n, Y, \ep).
	\end{align}

	\begin{claim}\label{extraclaim2}
For any $(x,t)\in \{\bar b\leq \r_{C,\delta}-\bar \delta^{-1}\}\times [-9, -\ep]$, 
	\begin{align}\label{potential104}
		(1-\ep)\bar d_t(\bar p, x) \leq d_t(p,\varphi_C(x))\leq 	(1+\ep) \bar d_t(\bar p, x).
	\end{align}
	\end{claim}

The upper bound in \eqref{potential104} is straightforward. Indeed, fix $t \in [-9, -\ep]$ and let $\tilde \gamma$ be a minimizing geodesic from $\bar p$ to $x$ with respect to $\bar g(t)$. Since $x \in \{\bar b\leq \r_{C,\delta}-\bar \delta^{-1}\}$, the entire curve $\tilde \gamma$ lies in $\{\bar b \le \r_{C,\delta}\}$. By Proposition \ref{stabilitymetric}, the length of $\varphi_C(\tilde \gamma)$ with respect to $g(t)$ is at most $(1+\ep)\bar d_t(\bar p, x)$, yielding the upper bound.

For the lower bound in \eqref{potential104}, fix $t \in [-9, -\ep]$ and let $\gamma(s)$ for $s \in [0, L]$ be a minimizing geodesic from $p$ to $\varphi_C(x)$ with respect to $g(t)$, where $L=d_t(p, \varphi_C(x))$. If $\gamma$ lies entirely in $\varphi_C \lc \{\bar b\leq \r_{C,\delta}\} \rc$, then applying $\varphi_C^{-1}$ and using the previous argument gives the lower bound. Otherwise, let $s_0$ be the smallest $s \in [0, L]$ for which $\gamma(s_0)$ lies on the boundary of $\varphi_C \lc \{\bar b\leq \r_{C,\delta}\} \rc$. Define $\gamma'=\varphi_C^{-1}(\gamma \vert_{[0, s_0]})$. By the definition of $s_0$, its length with respect to $\bar g(t)$ is at least $\r_{C,\delta}$. Thus, by Proposition \ref{stabilitymetric},
	\begin{align*}
d_t(p,\varphi_C(x)) \ge (1-\ep) \r_{C,\delta} \ge (1-\ep) \bar d_t(\bar p, x),
	\end{align*}
which establishes the lower bound in \eqref{potential104} and hence finishes the proof of Claim \ref{extraclaim2}.
	
By the definition of $\bar F$, it is easy to show that for any $(x, t) \in \bar M \times [-9, -\ep]$,
		\begin{align}\label{potential106}
	\bar F(x) \le \frac{\bar d_t^2( \bar p, x)}{4} \le  \bar F(x) +C(n).
	\end{align}
	
Combining \eqref{potential102}, \eqref{potential104} and \eqref{potential106}, we conclude that
	\begin{align*}
		F(\varphi_C(x),t)\geq & \frac{d_t^2(p,\varphi_C(x))}{4(1+\ep)} -C(n,Y,\ep) \ge  \frac{(1-\ep)^2}{4(1+\ep)} \bar d_t^2(\bar p, x) -C(n,Y,\ep) \ge \frac{(1-\ep)^2}{1+\ep}\bar F(x)-C(n, Y, \ep),
	\end{align*}
	which implies the lower bound in \eqref{stabpotential}.
	
	For the upper bound in \eqref{stabpotential}, we first choose an $H_n$-center $(z, -\ep)$ of $x_0^*$. From the reproduction formula, we have for $(x, t) \in M \times [-8, -1/8]$,
	\begin{align}\label{lowerpotential1}
		K(x_0, 0;x, t)\geq \int_{B_{-\ep}(z,\sqrt{2H_n\ep})} K(w,-\ep;x,t) \,\mathrm{d}\nu_{-\ep}(w).
	\end{align}
	 Since $(p,-\ep)$ is an $H(n,Y, \ep)$-center of $x_0^*$ by Claim \ref{extrahcenter1}, it follows that
	\begin{align}\label{distHncenter}
		d_{-\ep}(p,z)\leq C(n,Y, \ep).
	\end{align} 
	
	For $w\in B_{-\ep}\big(z,\sqrt{2H_n\ep}\big)$, we claim that for any $(x,t)\in \varphi_C\lc \{\bar b\leq \r_{C,\delta}-\bar \delta^{-1}\} \rc \times [-8,-1/8]$, it holds that
	\begin{align}\label{lowerpotential2}
		K(w,-\ep; x,t) \ge c(n,Y,\ep) \exp\left(-\frac{d^2_{-\ep}(w,x)}{4(1-\ep)(-\ep-t)}\right)>0.
	\end{align}
	In fact, by \cite[Corollary 9.5]{perelman2002entropy}, we have
	\begin{align*}
		K(x,-\ep;x,t)\geq \left(4\pi(-\ep-t)\right)^{-n/2}e^{-l_{x,-\ep}(x,t)} \ge c(n) e^{-l_{x,-\ep}(x,t)}>0,
	\end{align*}
	where $l_{x,-\ep}(x,t)$ is the reduced distance. By Proposition \ref{stabilitymetric}, $(1-t)|\Rm|\leq C(n)$ along $\{x\}\times [-9,-\ep]$, thus,
	\begin{align*}
		l_{x,-\ep}(x,t)\leq \frac{1}{2\sqrt{-\ep-t}}\int_{t}^{-\ep}\sqrt{-\ep-s}\scal(x,s) \,\mathrm{d}s\leq C(n),
	\end{align*}
	which implies 
		\begin{align}\label{potential105}
K(x,-\ep;x,t)\geq c(n)>0.
	\end{align}

By the same argument as in \cite[Corollary 6]{li2020heat} and Theorem \ref{thm:upper} (ii), we see that for any $A>1$,
	\begin{align*}
		K(x,-\ep;x,t)\leq C(n,Y, \ep, A)K(w,-\ep;x,t)^{\frac{1}{1+A}}\exp\left(\frac{d^2_{-\ep}(w,x)}{4A(-\ep-t)}\right).
	\end{align*}
Combined with \eqref{potential105}, we obtain
	\begin{align*}
		K(w,-\ep;x,t)\geq c(n,Y,A,\ep) \exp\left(-\frac{(1+A)d^2_{-\ep}(w,x)}{4A(-\ep-t)}\right)>0.
	\end{align*}
	By choosing $A :=\ep^{-1}$, we obtain \eqref{lowerpotential2}. 
	
	Combining \eqref{distHncenter} and \eqref{lowerpotential2}, it follows that for $(x,t)\in \varphi_C\lc \{\bar b\leq \r_{C,\delta}-\bar \delta^{-1}\} \rc \times [-8,-1/8]$,
	\begin{align}\label{lowerpotential}
		K(w,-\ep;x,t) \ge c(n,Y,\ep) \exp\left(-\frac{d^2_{-\ep}(p,x)}{4(1-\ep)(-\ep-t)}\right)>0,
	\end{align}
	by slightly adjusting $\ep$. From \eqref{lowerpotential1}, \eqref{lowerpotential} and the fact that $\nu_{-\ep}\left(B_{-\ep}(z, \sqrt{2H_n \ep})\right)\geq 1/2$ (cf. Proposition \ref{existenceHncenter}), we obtain
	\begin{align*}
		K(x_0, 0;x, t) \geq c(n,Y,\ep) \exp\left(-\frac{d^2_{-\ep}(p,x)}{4(1-\ep)(-\ep-t)}\right)>0
	\end{align*}
	for $(x,t)\in \varphi_C\lc \{\bar b\leq \r_{C,\delta}-\bar \delta^{-1}\} \rc \times [-8,-1/8]$. Combining this with \eqref{potential104} and \eqref{potential106},	we obtain for $(x, t) \in \{\bar b\leq \r_{C,\delta}-\bar \delta^{-1}\} \times [-8,-1/8]$
	\begin{align*}
		F(\varphi_C(x),t)\leq &\frac{(-t)}{(1-\ep)(-\ep-t)} \frac{d^2_{-\ep}(p,\varphi_C(x))}{4}+C(n,Y,\ep)\leq \frac{(1+\ep)^2(-t)}{(1-\ep)(-\ep-t)} \frac{\bar d^2_{-\ep}(\bar p,x)}{4}+C(n,Y,\ep)\nonumber\\
		\leq &\frac{(1+\ep)^2(-t)}{(1-\ep)(-\ep-t)} \bar F(x)+C(n,Y,\ep)\le (1+100 \ep) \bar F(x)+C(n, Y, \ep),
	\end{align*}
	which yields the upper bound in \eqref{stabpotential}.
	
	In sum, the proof is complete.
\end{proof}

\begin{cor}\label{highestimatepotential1}
Under the same assumptions of Proposition \ref{stabilitymetric}, for any $1 \le l \le \ep^{-1}$, we have
	\begin{align*}
		[\varphi_C^*F]_{l} \le C(n,Y,\ep) \exp \lc 10^5 \ep l \bar F \rc
	\end{align*}
	on $\{\bar b \le \r_{C,\delta}-2 \bar \delta ^{-1}\} \times [-4 ,-1/4]$.
\end{cor}
\begin{proof}
We set $K=K(x_0, 0;\cdot, \cdot)=(4\pi \tau)^{-\frac n 2} e^{-f}$. Then it is clear from $-\partial_t K=\Delta K-\scal K$ and the standard interior estimates that if $\delta\leq \bar \delta$, then for any $(x,t)\in \varphi_C\lc\{\bar b \le \r_{C,\delta}-2 \bar \delta^{-1}\} \rc \times [-4 ,-1/4]$,
	\begin{align*}
		\rVert K\rVert_{C^l\left( B_t(x,1)\times (t-\ep,t+\ep)\right)}\leq C(n,\ep, l)\rVert K\rVert_{C^0\left( B_t(x,2)\times (t-2\ep,t+2\ep)\right)}.
	\end{align*}
Thus, for any $l \in \mathbb N$,
\begin{align}\label{highorderesti1}
		\abs{\sum_{i_1+\cdots+i_k=l} \na^{i_1} f*\cdots * \na^{i_k}f(x, t)}\leq C(n,\ep,l)\exp\left(f(x,t)-\inf_{B_t(x,2)\times (t-2\ep,t+2\ep)}f\right).
	\end{align}
From Proposition \ref{stabilitypotential}, we obtain
	\begin{align}\label{highorderesti2}
		\left|f(x,t)-\inf_{B_t(x,2)\times (t-2\ep,t+2\ep)}f\right|\leq 10^5 \ep \bar F(\varphi_C^{-1}(x))+C(n,Y,\ep).
	\end{align}
	Combining \eqref{highorderesti1} with \eqref{highorderesti2}, it follows from induction that on $\varphi_C\lc\{\bar b \le \r_{C,\delta}-2 \bar \delta^{-1}\} \rc \times [-4 ,-1/4]$,
		\begin{align*}
		[F]_l \leq C(n,Y,\ep, l) \exp\lc 10^5 \ep l \bar F\circ\varphi_C^{-1} \rc,
	\end{align*}
for any $l \ge 1$. Thus, the conclusion holds from Proposition \ref{stabilitymetric}.
\end{proof}

Next, we define a function $\tilde F$ on $M$ such that
\begin{align} \label{eq:modiF}
	\square \tilde F=-\frac{n}{2} \quad \text{and} \quad \tilde F=F \quad \text{on} \quad t=-2.
\end{align}

\begin{prop} \label{betterpotential}
We have
	\begin{align*}
		\sup_{t \in [-2,-1/2]} \int_{M} \left|\na(\tilde F-F)\right|^2 \,\mathrm{d}\nu_t+\int_{-2}^{-1/2} \int_M \left|\na^2 (\tilde F-F)\right|^2\,\mathrm{d}\nu_t \mathrm{d}t \le 6 \left( \mathcal W_{x_0^*}(1/2)-\mathcal W_{x_0^*}(2) \right).
	\end{align*}
	In particular, 
	\begin{align*}
		\int_M \left|\na(\tilde F-F)\right|^2 \,\mathrm{d}\nu_{-1}+\int_{-2}^{-1/2} \int_M \left|\frac{g}{2}-\na^2 \tilde F-\tau \Ric\right|^2\,\mathrm{d}\nu_t \mathrm{d}t \le 7 \left( \mathcal W_{x_0^*}(1/2)-\mathcal W_{x_0^*}(2) \right).
	\end{align*}
\end{prop}

\begin{proof}
We set $w= \tau(2\Delta f-|\nabla f|^2+\scal)+f-n$ and $u=\tilde F-F$. Then the following evolution equation holds from \eqref{evolutionoff}:
	\begin{align*}
		\square u=w \quad \text{and} \quad u=0 \quad \mathrm{at}\quad t=-2.
	\end{align*}
	By the weighted Bianchi identity (see Lemma \ref{weightedBianchilem}), we calculate 
	\begin{align*}
		\frac{\dd}{\dd t}\int_{M} \left|\nabla u\right|^2\,\mathrm{d}\nu_t&=\int_{M} \square \left|\nabla u\right|^2\,\mathrm{d}\nu_t\nonumber\\
		&=2\int_{M} \la\nabla\square u,\nabla u\ra \,\mathrm{d}\nu_t-2\int_{M} \left|\nabla^2 u\right|^2\,\mathrm{d}\nu_t\nonumber\\
		&=2\int_{M} \la\nabla w,\nabla u\ra \,\mathrm{d}\nu_t-2\int_{M} \left|\nabla^2 u\right|^2\,\mathrm{d}\nu_t\nonumber\\
		&=4\int_{M} \la \Div_f\TT,\nabla u\ra \,\mathrm{d}\nu_t-2\int_{M} \left|\nabla^2 u\right|^2\,\mathrm{d}\nu_t,
	\end{align*}
	where $\TT=\tau \Ric+\nabla^2(\tau f)-\frac{g}{2}$. Using integration by parts and integrating in time, we obtain that for any $t_1\in [ -2, 0]$,
	\begin{align}\label{splitting3}
		2\int_{-2}^{t_1}\int_{M} \left|\na^2u \right|^2\,\mathrm{d}\nu_t \mathrm{d}t&=-4\int_{-2}^{t_1}\int_{M}\la\TT,\nabla^2 u\ra \,\mathrm{d}\nu_t \mathrm{d}t-\int_{M}\left|\nabla u\right|^2\,\mathrm{d}\nu_{t_1} \leq \int_{-2}^{t_1}\int_{M}4\left|\mathcal{T}\right|^2+\left|\nabla^2u \right|^2\,\mathrm{d}\nu_t \mathrm{d}t
	\end{align}
	and thus we get
	\begin{equation}\label{constructsplittingmap1}
		\int_{-2}^{-1/2}\int_{M} \left|\nabla^2u \right|^2\,\mathrm{d}\nu_t \mathrm{d}t \leq 4 \int_{-2}^{-1/2}\int_{M} \left|\TT\right|^2\,\mathrm{d}\nu_t \mathrm{d}t.
	\end{equation}
	
	Note that by Proposition \ref{propNashentropy} (iii),
		\begin{align*}
\WW_{x_0^*}(1/2)-\WW_{x_0^*}(2)=	2 \int_{-2}^{-1/2}\int_M\tau^{-1}\left|\TT\right|^2\,\mathrm{d}\nu_t \mathrm{d}t \ge \int_{-2}^{-1/2}\int_M\left|\TT\right|^2\,\mathrm{d}\nu_t \mathrm{d}t.
	\end{align*}
Combining this with \eqref{constructsplittingmap1} and using the definition of $u$ and $\TT$, we have
	\begin{align*}
\int_{-2}^{-1/2}\int_{M} \left|\tau \Ric+\nabla^2 \tilde F-\frac{g}{2}\right|^2\,\mathrm{d}\nu_t \mathrm{d}t  \le  5 \int_{-2}^{-1/2}\int_{M} \left|\TT\right|^2\,\mathrm{d}\nu_t \mathrm{d}t \le 5 \left(\mathcal{W}_{x_0^*}( 1/2)-\mathcal{W}_{x_0^*}(2)\right).
	\end{align*}	
	
Moreover, by \eqref{splitting3}, for any $t_1\in [ -2,-1/2]$,
	\begin{align*}
		\int_{M}\left|\nabla u\right|^2\,\mathrm{d}\nu_{t_1}&\leq -4\int_{-2}^{t_1}\int_{M}\la\TT,\nabla^2 u\ra \,\mathrm{d}\nu_t \mathrm{d}t-2\int_{-2}^{t_1}\int_{M} \left|\na^2u \right|^2\,\mathrm{d}\nu_t \mathrm{d}t\\
		&\leq 2\int_{-2}^{-1/2}\int_M |\mathcal{T}|^2\,\mathrm{d}\nu_t \mathrm{d}t	 \le 2 \left(\mathcal{W}_{x_0^*}( 1/2)-\mathcal{W}_{x_0^*}(2)\right).
	\end{align*}
	
Thus, it is clear that the estimates hold.
\end{proof}

\begin{lem} \label{estimateforbetterpotential}
In the same setting as Proposition \ref{stabilitymetric}, we have
	\begin{equation}\label{stabpotentialbetter}
		(1-300\ep)\bar F(x)-C(n,Y,\ep)\leq \tilde F(\varphi_C(x),t)\leq (1+300\ep)\bar F(x)+C(n,Y,\ep).
	\end{equation}
for any $(x,t)\in \left\{\bar b\leq \r_{C,\delta}-\bar \delta^{-1}\right\}\times [-2,-1/2]$. Moreover, for any $1 \le l \le \ep^{-1}$, we have on $\{\bar b \le \mathbf{r}_{C,\delta}- 2 \bar \delta^{-1}\} \times [-2, -1/2]$,
	\begin{align*}
		[\varphi_C^*\tilde F]_{l} \le C(n,Y,\ep) \exp\lc 10^5 \ep l \bar F \rc.
	\end{align*}
\end{lem}

\begin{proof}
Since $w= \tau(2\Delta f-|\nabla f|^2+\scal)+f-n\leq0$, it follows from
	\begin{align*}
		\square F=-\frac{n}{2}-w \ge \square \tilde F,
	\end{align*}
and the maximum principle that
	\begin{align*}
		\tilde F(x,t) \leq F(x,t)
	\end{align*}
	for any $(x, t) \in M \times [-2, 0]$. By Proposition \ref{stabilitypotential}, this yields the upper bound of \eqref{stabpotentialbetter}.
	
	On the other hand, for any $(x,t)\in \varphi_C\lc\{\bar b \le \mathbf{r}_{C,\delta}- \bar \delta^{-1}\} \rc \times [-2, -1/2]$ and any its $H_n$-center $(z,-2)$, it follows from $\square \left( \tilde F+nt/2 \right)=0$ that
	\begin{align}\label{betterpoten2}
		\tilde F(x,t)+\frac{n}{2}t=\int_M F(\cdot,-2)-n \,\mathrm{d}\nu_{x,t;-2}		\geq \int_{B_{-2}\left(z, \sqrt{\ep^{-1}H_n(t+2)}\right)}F(\cdot,-2) \,\mathrm{d}\nu_{x,t;-2}-C(n, Y),
	\end{align}
	where we used the fact $F(\cdot, -2) \ge -C(n, Y)$, which follows from Theorem \ref{thm:upper} (ii).
	
	By Proposition \ref{existenceHncenter}, we have
	\begin{align}\label{betterpoten1}
		\nu_{x,t;-2}\left(B_{-2}\big(z, \sqrt{\ep^{-1}H_n(t+2)}\big)\right)\geq 1-\ep.
	\end{align} 
From Proposition \ref{stabilitymetric} and \cite[Proposition 2.21(i)]{fang2025RFlimit}, we have $d_{-2}(x,z)\leq C(n,Y)$. Thus, for $y\in B_{-2}\left(z, \sqrt{\ep^{-1}H_n(t+2)}\right)$, it follows from Proposition \ref{stabilitypotential} that
	\begin{align}\label{betterpoten3}
		F(y,-2)\geq (1-100\ep)\bar F(\varphi_C^{-1}(y))-C(n,Y,\ep)\geq (1-200 \ep)\bar F(\varphi_C^{-1}(x))-C(n,Y,\ep).
	\end{align}
	Plugging \eqref{betterpoten1} and \eqref{betterpoten3} into \eqref{betterpoten2}, it follows that
	\begin{align*}
		\tilde{F}(x,t)\geq (1-300 \ep)\bar F(\varphi_C^{-1}(x))-C(n,Y,\ep).
	\end{align*}
	Therefore, we obtain the lower bound for $\tilde F$ in \eqref{stabpotentialbetter}. Since $\square \left( \tilde F+nt/2 \right)=0$, the higher order estimates follow from standard parabolic regularity and a similar argument as in Corollary \ref{highestimatepotential1}.
\end{proof}

\begin{defn}[Entropy radius]\label{defnrE}
	The \textbf{entropy radius} $\rE$ is defined as
	\begin{align*}
		\exp\left(-\frac{\mathbf{r}^2_E}{4}\right):=\mathcal W_{x_0^*}(1/2)-\mathcal W_{x_0^*}(2).
	\end{align*}\index{$\rE$}
	Here, we implicitly assume that $\mathcal W_{x_0^*}(1/2)-\mathcal W_{x_0^*}(2) <1$ so that $\rE$ is well-defined.
\end{defn}

We now fix a small constant $\sigma\in (0,1/10)$, an integer $l \in \mathbb N$ and define:
	\begin{align*}
\bar \delta_l := \frac{1}{2}\bar \delta \lc n, 10^{-100}n^{-1} l^{-1}\sigma \rc,
	\end{align*}\index{$\bar \delta_l$}
for $l \in \mathbb N$, where $\bar \delta$ is the function defined in Proposition \ref{stabilitymetric}.

In the following, we always assume $\r_{C, \bar \delta_l} \ge \bar \delta_l^{-2}$, and we identify $g$, $f$, $F$, and $\tilde F$ with their pullbacks $\varphi_C^* g$, $\varphi_C^*f$, $\varphi_C^* F$, and $\varphi_C^* \tilde F$, respectively, where $\varphi_C$ is the diffeomorphism corresponding to $\r_{C, \bar \delta_l}$ in Definition \ref{defnrCd}. Furthermore, we set
	\begin{align*}
\bar L:=\r_{C,\bar \delta_l}-\bar \delta_l^{-1}
	\end{align*}
so that the preceding estimates hold on the set $\{\bar b\leq \bar L\}\times [-2,-1/2]$ with constant $\ep$ replaced by $10^{-100}n^{-1} l^{-1}\sigma$.

\begin{lem} \label{lem:007}
	If $\bar L\leq (1-\sigma)\rE$, then on $\{\bar b\leq\bar L-1\}\times [-2,-1/2]$, we have
	\begin{align}\label{betterpoten9}
		\left[\na(\tilde F-F)\right]_{l}\leq C(n,Y,  \sigma, l) e^{\frac{\bar f}{2}-\frac{\bar L^2}{8(1-\sigma)}}.
	\end{align}
\end{lem}
\begin{proof}
By the definition of $\bar L$ and Lemma \ref{betterpotential}, we have
	\begin{align}\label{betterpoten4}
		\sup_{t\in [-2,-1/2]}\int_M \left|\nabla (F-\tilde F)\right|^2 
		\,\mathrm{d}\nu_t\leq Ce^{-\frac{\bar L^2}{4(1-\sigma)^2}}.
	\end{align}
	By Corollary \ref{highestimatepotential1} and Lemma \ref{estimateforbetterpotential}, we have
	\begin{align}\label{betterpoten5}
		\left[F-\tilde F\right]_i \leq C(n,Y,  \sigma, l) \exp \lc 10^{-90} i n^{-1} l^{-1} \sigma\bar L^2 \rc
	\end{align}
on $\{\bar b\leq\bar L-1\}\times [-2,-1/2]$, for any $0 \le i \le 10^{20}n\sigma^{-1} l$.
	
For any $(x,t)\in\{\bar b<\bar L-1\}\times [-2,-1/2]$, we set $B=B_t(x,1/2)$ and $2B=B_t(x,1)$. Applying the interpolation inequality \cite[Lemma B.1]{colding2015uniqueness} to $U=|\nabla (F-\tilde F)|^2e^{-f}$, we get
	\begin{align}\label{betterpoten6}
		\rVert U\rVert_{L^\infty(B)}&\leq C(n,k)\left(\rVert U\rVert_{L^1(2B)}+\rVert U\rVert_{L^1(2B)}^{a_{k,n}}\rVert \nabla^{k}U\rVert_{L^\infty(2B)}^{1-a_{k,n}}\right),
	\end{align}
	where the norms are unweighted. By \eqref{betterpoten4}, $\rVert U\rVert_{L^1(B)}\leq Ce^{-\frac{\bar L^2}{4(1-\sigma)^2}}$. Moreover, it follows from \eqref{betterpoten5} and Corollary \ref{highestimatepotential1} that for $k=10^{10}n\sigma^{-1}l$,
		\begin{align*}
\rVert \na^k U\rVert_{L^\infty(2B)}\leq C(n,Y,  \sigma, l) \exp\left(\bar  L^2-\inf_{2B}f\right).
	\end{align*}
Plugging these estimates into \eqref{betterpoten6}, it follows that
	\begin{align}\label{betterpoten7}
		\rVert U\rVert_{L^\infty(B)}\leq C(n,Y, \sigma, l) \lc \exp\left(-\frac{\bar L^2}{4(1-\sigma)^2}\right)+\exp\left(-\frac{a_{k,n}\bar L^2}{4(1-\sigma)^2}+(1-a_{k,n})(\bar L^2-\inf_{2B}f)\right) \rc.
	\end{align}
	
Since
		\begin{align*}
1-a_{k,n}=\frac{n}{k+n}=\frac{1}{10^{10}\sigma^{-1}l+1},
	\end{align*}	
it is clear by a direct calculation that
	\begin{align*}
		-\frac{a_{k,n}\bar L^2}{4(1-\sigma)^2}+(1-a_{k,n})\bar L^2\leq -\frac{\bar L^2}{4(1-3\sigma/2)}.
	\end{align*}
	Note that by Proposition \ref{stabilitypotential},
	\begin{align*}
\left|f(x,t)-\inf_{2B}f\right|+\left|f(x,t)-\bar f(x,t)\right| \le 10^{-90}n^{-1} l^{-1}\sigma \bar L^2.
	\end{align*}	
	
	Thus, \eqref{betterpoten7} yields:
	\begin{align}\label{betterpoten8}
		\left|\na (F-\tilde{F})\right|\leq C(n,Y,  \sigma, l) \exp\left(\frac{f}{2}-\frac{\bar L^2}{8(1-3\sigma/2)}\right)\leq C(n,Y,  \sigma, l) \exp\left(\frac{\bar f}{2}-\frac{\bar L^2}{8(1-1.1 \sigma)}\right).
	\end{align}
	Consequently,  \eqref{betterpoten9} follows from the interpolation based on \eqref{betterpoten5} and \eqref{betterpoten8}.	

\end{proof}

\begin{lem} \label{timesliceshrinkerquantity}
	If $\bar L \le (1-\sigma) \rE$, then for any $t\in [-3/2,-1/2]$ we have
	\begin{align*}
		\int_{\left\{\bar b \le\bar  L-1\right\}} \left|\frac{g}{2}-\na^2 \tilde F-\tau \Ric\right|^2\,\mathrm{d}\nu_t \le C(n)e^{-\frac{ \bar L^2}{4(1-\sigma)^2}}.
	\end{align*}
\end{lem}

\begin{proof}
	We set $\tilde \Phi:=g/2-\tau \Ric-\na^2 \tilde F$. It follows from a direct calculation that $\left(\partial_t -\Delta_L\right) \tilde \Phi=0$, where
	\begin{align*}
		\left(\Delta_L h\right)_{ij}:=\Delta h_{ij}+2\Rm_{ikjl}h_{kl}-\Ric_{ik}h_{jk}-\Ric_{jk}h_{ik}
	\end{align*}
	denotes the Lichnerowicz Laplacian with respect to $g$. Thus, we have
	\begin{align}\label{timesliceeq1}
		(\partial_t-\Delta)|\tilde \Phi|^2\le -2|\na \tilde \Phi|^2+C(n)|\Rm| |\tilde \Phi|^2.
	\end{align}
	
	Choose a cut-off function $\eta$ such that $\eta=1$ on $\{\bar b<\bar L-1\}$ and $\eta =0$ outside $\{\bar b\leq \bar L\}$. By Lemma \ref{stabilitymetric}, we may assume that $|\na \eta|+|\na^2 \eta| \le C(n)$. Moreover, $|\Rm|$ is bounded by $C(n)$ on the support of $\eta$.
	
According to \eqref{timesliceeq1}, we obtain
	\begin{align*}
\square \left(\eta^2|\tilde \Phi|^2\right)&\leq -2\eta^2|\na \tilde \Phi|^2+C(n) \eta^2|\tilde \Phi|^2-\left(\Delta\eta^2\right)|\tilde \Phi|^2-2\la \nabla \eta^2,\nabla|\tilde \Phi|^2\ra\\
		&\leq -2\eta^2|\na \tilde \Phi|^2+C(n) \chi_{\spt(\eta)}|\tilde \Phi|^2+2\eta^2|\nabla \tilde \Phi|^2+8|\nabla\eta|^2|\tilde \Phi|^2 \leq C(n)\chi_{\spt(\eta)}|\tilde \Phi|^2,
	\end{align*}
	where $\chi_{\spt(\eta)}$ denotes the characteristic function of $\spt(\eta)$, and	in the second inequality, we have used 
	\begin{align*}
		\left|2\la\nabla \eta^2,\nabla |\tilde \Phi|^2\ra\right| \le 8|\eta||\na \eta| |\tilde \Phi| |\na \tilde \Phi| \le 2\eta^2 |\na \tilde \Phi|^2+8|\na \eta|^2 |\tilde \Phi|^2.
	\end{align*}
Thus, we have
	\begin{align*}
		\frac{\mathrm{d}}{\mathrm{d}t}\int_M \eta^2|\tilde \Phi|^2 \,\mathrm{d}\nu_t=\int_M \square(\eta^2|\tilde \Phi|^2) \,\mathrm{d}\nu_t\le C(n)\int_{\spt(\eta)}|\tilde \Phi|^2 \,\mathrm{d}\nu_t.
	\end{align*}
For any $t\in [-3/2,-1/2]$, it follows that
	\begin{align}\label{betterpoten12}
		\int_{\left\{\bar b<\bar L-1\right\}}|\tilde \Phi|^2 \,\mathrm{d}\nu_t \leq& \min_{t\in [-2,-\frac{3}{2}]}\int_{\left\{\bar b< \bar L\right\}}|\tilde \Phi|^2 \,\mathrm{d}\nu_t+C(n)\int_{-2}^{-\frac{1}{2}}\int_{\left\{\bar b<\bar L\right\}}|\tilde \Phi|^2 \,\mathrm{d}\nu_t \mathrm{d}t \notag \\
		\leq & C(n) \int_{-2}^{-\frac{1}{2}}\int_{\left\{\bar b< \bar L\right\}}|\tilde \Phi|^2 \,\mathrm{d}\nu_t \mathrm{d}t.
	\end{align}
	
By Proposition \ref{betterpotential}, we get
	\begin{align*}
		\int_{-2}^{-\frac{1}{2}}\int_{\left\{\bar b< \bar L\right\}}|\tilde \Phi|^2 \,\mathrm{d}\nu_t \mathrm{d}t\leq  C e^{-\frac{\bar L^2}{4(1-\sigma)^2}},
	\end{align*}
	which, by \eqref{betterpoten12}, implies that for any $t\in [-3/2,-1/2]$, we have
	\begin{align*}
		\int_{\left\{\bar b< \bar L-1\right\}}|\tilde \Phi|^2 \,\mathrm{d}\nu_t\leq C(n)e^{-\frac{\bar L^2}{4(1-\sigma)^2}}.
	\end{align*}
	This finishes the proof.
\end{proof}

Combining Lemma \ref{lem:007} and Lemma \ref{timesliceshrinkerquantity}, we obtain the following time slice estimate:

\begin{cor} \label{cor:002}
	If $\bar L \le (1-\sigma) \rE$, then for any $t\in [-3/2,-1/2]$,
	\begin{align*}
		\int_{\left\{\bar b \le \bar L-1\right\}} \abs{\frac{g}{2}-\Ric-\nabla^2 F}^2\,\mathrm{d}\nu_t \le  C(n,Y,  \sigma, l) e^{-\frac{\bar L^2}{4}}.
	\end{align*}
\end{cor}

\begin{proof}
	By Proposition \ref{stabilitypotential} and Lemma \ref{lem:007}, we have
	\begin{align*}
		|\na^2(\tilde{F}- F)|^2e^{-f}\leq C(n,Y,  \sigma, l) \exp\left(10^{-90}n^{-1} l^{-1}\sigma \bar L^2 -\frac{\bar L^2}{4(1-\sigma)}\right) \leq C(n,Y,  \sigma, l) \exp\left(-\frac{\bar L^2}{4(1-0.5\sigma)}\right),
	\end{align*} 
	which implies for any $t\in [-3/2,-1/2]$,
	\begin{align}\label{timesliceine2}
		\int_{\left\{\bar b \le \bar L-1\right\}}|\na^2(\tilde{F}- F)|^2 \,\mathrm{d}\nu_t\leq C(n,Y,  \sigma, l) \bar L^n\exp\left(-\frac{\bar L^2}{4(1-0.5\sigma)}\right)\leq C(n,Y,  \sigma, l) e^{-\frac{\bar L^2}{4}}.
	\end{align}
	Consequently, the conclusion follows from Lemma \ref{timesliceshrinkerquantity} and \eqref{timesliceine2}.
\end{proof}

\begin{prop} \label{pointwiseshrinkeruantity}
	If $\bar L \le (1-\sigma) \rE$, then on $\{\bar b\leq \bar L-2\}\times [-3/2,-1/2]$,
	\begin{align*}
		\left[\frac{g}{2}-\na^2 \tilde F-\tau \Ric\right]_{l}+\left[\frac{g}{2}-\na^2 F-\tau \Ric\right]_{l} \le C(n,Y,  \sigma, l) e^{\frac{\bar f}{2}-\frac{\bar L^2}{8(1-\sigma)}}.
	\end{align*}
\end{prop}
\begin{proof}
	It suffices to prove the conclusion for $[g/2-\na^2 \tilde F-\tau \Ric]_{l}$ since the estimate for $[g/2-\na^2 F-\tau \Ric]_{l}$ follows from Lemma \ref{lem:007}. By Lemma \ref{timesliceshrinkerquantity}, we obtain
	\begin{align}\label{timeslice1}
		\int_{\left\{\bar b \le \bar L-1\right\}} \left|\frac{g}{2}-\na^2 \tilde F-\tau \Ric\right|^2\,\mathrm{d}\nu_t \le C(n)e^{-\frac{ \bar L^2}{4(1-\sigma)^2}}.
	\end{align}
	
	On the other hand, it follows from Proposition \ref{stabilitymetric} and Lemma \ref{estimateforbetterpotential} that on $\{\bar b\leq \bar L-2\}\times [-3/2,-1/2]$,
	\begin{align}\label{timeslice2}
		\left[\frac{g}{2}-\na^2 \tilde F-\tau \Ric\right]_{10^{10}n\sigma^{-1} l-2}\leq C(n,Y,  \sigma, l) e^{\bar L^2}.
	\end{align}
	Combining \eqref{timeslice1} and \eqref{timeslice2}, we can follow the interpolation argument as in the proof of Lemma \ref{lem:007} to conclude the proof.
\end{proof}

\subsection*{Comparison of radius functions II}

Next, we compare the radius functions $\rA$ and $\mathbf{r}_{C,\bar \delta_l}$ and their corresponding diffeomorphisms. Here, $\rA,\rCd$ denote the radius functions for $(M, g(-1), f(-1))$.

\begin{thm}\label{thm:ext1}
For any $\sigma\in (0, 1/10)$, $D>1$ and $l \in \mathbb N$, there exists a large constant $L'=L'(n,Y,  \sigma, l, D)>1$ satisfying the following property.

Let $\varphi_A$ be a diffeomorphism corresponding to $\rA$ in Definition \ref{defnradii}. If $\mathbf{r}_A \in [ L', (1-\sigma) \mathbf{r}_E]$, then there exists another diffeomorphism $\varphi$ from $\{\bar b \le \mathbf{r}_A+D \}$ onto a subset of $M$ such that $\varphi=\varphi_A$ on $\{\bar b \le \rA-2\bar \delta_l^{-1}\}$ and
	\begin{align*}
		\left[\varphi^* g(-1)-\bar g\right]_2\leq \bar \delta_l
	\end{align*}
on	$\{\bar b \le \mathbf{r}_A+D \}$. In particular, we have
	\begin{align*} 
		\mathbf{r}_{C,\bar \delta_l} \ge \mathbf{r}_A+D.
	\end{align*}
\end{thm}

\begin{proof}
By our definition of $\bar \delta_l$, the conclusions of Lemma \ref{lem:007}, Lemma \ref{timesliceshrinkerquantity}, and Proposition \ref{pointwiseshrinkeruantity} hold.

By the definitions of $\rA$ and $\r_{C, \bar \delta_l}$, we conclude that $\r_{C, \bar \delta_l} \ge \mathbf{r}_A -1$, provided that $L'$ is sufficiently large.

Suppose the conclusion fails. Then there exists a sequence of Ricci flows $\XX^i=\{M^n_i, (g_i(t))_{t\in [-10,0]}\}$ with entropy bounded below by $-Y$. Moreover, there exist base points $x_{0,i}^*=(x_{0,i},0)\in\XX^i$ such that $\rA^i \to +\infty$ and $\rA^i \le (1-\sigma)\mathbf{r}^i_E$. For each $i$, there exists a diffeomorphism $\varphi_{A,i}$ from $\{\bar b \le \rA^i\}$ onto a subset of $M$ corresponding to $\rA^i$ in Definition \ref{defnradii}. However, it is not possible to find a diffeomorphism from $\{\bar b \le \rA^i+D_0\} \to M_i$ for some constant $D_0>0$ that satisfies the required properties. Here and throughout, we use subscript or superscript $i$ to indicate the corresponding quantities associated with $\XX^i$. 

We define $D_1:=D_0+2\bar \delta_l^{-1}$ and $L_i:=\rA^i-2 \bar \delta_l^{-1}$. Then, we define the hypersurface $\Sigma_i:=\varphi_{A, i} \lc \{\bar b=L_i\} \rc$. Moreover, we choose a base point $q_i^*=(q_i, -1/2) \in \XX^i$ such that $q_i \in \Sigma_i$.

Define the time intervals $\III^{++}=[-10, 0]$, $\III^+=[-9.9, 0]$, $\III=[-9.8, 0]$ and $\III^-=(-9.8, 0]$. Then, by Theorem \ref{thm:intro1}, passing to a subsequence if necessary, we have
	\begin{align}\label{extenconv}
		\left(\XX_\III^i,d_i^*,q_i^*,\t_i\right)\xrightarrow[i\to\infty]{\quad \mathrm{pGH}\quad} \left(Z, d_Z, q,\t\right),
	\end{align}
	where $d_i^*$ denotes the spacetime distance induced by $g_i(t)$ (see Definition \ref{defnd*distance}). In particular, $ \left(Z, d_Z,q,\t\right)$ is a noncollapsed Ricci flow limit space over $\III$. In addition, the convergence \eqref{extenconv} is smooth on the regular part $\RR$ of $Z$ in the sense of Theorem \ref{thm:intro3}. Note that $\RR$ carries a structure of Ricci flow spacetime $(\RR, \t, \partial_\t, g^Z)$. Let $\phi_i$ denote the diffeomorphisms given in Theorem \ref{thm:intro3}.

It follows from Proposition \ref{stabilitymetric} that there exists a constant $c_0=c_0(n)>0$ such that
	\begin{align*} 
r^i_{\Rm}(x^*) \ge c_0
	\end{align*}
for any $x^* \in \Omega_i:=\varphi_{A, i} \lc \{\bar b \le L_i\} \rc \times [-2, -1/2]$, where $r^i_{\Rm}$ denotes the curvature radius of $\XX^i$ (see Definition \ref{defncurvatureradius}). Therefore, by the definition of $\rA$, the convergence \eqref{extenconv} is smooth on $\Omega_i$, and the domains $(\Omega_i, g_i(t), q_i^*)$, via the diffeomorphisms $\phi_i$, converge smoothly to a domain $\Omega \subset Z$ containing the point $q$. Moreover, it follows from our construction that the maps $\varphi_{A, i}^{-1} \circ \phi_i: \Omega \to \bar M$ converge smoothly to an isometry:
	\begin{align} \label{eq:keyine102a}
\psi: \Omega \longrightarrow (\R_-\times\R^{n-m-1}\times S^m) \times [-2, -1/2],
	\end{align}
where the half-cylinder $\R_-\times\R^{n-m-1}\times S^m$ is equipped with a family of standard metrics $g_c(t)$ for $t \in [-2, -1/2]$. For simplicity, we set $q(t) \in \Omega_t$ to be the flow line of $\partial_\t$ from $q$.

Next, we consider a metric flow $\XX^q$ associated with $q$ (cf. Theorem \ref{thm:iden}) and set $\RR'=\iota_q(\RR^q) \subset \RR$, where $\RR^q$ denotes the regular part of $\XX^q$. It is clear from our construction that $\Omega_{[-2, -1/2)} \subset \RR'$.

We consider an auxiliary function:
	\begin{align}\label{eq:auxi}
		F_i'(x,t):=2\frac{\tilde F_i(x,t)-\tilde F_i(q_i,-1)}{L_i}.
	\end{align}
	
	By a direct computation (see \eqref{eq:modiF}), we have
		\begin{align} \label{eq:keyine102b}
\square F_i'=-\frac{n}{L_i}.
	\end{align}
	
We aim to show that the limit of $F_i'$ exists on $\RR'_{(-3/2, -1+\ep)}$, where $\ep:=10^{-5} \sigma$.
	
First, it follows from Proposition \ref{betterpotential} and our assumption on $L_i$ that
	\begin{equation}\label{integralbd1}
		\int_{-2}^{-1/2}\int_{M_i} \left|\tau\Ric(g_i)+\nabla_i^2\tilde F_i-\frac{g_i}{2}\right|^2\,\mathrm{d}\nu_{x_{0,i}^*;t} \mathrm{d}t\leq Ce^{-\frac{L_i^2}{4(1-\sigma)^2}}.
	\end{equation}
	
	\begin{claim}\label{extension2}
		For any compact subset $S$ of $\RR'_{(-3/2, -1+\ep)}$, we have smooth convergence of $\phi_i^*\big(\na_i^2 \tilde{F}_i\big)$ as symmetric $2$-tensors. 
	\end{claim}
	
	For any $z \in S$, there exists a smooth spacetime curve $\gamma(s) \subset \RR'_{s}$ for $s \in [\t(z), -1+\ep]$, such that $\gamma(\t(z))=z$ and $\gamma(-1+\ep)=q(-1+\ep)$. Since $\RR'_{(-3/2, -1+\ep]}$ is connected, such a curve always exists. Moreover, because $S$ is compact, one can choose $\gamma$ so that
	\begin{align*} 
\int_{\t(z)}^{-1+\ep} |\scal_{g^Z}(\gamma(s))|+|\gamma'(s)|^2_{g^Z_s} \,\mathrm{d}s \le C(S),
	\end{align*}
	where $C(S)$ depends only on the compact $S$ and the geometry of $\RR'$, but is independent of the choice of $z$ in $S$. Since $\phi_i^{-1}((q_i, -1+\ep)) \to q(-1+\ep)$ and $r_{\Rm}^i((q_i, -1+\ep)) \ge c_0$, the curve $\gamma$ can be used to construct a corresponding spacetime curve $\gamma_i(s) \in \XX^i_s$ for $s \in [\t(z), -1+\ep]$ so that $\gamma_i(\t(z))=\phi_i(z)$, $\gamma_i(-1+\ep)=(q_i, -1+\ep)$, and
		\begin{align*}
\int_{\t(z)}^{-1+\ep} |\scal_i(\gamma_i(s))|+|\gamma_i'(s)|^2_{g_i(s)} \,\mathrm{d}s \le C(S),
	\end{align*}
	which is independent of $i$. In fact, $\gamma_i$ can be obtained by concatenating $\phi_i(\gamma)$ with a short curve in a neighborhood of $(q_i, -1+\ep)$.
	
By Perelman's differential Harnack inequality (see \cite[Corollary 9.4]{perelman2002entropy}), we obtain that if $(x, \t(z))=\phi_i(z)$, then
	\begin{align*}
\sqrt{|\t(z)|}f_i(x, \t(z)))\leq & \sqrt{1-\ep} f_i(q_i, -1+\ep)+\int_{\t(z)}^{-1+\ep}\sqrt{|s|}\left(|\scal_i(\gamma_i(s))|+|\gamma_i'(s)|_{g_i(t)}^2\right) \,\mathrm{d}s \\
 \le & \sqrt{1-\ep} f_i(q_i, -1+\ep)+C(S).
	\end{align*}
	
Thus, for any $(x,t)\in \phi_i(S)$, we have 
		\begin{align*}
f_i(x,t)\leq f_i(q_i,-1+\ep)+C(S)\leq \frac{(1+2\ep)L_i^2}{4}+C(S),
	\end{align*}
	where we used Proposition \ref{stabilitypotential} for the last inequality. It implies
	\begin{align*}
		\mathrm{d}\nu_{x_{0,i}^*;t}(x)= (4\pi \tau)^{-n/2}e^{-f_i(x,t)}\,\mathrm{d}V_{g_{i}(t)}\geq  c(n)\exp\left(-\frac{(1+2\ep)L_i^2}{4}-C(S)\right)\,\mathrm{d}V_{g_{i}(t)}.
	\end{align*}
	Combining this with \eqref{integralbd1}, we obtain
	\begin{align}\label{extension4}
		&\iint_{\phi_i(S)} \left|\tau\Ric(g_i)+\nabla_i^2\tilde F_i-\frac{g_i}{2}\right|^2 \,\mathrm{d}V_{g_{i}(t)} \mathrm{d}t\nonumber\\
		\leq& C\exp\left(-\frac{L_i^2}{4(1-\sigma)^2}+\frac{(1+2\ep)L_i^2}{4}+C(S)\right) \leq C(S) e^{-\sigma L_i^2/4},
	\end{align}
	where the last inequality holds since $\ep=10^{-5} \sigma$.

We set $\TT_i:=\tau\Ric_i+\nabla_i^2\tilde F_i-g_i/2$. Then we have $(\partial_t-\Delta_{i,L}) \TT_i=0$, where $\Delta_{i,L}$ denotes the Lichnerowicz Laplacian with respect to $g_i(t)$. Therefore, it follows from the standard parabolic regularity theory and the unweighted $L^2$-bound in \eqref{extension4} that $\phi_i^*(\TT_i)$ converge smoothly to $0$ on $\RR'_{(-3/2, -1+\ep)}$. In particular, $\phi_i^*(\nabla_i^2\tilde F_i)$ converge smoothly to a smooth symmetric $2$-tensor on $\RR'_{(-3/2, -1+\ep)}$. This completes the proof of the Claim \ref{extension2}.

Since $\nabla^2_i F_i'=2 \nabla_i^2 \tilde F_i/L_i$ and $L_i\to\infty$, it follows from Claim \ref{extension2} and \eqref{eq:keyine102b} that
	\begin{align}\label{limiteuqation1}
		\nabla^2_i F_i'\to 0 \quad \mathrm{and}\quad \square_i F_i'\to 0\quad \text{smoothly on}\quad\RR'_{(-3/2,-1+\ep)}.
	\end{align}
	\begin{claim}\label{extension5}
		The following equations hold:
		\begin{align*}
			F_i'(q_i,-1)=0 \quad \mathrm{and}\quad  \left|\na_i F_i'(q_i,-1)\right|= 1+\Psi(L_i^{-1}).
		\end{align*}
	\end{claim}
	The first equality follows from the definition. For the second, first note that for the standard potential function $\bar F$ on the cylinder, we have $2\left|\nabla \bar F(q_i)\right|/L_i=1+\Psi(L_i^{-1})$. By Definition \ref{defnradii} (A), we have
	\begin{align*}
		\left|\,|\na \bar F(q_i)|-|\na_i F_i(q_i,-1)|\,\right|\leq \Psi(L_i^{-1}).
	\end{align*}
	Thus, we also have 
	\begin{align}\label{closenpoten1}
		2\frac{\left|\nabla_i F_i(q_i,-1)\right|}{L_i}= 1+\Psi(L_i^{-1}).
	\end{align}
	By Lemma \ref{lem:007}, we have
	\begin{align}\label{closepoten2}
		\abs{\nabla_i (F_i -\tilde F_i)(q_i, -1)} \leq \Psi(L_i^{-1}).
	\end{align}
	
	 Combining \eqref{closenpoten1} and \eqref{closepoten2}, we conclude that 
	\begin{align*}
		2\frac{\left|\nabla_i \tilde F_i(q_i,-1)\right|}{L_i}=1+\Psi(L_i^{-1}),
	\end{align*}
	which gives the second inequality. This finishes the proof of the Claim \ref{extension5}.
	
	By \eqref{limiteuqation1} and Claim \ref{extension5}, we conclude that $\phi_i^*(F_i')$ converge smoothly to a smooth function $F_\infty$ on $\RR'_{(-3/2, -1+\ep)}$. It follows from \eqref{limiteuqation1} that
		\begin{align*}
\partial_\t F_\infty=\na^2 F_\infty=0
	\end{align*}
	on $\RR'_{(-3/2, -1+\ep)}$. Moreover, we have by Claim \ref{extension5}
			\begin{align*}
F_\infty(q(-1))=0 \quad \text{and} \quad |\na F_\infty|(q(-1))=1.
	\end{align*}
	Thus, it follows from \cite[Theorem 4.19]{fang2025RFlimit} (see also \cite[Theorem 15.50]{bamler2020structure}) that $\na F_\infty$ induces a splitting direction on $\RR'_{(-3/2, -1+\ep)}$. On the other hand, by our construction of $F_\infty$, we conclude that $\psi_*(\na F_\infty)$ agrees with the splitting direction of $\R_-$ in $(\R_-\times \R^{n-m-1}\times S^m) \times (-3/2, -1+\ep)$; see \eqref{eq:keyine102a}. Thus, there exists a domain $\Omega' \subset \RR'_{(-3/2, -1+\ep)}$ containing $\Omega_{(-3/2, -1+\ep)}$ so that $\psi$ can be extended to an isometry:
		\begin{align*}
\psi: \Omega' \longrightarrow (\R \times\R^{n-m-1}\times S^m) \times (-3/2, -1+\ep).
	\end{align*}
Since $\Omega'_t$ is complete for any $t \in (-3/2, -1+\ep)$, we conclude that 
			\begin{align*}
\Omega'=\RR'_{(-3/2, -1+\ep)}=\iota_q(\XX^q_{(-3/2, -1+\ep)}).
	\end{align*}	
	
	From the definition \eqref{eq:auxi} of $F_i'$, we conclude that before we reach the limit, 
			\begin{align*}
V_i:=\na_i F_i'=2\frac{\na_i \tilde F_i}{L_i}
	\end{align*}		
is almost a splitting vector field. It is important to note that although the definition of $F_i'$ depends on the choice of base point $q_i$, the definition of $V_i$ is independent of $q_i$.

Since the above argument holds uniformly for any choice of $q_i \subset \Sigma_i$, we conclude that on $B_{g_i(-1)}(\Sigma_i, 100 D_1)$, we have
		\begin{align}\label{eq:keyine105b}
	\abs{|V_i|-1} \le \Psi(i^{-1}) \quad \text{and} \quad  [\na_i V_i]_{100} \le \Psi(i^{-1}).
	\end{align} 

On the other hand, we set $\bar V$ to be the vector field $\partial_r$ on $\bar M$, where $r=|\vec{x}|$. We set the flow lines of $V_i$ and $\bar V$ by $\psi^i_s$ and $\bar \psi_s$, respectively, such that $\psi^i_0=\bar \psi_0=\mathrm{id}$. Using the flow lines, we obtain a diffeomorphism 
		\begin{align*}
	\kappa_i:\Sigma_i \times [0, 50D_1] \to M_i
	\end{align*} 
defined by $\kappa_i(w, s)=\psi^i_s(w)$. It is clear from \eqref{eq:keyine105b} that if $i$ is sufficiently large, $\kappa_i$ is $C^{10}$-close to an isometry up to an error $\Psi(i^{-1})$. Similarly, we define
		\begin{align*}
	\bar \kappa_i : \{\bar b=L_i\}  \times [0, 50D_1] \to \bar M
	\end{align*} 
defined by $\bar \kappa_i(w, s)=\bar \psi_s(w)$. Now, we can define $\varphi_i':\{L_i\leq\bar b\leq L_i+40 D_1\}\to M_i$ by
		\begin{align}\label{gluemetric0}
\varphi_i':=\kappa_i  \circ \hat \iota\circ (\bar \kappa_i)^{-1},
	\end{align} 
where $\hat \iota(w, s):=(\varphi_{A, i}(w), s)$ for any $w \in \{\bar b=L_i\}$ is $C^{5}$-close to an isometry up to an error of $\Psi(i^{-1})$. From our construction, it is clear that $\varphi'_i$ is $C^{5}$-close to an isometry up to an error $\Psi(i^{-1})$.
	
	Choose a cut-off function $\eta:\bar M\to [0,1]$ such that $\eta=1$ on $\{\bar b\leq L_i+1/2\}$ and $\spt(\eta)\subset \{\bar b\leq L_i+1\}$, then we can glue $\varphi_i'$ with $\varphi_{A, i}$ by defining 
	\begin{align}\label{gluemetric1}
		\varphi_i''(x)=\eta (x)\varphi_{A, i}(x)+\left(1-\eta(x)\right)\varphi_i'(x),\quad \forall x\in \{\bar b \leq L_i+20 D_1\}.
	\end{align}

	Since $\varphi_{A, i}$ is $C^{5}$-close to an isometry on $\{\bar b\leq L_i+10\}$, we can assume that $\varphi_{A, i}(x)$ and $\varphi_i'(x)$ lie in a chart of a fixed coordinate system for $x\in \{L_i+1/2\leq\bar b\leq L_i+1\}$. Therefore, \eqref{gluemetric1} is well-defined.  
	From the construction, we see that on $\{\bar b \leq L_i+20 D_1\}$,
	\begin{align*}
		\left[(\varphi_i'')^*g_{i}(-1)-\bar g(-1)\right]_2\leq \Psi(i^{-1}).
	\end{align*}
Since $L_i+20 D_1 > \rA^i+D_0$, we obtain a contradiction for sufficiently large $i$.
	
	In sum, the proof is complete.
\end{proof}

Next, for fixed $\sigma \in (0, 1/10)$ and $l \in \mathbb N$, we define
	\begin{align} \label{eq:dconstant}
\bar D_l=\bar D_l(n, \sigma)=10^4 \bar \delta_l^{-1} \gg 1.
	\end{align}\index{$\bar D_l$}
By Theorem \ref{thm:ext1}, we fix a diffeomorphism $\tilde \varphi_A$ from $\{\bar b \le \rA+\bar D_l\}$ onto a subset of $M$ such that
	\begin{align} \label{eq:fix01}
		\left[\tilde \varphi_A^* g(-1)-\bar g\right]_2\leq \bar \delta_l
	\end{align}
on	$\{\bar b \le \mathbf{r}_A+\bar D_l \}$, and
	\begin{align} \label{eq:fix02}
\left[\bar g-\tilde \varphi_A^* g(-1)\right]_5+\left[\bar f-\tilde \varphi_A^* f(-1)\right]_5\leq e^{\frac{\bar f}{4}-\frac{\rA^2}{16}}
	\end{align}
on	$\{\bar b \le \mathbf{r}_A-2\bar \delta_l^{-1} \}$.

\begin{cor} \label{almostshrinkerlarger}
There exists a large constant $L''=L''(n, Y,  \sigma, l)>1$ such that if $\mathbf{r}_A \in [ L'', (1-2\sigma) \mathbf{r}_E]$, then on $\{\bar b \le \rA+0.5\bar D_l\} \times \{-1\}$,
	\begin{align*}
		\left[\tilde \varphi_A^* \lc \frac{g}{2}-\na^2 f-\Ric \rc \right]_{l} \le \exp \lc\frac{\bar f}{2}-\frac{(\rA+0.5\bar D_l)^2}{8(1-\sigma)}\rc,
	\end{align*}
where the norm is with respect to $\bar g$. Moreover, on $\tilde\varphi_A\lc \left\{2\sqrt{\bar f(-1)} \le \rA-2.5 \bar \delta_l^{-1} \right\}\rc$,
	\begin{align} \label{eq:absestimate}
		\left[ \frac{g}{2}-\na^2 f-\Ric \right]_{l} \le \exp \lc\frac{f}{2}-\frac{(\rA+0.4\bar D_l)^2}{8(1-\sigma)}\rc,
	\end{align}
where the norm is with respect to $g(-1)$.
\end{cor}

\begin{proof}
It is straightforward from our parameter choices that $\rA + \bar D_l - \bar\delta_l^{-1} \le (1 - \sigma) \mathbf{r}_E$, provided that $L''$ is sufficiently large. Therefore, the first conclusion follows from Proposition \ref{pointwiseshrinkeruantity} by setting $\bar L = \rA + \bar D_l - \bar\delta_l^{-1}$. The second conclusion then follows from the first together with Proposition \ref{stabilitymetric} and \eqref{eq:fix02}.
\end{proof}

\subsection*{Comparison of radius functions III}

Next, we prove the following crucial extension result.

\begin{thm}\label{thm:ext2}
There exists a large constant $\hat L=\hat L(n, Y, \sigma)>1$ such that if $\mathbf{r}_A \in [ \hat L, (1-2\sigma) \mathbf{r}_E]$, then
	\begin{align*}
		\rBC \ge \mathbf{r}_A+\bar D_{100}/10.
	\end{align*}
	Here, both $\rA$ and $\rBC$ denote the radius functions for $(M, g(-1), f(-1))$ (cf. Definition \ref{defnradii}), and $\bar D_{100}$ is the constant defined in \eqref{eq:dconstant} with $l=100$.
\end{thm}

The proof of this theorem follows the same overall strategy as in \cite[Theorem 5.3]{li2023rigidity}, but with two key differences. 

First, while the weighted Riemannian manifold considered in \cite[Theorem 5.3]{li2023rigidity} is an exact Ricci shrinker, in our case $(M, g(-1), f(-1))$ is only an almost Ricci shrinker, as ensured by Corollary \ref{almostshrinkerlarger}.

Second, in \cite[Theorem 5.3]{li2023rigidity}, the authors used the pseudolocality theorem and the self-similarity of Ricci shrinkers to produce a neck region with bounded geometry. That approach does not directly apply in our setting. However, Theorem \ref{thm:ext1} allows us to obtain a similar neck region, albeit of smaller size compared to that in \cite[Theorem 5.3]{li2023rigidity}. 
Nevertheless, this smaller neck region suffices for our proof.

We now sketch the main ideas of the argument. As in \cite[Theorem 5.3]{li2023rigidity}, one can glue the end of $\bar M$ to $M$ through the neck region, resulting in a weighted Riemannian manifold $\left(\bar M,g',f'\right)$ that is almost a Ricci shrinker. Following the strategy of \cite{colding2021singularities}, we then use the first $n-m$ eigenfunctions of the weighted $\Delta_{f'}$ along with their growth estimates (see \cite[Proposition A.1]{li2023rigidity} and \cite[Theorem 4.1]{colding2021singularities}) to extend the almost $(n-m)$-splitting to a larger region. Finally, we construct a diffeomorphism from the cylinder to this extended region such that the estimates in the definition of $\rBC$ are satisfied.

\begin{proof}
We choose $\hat L \ge \max\{L'(n,Y,   \sigma,100,\bar D_{100}), L''(n, Y, \sigma, 100)\}$, where $L'$ and $L''$ are from Theorem \ref{thm:ext1} and Corollary \ref{almostshrinkerlarger}, respectively. In the proof, the constant $C$ may be different from line to line, depending on $n$, $Y$, and $\sigma$. We also use the notations $C_k$ to represent constants that depend on $n$, $Y$, $\sigma$ and $k$.

We set $(g, f)=(g(-1), f(-1))$, $L=\mathbf{r}_A-3\bar \delta_{100}^{-1}$, and $\bar D=\bar D_{100}$ for simplicity. Moreover, we consider the diffeomorphism $\tilde \varphi_A$ satisfying \eqref{eq:fix01} and \eqref{eq:fix02}. In particular, we have
	\begin{align}\label{keyext3}
		\left[\tilde \varphi_A^*g -\bar g\right]_5+\left[\tilde \varphi_A^* f -\bar f\right]_5\leq e^{\frac{\bar f}{4}-\frac{L^2}{16}}, \quad \mathrm{on}\quad \{\bar b\leq L\}.
	\end{align}

It follows from Proposition \ref{stabilitypotential} and Corollary \ref{highestimatepotential1} that
	\begin{align}\label{keyext1}
		(1-\bar \ep)\bar F-C\leq \tilde \varphi_A^*f \leq (1+\bar \ep)\bar F+C.
	\end{align}
	and
	\begin{align}\label{keyext1a}
		[\tilde \varphi_A^* f]_{100} \le C e^{\bar \ep \bar F},
	\end{align}
	on $\{\bar b\leq \rA+0.5\bar D\}$, where $\bar \ep=10^{-50} n^{-1} \sigma$.
	
Let $\phi^t, \bar \phi^t$ be the family of diffeomorphisms generated by $V:=\frac{\na b}{|\na b|^2}, \bar V:=\frac{\bar \na \bar b}{|\bar \na \bar b|^2}$, respectively, with $\phi^0=\bar \phi^0=\mathrm{id}$. Define a map $\psi_1:\{L-2\leq\bar b\leq L-1\}\subset \bar M\to \left\{L-2\leq b\leq L-1\right\}$ by:
\begin{equation*}
	\psi_1(x):=   \phi^{\bar b(x)-b(\tilde \varphi_A(x))}(\tilde \varphi_A(x))
	\end{equation*}
	Note that $\bar b=\psi_1^* b$. Moreover, it follows from \eqref{keyext3} that $\bar b$ and $\tilde \varphi_A^* b$ are $C^5$-close. Moreover, $V$ and $\bar V$ are $C^4$-close under $\mathrm{d}\varphi_A$ on $\{\bar b \le L-1\}$. Therefore, $\psi_1$ is a diffeomorphism such that
	\begin{align*}
		\left[\bar g-\psi_1^*g \right]_4 \leq Ce^{-CL}.
	\end{align*}
Next, we fix a cut-off function $\eta$ on $\R$ such that $\eta=1$ on $(-\infty,L-2]$ and $\eta=0$ on $[L-3/2,\infty)$. Then we define $\psi_2:\{\bar b\leq L-1\}\to \left\{b\leq L-1\right\}$ by
	\begin{equation*}
		\psi_2(x):= \eta\left(\bar b(x)\right)\tilde \varphi_A(x)+\left(1-\eta(\bar b(x))\right)\psi_1(x).
	\end{equation*}
	Since $\psi_1$ is close to $\tilde \varphi_A$ in the $C^5$ sense on $\{L-2\leq b\leq L-1\}$, we can always assume that $\psi_1(x)$ and $x$ lie in a chart of a fixed coordinate system so that the above definition is well-defined. Moreover, we conclude that $\psi_2$ is a diffeomorphism such that
		\begin{align*}
		\left[\bar g-\psi_2^*g \right]_4+\left[\bar f-\psi_2^*f\right]_4\leq e^{\frac{\bar f}{4}-\frac{L^2}{16}}
	\end{align*}
	on $\{b \le L-2\}$ and
			\begin{align*}
		\left[\bar g-\psi_2^*g \right]_4+\left[\bar f-\psi_2^*f\right]_4\leq C e^{-CL}
	\end{align*}
		on $\{L-2 \le b \le L-1\}$.
	
By Corollary \ref{almostshrinkerlarger}, we see that on $\{\bar b\leq L+0.5 \bar D\}$,
	\begin{align}\label{extensionineq1}
		\left[ \tilde \varphi_A^* \lc \frac{g}{2}-\na^2 f- \Ric\rc\right]_{100} \le \exp\left(\frac{\bar f}{2}-\frac{(L+0.5 \bar D)^2}{8(1-\sigma)}\right)\leq e^{-\frac{\sigma L^2}{10}}.
	\end{align}
	
Set $\Omega:=\tilde \varphi_A \lc \{\bar b \le L+0.5 \bar D\} \rc$ and $\Phi:=g/2-\na^2 f- \Ric$. By the weighted Bianchi identity (see Lemma \ref{weightedBianchilem}), we have
	\begin{align*}
		\frac{1}{2}\na \left(f-\scal-|\na f|^2\right)=\na \left(\Tr(\Phi)\right)-\Div_f(\Phi).
	\end{align*}
	From \eqref{keyext1a}, \eqref{extensionineq1}, and the fact that $\bar \ep \ll \sigma$, we obtain on $\Omega$,
		\begin{align*}
		\abs{\na \left(f-\scal-|\na f|^2\right) } \le e^{-\frac{\sigma L^2}{20}}.
	\end{align*}	
Thus, it follows from \eqref{keyext3} that there exists a constant $a_0$ with $|a_0| \le C(n)$ such that on $\Omega$,
		\begin{align}\label{shrinkerquantesti2}
		\abs{f-\scal-|\na f|^2-a_0} \le e^{-\frac{\sigma L^2}{30}}.
	\end{align}	

From this, we obtain
	\begin{align}\label{shrinkerquantesti2a}
		\left|\frac{1}{|\na b|^2}-1\right|=\left|\frac{f}{|\na f|^2}-1\right|=\left|\frac{f}{f-\scal+a_0+\Psi(L^{-1})}-1\right|\leq C(1+|f|)^{-1},
	\end{align}
where we used \eqref{keyext1} for the last inequality. 

Now, we extend $\psi_2$ to a diffeomorphism $\psi_3$ from $\{\bar b\leq L+0.4\bar D\}$ onto a subset of $\Omega$ by
	\begin{equation} \label{eq:definpsi3}
		\psi_3(x):= \phi^{\bar b(x)-L+1}\circ \psi_1\circ\bar\phi^{L-1-\bar b(x)}(x)
	\end{equation}
	for any $x \in \{L-1 \le \bar b \le L+0.4\bar D\}$. Notice that by \eqref{eq:fix01} and \eqref{shrinkerquantesti2a}, $\psi_3$ is well-defined. Moreover, by our construction, we know that $\bar b=\psi_3^* b$ on $\{L-1\leq \bar b\leq L+0.4\bar D\}$.
	
		\begin{claim}\label{keyclaim1}
		The above $\psi_3$ satisfies:
		\begin{align*}
			&\left[\bar g-\psi_3^*g\right]_4+\left[\bar f-\psi_3^* f\right]_4\leq Ce^{\frac{\bar f}{4}-\frac{L^2}{16}},\quad \mathrm{on}\quad \{\bar b\leq L-2\};\nonumber\\
			&\left[\bar g-\psi_3^*g\right]_4+\left[\bar f-\psi_3^* f\right]_4\leq C e^{-CL} \quad\mathrm{on}\quad \{L-2\leq \bar b\leq L-1\};\nonumber\\
			&\left[\bar g-\psi_3^*g\right]_2\leq CL^{-\frac 1 2},\quad \bar f=\psi_3^* f \quad \mathrm{on} \quad \{ L-1\leq\bar b\leq L+0.4 \bar D\}.
		\end{align*}
	\end{claim}
	
	The first two estimates follow from the construction directly, so we focus on the third estimate. 
	
	Set $\Sigma=\{b=L-1\} \cap \Omega$ and let $g_\Sigma$ denote the induced metric of $g$ on $\Sigma$. Define $\psi_4:\Sigma\times [L-1,L+0.4 \bar D]\to \{L-1 \leq b\leq L+0.4 \bar D\} \cap \Omega$ by 
	\begin{align*}
		\psi_4(w,r)=\phi^{r-L+1}(w)
	\end{align*}
	for any $w \in \Sigma$.	Then we claim that on $\Sigma\times [L-1,L+0.4 \bar D]$, 
	\begin{align}\label{extensionineq2}
		\rVert \psi_4^* g-(\mathrm{d} r^2+g_\Sigma)\rVert_{C^2}\leq C L^{-\frac 1 2}.
	\end{align}
	In fact, it follows from \eqref{shrinkerquantesti2a} that
	\begin{align*}
		\left|\psi_4^*g(\partial_r,\partial_r)-1\right|\leq CL^{-2}.
	\end{align*}
	This implies that $\left|\psi_4^*g-(\mathrm{d}r^2+g_r)\right|\leq CL^{-2}$,
	where we set $g_r$ to be the induced metric on $\Sigma\times\{r\}$ of $\psi_4^* g$. On the other hand, we compute
	\begin{align*}
		\left|\mathcal L_Vg\right|=\left|\mathcal L_{\frac{\sqrt{f}}{|\na f|^2}\na f}g \right|\leq \frac{\sqrt{f}}{|\na f|^2} \left|\mathcal L_{\na f}g\right|+C|\na f|\left|\na \frac{\sqrt{f}}{|\na f|^2}\right|.
	\end{align*}
	Since $\mathcal L_{\na f} g=2\na^2 f$, which is bounded by $C(n)$ from \eqref{extensionineq1}, it follows from \eqref{shrinkerquantesti2} that on $\Omega$,
		\begin{align*}
		\left|\mathcal L_Vg\right| \le C L^{-1}.
	\end{align*}
	Thus, we obtain $|\partial_r g_r| \le C L^{-1}$, which implies
	\begin{align*}
		|g_r-g_\Sigma|\leq C L^{-1}.
	\end{align*}
Consequently, we have on $\Sigma \times [L-1, L+0.45 \bar D]$,
	\begin{align*}
		\rVert \psi_4^* g-(\mathrm{d} r^2+g_\Sigma)\rVert_{C^0}\leq C L^{-1},
	\end{align*}
	which implies the $C^0$-estimate in \eqref{extensionineq2}. The higher-order estimates follow from the standard interpolation, since by the two-sided pseudolocality theorem (see Theorem \ref{thm:twoside}) and Shi’s local estimates, we have the higher-order curvature estimates for both $g$ and $g_{\Sigma}$.
	
Since $\|\psi_3^* g-\bar g\|_{C^2} \le C e^{-CL}$ on $\{\bar b  \le L-1\}$, it follows from the definition of $\psi_3$ (see \eqref{eq:definpsi3}) and \eqref{extensionineq2} that
		\begin{align*}
\|\psi_3^* g-\bar g\|_{C^2} \le C e^{-CL}+C L^{-\frac 1 2} \le C L^{-\frac 1 2}
	\end{align*}
	on $\{L-1 \le \bar b  \le L+0.4 \bar D\}$. This completes the proof of Claim \ref{keyclaim1}.

	Choose a cut-off function $\eta_1: \R\to\R$ such that $\eta_1=1$ on $(-\infty, L+0.3 \bar D]$ and $\eta_1=0$ on $[L+0.4 \bar D,\infty)$. Then we construct a complete weighted Riemannian manifold $(\bar M,g',f')$ by:
	\begin{equation*}
		g':=\eta_1(\bar b)\psi_3^*g+(1-\eta_1(\bar b))\bar g,\quad f':=\eta_1(\bar b)\psi_3^*f+\left(1-\eta_1(\bar b)\right)\bar f.
	\end{equation*}
	For the rest of the proof, we set $b':=2\sqrt{|f'|}$, and the underlying metric is $g'$ by default.
	
	
	\begin{claim}\label{conditionappen1}
		Define two functions $R_1:=n/2-\Delta f'$ and $R_2:=f'-|\na f'|^2$. On $(\bar M,g',f')$, the following estimates hold: 
		\begin{equation*}
			|R_1|+|R_2|\leq C L^{-\frac 1 2} (b')^2+C.
		\end{equation*}
	\end{claim}
	Note that on the standard cylinder, we have $n/2-\bar \Delta \bar f=\bar f-|\bar \na \bar f|^2=m/2$. Thus, the conclusion holds for $\{b' \ge L +0.4 \bar D\}$.
	
On $\{b' < L +0.4 \bar D\}$, it follows from \eqref{na2festimate2}, Claim \ref{keyclaim1} and our construction that
		\begin{equation*}
|\na f'| \le C b'+C \quad \text{and} \quad |\na^2 f'| \le C L^{\frac 1 2}.
		\end{equation*}
Thus, by Claim \ref{keyclaim1} and a standard variational argument, we obtain
		\begin{equation*}
\abs{R_1-m/2} \le C e^{-CL} b'+CL^{-\frac 1 2}b'+C \le C L^{-\frac 1 2} b'+C
		\end{equation*}
and
		\begin{equation*}
\abs{R_2-m/2} \le C e^{-CL} (b')^2+CL^{-\frac 1 2}(b')^2+C \le C L^{-\frac 1 2} (b')^2+C.
		\end{equation*}
This completes the proof of Claim \ref{conditionappen1}.

Claim \ref{conditionappen1} allows us to apply \cite[Proposition A.1]{li2023rigidity} to get slow exponential growth of eigentensors. Let $0<\mu_1\leq\mu_2\cdots$ be the eigenvalue of $\De_{f'}$ with respect to $\mathrm{d}V_{f'}$ counted with multiplicities. And we choose $(n-m)$-orthonormal eigenfunctions $v_i$ with $\De_{f'} v_i+\mu_iv_i=0$ for $1 \le i \le n-m$ with 
	\begin{align}\label{orthoeigen1}
		\int_{\bar M}v_iv_j\,\mathrm{d}V_{f'}=\delta_{ij},\quad \int_{\bar M}|\na v_i|^2\,\mathrm{d}V_{f'}=\mu_i.
	\end{align}
	
	\begin{claim}\label{claimsexa0}
		For $i \in \{1,\ldots,n-m\}$, 
		\begin{equation}\label{extesti1}
			\rVert \na^2 v_i\rVert_{L^2}^2+\left|\mu_i-1/2\right|\leq Ce^{-\frac{L^2}{16}}.
		\end{equation}
	\end{claim}
	
	\eqref{extesti1} can be proved in a similar way as \cite[Lemma 7.4]{colding2021singularities} and \cite[Theorem 5.3]{li2023rigidity}. Choose $\bar c$ such that $\int_M \bar c^2dV_{\bar f}=1$. Let $\bar \omega_0=\bar c$, $\bar \omega_i=\frac{\bar c}{\sqrt 2}x_i$, for $1 \le i \le n-m$. By direct computation, we have
	\begin{align}\label{orthocondition1}
		\left|\delta_{ij}-\int_{\bar b<L}\bar\omega_i\bar\omega_j\,\mathrm{d}V_{\bar f}\right|&\leq CL^{n-m}e^{-\frac{L^2}{4}}, \quad \forall i,j\in \{0,1,\ldots,n-m\},
\nonumber\\
		\left|\frac{1}{2}\delta_{ij}-\int_{\bar b<L}\la \na\bar\omega_i,\na\bar\omega_j\ra \,\mathrm{d}V_{\bar f}\right|&\leq CL^{n-m} e^{-\frac{L^2}{4}}, \quad  \forall i,j\in \{1,\ldots,n-m\}.
	\end{align}
	Choose a cut-off function $\eta_2$ which equals to 1 on $\left\{b'<L-1\right\}$ and equals to 0 outside $\left\{b'<L\right\}$. Set $\omega_i=\eta_2 \bar \omega_i$ and define $a_{ij}$ and $b_{ij}$ by
	$$a_{ij}:=\int_{\bar M} \omega_i\omega_j\,\mathrm{d}V_{f'},\quad b_{ij}:=\int_{\bar M} \la\na \omega_i,\na\omega_j\ra \,\mathrm{d}V_{f'}.$$
	
	By \eqref{keyext3} and \eqref{orthocondition1} (see also \cite[Section 7]{colding2021singularities}), we obtain
	\begin{align*}
		\sup_{i,j\geq 0}\left|a_{ij}-\delta_{ij}\right|+\sup_{i,j\geq 1}\left|2b_{ij}-\delta_{ij}\right|+\left|b_{00}\right|\leq Ce^{-\frac{L^2}{16}}.
	\end{align*}
From the min-max principle, we have
	\begin{align}\label{eignvalueestimate1}
		\sum_{i=1}^{n-m} \mu_i\leq \sum_{i,j}a^{ij}b_{ij}\leq \frac{n-m}{2}+Ce^{-\frac{L^2}{16}},
	\end{align}
where $(a^{ij})$ is the inverse matrix of $(a_{ij})$.

	Set $h_i=\mu_i g'-\Ric(g')-\na^2 f'$ and $\Phi':=\mathbf{\Phi}(g', f')=g'/2-\Ric(g')-\na^2 f'$. By our construction, $h_i=(\mu_i-1/2)g'$ on $S_1:=\left\{b'> L+0.4 \bar D\right\}$, and $|h_i|\leq C L^{\frac 1 2}$ on $S_2:=\left\{L-2\leq b'\leq L+0.4 \bar D\right\}$. On $S_3:=\left\{b'< L-2\right\}$, it holds that $h_i=\Phi'+(\mu_i-1/2)g'$, and by Corollary \ref{almostshrinkerlarger} we see that 
	\begin{align}\label{keyext8}
		\int_{S_3}\left|\Phi'\right|^2\,\mathrm{d}V_{f'}\leq Ce^{-\frac{L^2}{4}}.
	\end{align}
	 
Recall the following Bochner formula:
	\begin{align}\label{eignvalueestimate1c}
\frac{1}{2}\De_{f'}|\na v_i|^2=|\na^2 v_i|^2-h_i(\na v_i,\na v_i).
	\end{align}

Since $|\na v_i|^2$ is integrable by \eqref{orthoeigen1}, we can obtain from \eqref{eignvalueestimate1c} that $|\na^2 v_i|^2$ is integrable. In fact, for $A>100$, we choose a cutoff function $\eta_A:\bar M\to [0,1]$ such that $\eta_A$ is supported in $\left\{\sqrt{f'}\leq A+1\right\}$ and equals to $1$ on $\left\{\sqrt{f'}\leq A\right\}$. Moreover, it holds that $|\na \eta_A|^2\leq 100\eta_A$. Applying integration by parts to \eqref{eignvalueestimate1c}, we obtain
\begin{align*}
	2\int_{\bar M}\eta_A|\na^2 v_i|^2\, \mathrm{d}V_{f'}&=2\int_{\bar M}\eta_Ah_i(\na v_i,\na v_i)\,\mathrm{d}V_{f'}-\int_{\bar M}\la \na \eta_A,\na |\na v_i|^2\ra\,\mathrm{d}V_{f'}\\
	&\leq C L^{\frac 1 2} \mu_i+2\int_{\bar M}|\na\eta_A||\na v_i||\na^2 v_i|\,\mathrm{d}V_{f'}\\
	&\leq CL^{\frac 1 2} \mu_i+\int_{\bar M}\eta_A^{-1}|\na \eta_A|^2|\na v_i|^2\,\mathrm{d}V_{f'}+\int_{\bar M}\eta_A|\na^2 v_i|^2\,\mathrm{d}V_{f'}\\
	&\leq (CL^{\frac 1 2}+100)\mu_i+\int_{\bar M}\eta_A|\na^2 v_i|^2\,\mathrm{d}V_{f'},
\end{align*} 
which, by sending $A\to\infty$, implies that $|\na^2 v_i|^2$ is integrable.

Thus, we can use integration by parts to \eqref{eignvalueestimate1c} to obtain
	\begin{align}\label{bochnerintegration1}
		0\leq&\int_{\bar M}|\na^2 v_i|^2\,\mathrm{d}V_{f'}=\int_{\bar M} h_i(\na v_i,\na v_i)\,\mathrm{d}V_{f'}\nonumber\\
		=&(\mu_i-1/2)\int_{S_1\cup S_3}|\na v_i|^2\,\mathrm{d}V_{f'}+\int_{S_2}h_i(\na v_i,\na v_i) \,\mathrm{d}V_{f'}+\int_{S_3} \Phi'(\na v_i,\na v_i)\,\mathrm{d}V_{f'} \notag\\
		= &(\mu_i-1/2) \mu_i-(\mu_i-1/2)\int_{S_2}|\na v_i|^2 \,\mathrm{d}V_{f'}+\int_{S_2}h_i(\na v_i,\na v_i) \,\mathrm{d}V_{f'}+\int_{S_3} \Phi'(\na v_i,\na v_i)\,\mathrm{d}V_{f'} \notag \\
		\le & (\mu_i-1/2) \mu_i+CL^{\frac 1 2} \int_{S_2}|\na v_i|^2 \,\mathrm{d}V_{f'}+\int_{S_3} \abs{\Phi'(\na v_i,\na v_i)}\,\mathrm{d}V_{f'}.
	\end{align}
	Note that on $\bar M$, we have $\De_{f'}\na v_i=-h_i(\na v_i)$, thus,
	$$\la \De_{f'} \na v_i,\na v_i\ra=-h_i(\na v_i,\na v_i)\geq -C L^{\frac 1 2} |\na v_i|^2$$
	and $\De_{f'} \na v_i\in L^2(\,\mathrm{d}V_{f'})$. Now we can use Claim \ref{conditionappen1} and apply \cite[Proposition A.1]{li2023rigidity} to conclude that for any $r\in [C, L+0.4 \bar D]$,
	\begin{align*}
		I_{\na v_i}(r):=r^{1-n} \int_{b'=r}|\na v_i|^2|\na b'|\,\mathrm{d}V_{g'}\leq C L^{CL^{\frac 1 2}}e^{C L^{\frac 3 2}}\mu_i \le C e^{C L^{\frac 3 2}}\mu_i.
	\end{align*}
		
	Since $|\na b'|$ is almost 1, we obtain
	\begin{align}\label{keyext5}
		L^{\frac 1 2}\int_{S_2}|\na v_i|^2 \,\mathrm{d}V_{f'}\leq C \exp \lc C L^{\frac 3 2}-\frac{L^2}{4} \rc \mu_i.
	\end{align}
By the standard elliptic estimate on $B(x,L^{-1})$, we obtain, for any $l \in \mathbb N$,
	\begin{align}\label{keyext7}
		\sup_{B(x,L^{-1}/2)}\rVert\na  v_i\rVert^2_{C^{l,1/2}}&\leq C_lL^{2l+n+1}\int_{B(x,L^{-1})}|\na v_i|^2 \,\mathrm{d}V_{g'}\nonumber\\
		&\leq C_lL^{2l+2n}\sup_{r\in [b'(x)-L^{-1},b'(x)+L^{-1}]}I_{\na v_i}(r) \leq C_lL^{2l+2n} e^{C L^{\frac 3 2}}\mu_i.
	\end{align}
	Combining \eqref{keyext8} with \eqref{keyext7}, we get
	\begin{align}\label{keyext9}
		\int_{S_3}\abs{\Phi'(\na v_i,\na v_i)} \,\mathrm{d}V_{f'}\leq C \exp \lc C L^{\frac 3 2}-\frac{L^2}{8} \rc \mu_i.
	\end{align}

	Plugging \eqref{keyext5} and \eqref{keyext9} into \eqref{bochnerintegration1}, we obtain
	\begin{align*}
\mu_i-1/2 \ge -C \exp \lc C L^{\frac 3 2}-\frac{L^2}{8} \rc,
	\end{align*}
	which, when combined with \eqref{eignvalueestimate1}, implies
	\begin{align*}
		\left|\mu_i-1/2\right|\leq Ce^{-\frac{L^2}{16}}.
	\end{align*}
Using \eqref{bochnerintegration1}, we can conclude that
	\begin{align*}
		\int_{\bar M}|\na^2 v_i|^2\,\mathrm{d}V_{f'}\leq Ce^{-\frac{L^2}{16}}.
	\end{align*}
	Thus, the claim is proved.
	
	\begin{claim}\label{claimsexa}
		There exist $n-m$ functions $u_i$ such that $\int u_i \,\mathrm{d}V_{f'}=0$ and on $\{b\leq L+0.4 \bar D-1\}$,
		\begin{equation}\label{extensionineq3}
			\left|\delta_{ij}-\la \na u_i,\na u_j\ra \right|+\rVert \na^2 u_i\rVert_{C^3}+ \left|2\la \na u_i,\na f'\ra-u_i\right|\leq C \exp \lc C L^{\frac 3 2}-\frac{L^2}{32} \rc.
		\end{equation}
		Moreover, $\rVert u_i\rVert_{C^0}\leq C e^{CL^{\frac 3 2}}$ and for any $k\in [4, 10^{12}n\sigma^{-1}]$, $\rVert \na^2 u_i\rVert_{C^k}\leq e^{-\frac{L^2}{33}}$ on $\{b'\leq L+0.4 \bar D-1\}$. 
	\end{claim}
	
	The proofs of \eqref{extensionineq3} and the estimate for $\rVert u_i\rVert_{C^0}$ are identical to the proof of \cite[Claim 4, page 56]{li2023rigidity}. Notice that the $\{u_1, \cdots, u_{n-m}\}$ can be regarded as a Gram-Schmidt orthogonalization of $\{v_1, \cdots, v_{n-m}\}$. Moreover, the constant $\theta_2$ in \cite[Claim 4, page 56]{li2023rigidity} can be replaced with $CL^{-\frac 1 2}$ in our case, thanks to Claim \ref{conditionappen1}. For the higher estimates for $\rVert \na^2 u_i\rVert_{C^k}$, by the same argument as in \cite[Equation (5.62)]{li2023rigidity}, we have on $\{b'\leq L+0.4 \bar D-1\}$,
		\begin{align*}
\rVert \na^2 u_i\rVert_{C^k} \le C_k L^{m_k} \exp \lc C L^{\frac 3 2}-\frac{L^2}{32} \rc
	\end{align*}
	for $k \in [4, 10^{12}n\sigma^{-1}]$. Thus, if $\hat L$ is sufficiently large, we have for $k \in [4, 10^{12}n\sigma^{-1}]$,
			\begin{align*}
\rVert \na^2 u_i\rVert_{C^k} \le e^{-\frac{L^2}{33}}
	\end{align*}
on $\{b'\leq L+0.4 \bar D-1\}$. This completes the proof of Claim \ref{claimsexa}.

As in \cite{li2023rigidity}, we define
	\begin{align*}
f_0:=\frac{m}{2}+\frac{1}{4}\sum_{i=1}^{n-m} u_i^2.
	\end{align*}

Following the same proof as \cite[Claim 5, page 58]{li2023rigidity}, we obtain
	\begin{claim} \label{claimsexa1}
		On $\{b'\leq L+0.4 \bar D-2\}$,
		\begin{equation*}
			\left|f'-f_0 \right|\leq C \exp \lc C L^{\frac 3 2}-\frac{L^2}{32} \rc,
		\end{equation*}
		and for any $k \in [1, 10^{12}n\sigma^{-1}]$, $\rVert f'-f_0\rVert_{C^k}\leq e^{-\frac{L^2}{33}}$.
	\end{claim}

Now, we can follow the same argument as in \cite[Theorem 5.3]{li2023rigidity} to conclude the proof. For readers' convenience, we sketch the rest of the proof as follows.

We set $N_0:=\{f_0=m/2\}$ and let $\kappa_i^t$ be the family of diffeomorphisms generated by $\na u_i$. We define the map $\psi_4: \R^{n-m} \times N_0  \to \bar M$ by
	\begin{align*}
\psi_4 (\vec{x}, w)=\kappa_1^{x_1}\circ \cdots \circ \kappa_{n-m}^{x_{n-m}}(w),
	\end{align*}
where $\vec{x}=(x_1, \cdots, x_{n-m}) \in \R^{n-m}$ and $w \in N_0$. Then it follows from Claim \ref{claimsexa} that $\psi_4$ is a diffeomorphism from $B_{g_E}(0, L+0.3 \bar D) \times N_0$ onto its image and it satisfies
	\begin{align*}
\|\psi_4^* g'-g_E\times g_{N_0}\|_{C^3} \le C \exp \lc C L^{\frac 3 2}-\frac{L^2}{32} \rc,
	\end{align*}
where $g_{N_0}$ is the restriction of $g'$ on $N_0$, and for any $k \in [4, 10^{11}n\sigma^{-1}]$,
	\begin{align*}
\|\psi_4^* g'-g_E\times g_{N_0}\|_{C^k} \le e^{-\frac{L^2}{33}}.
	\end{align*}

In addition, we can define a diffeomorphism $\phi: S^m \to N_0$ (cf. \cite[Equations (5.82) and (5.83)]{li2023rigidity}) so that
	\begin{align*}
\|\phi^* g_{N_0}-g_{S^m}\|_{C^{2, 1/2}} \le C \exp \lc C L^{\frac 3 2}-\frac{L^2}{32} \rc
	\end{align*}
and for any $k \in [3, 10^{11}n\sigma^{-1}]$, 
	\begin{align*}
\|\phi^* g_{N_0}-g_{S^m}\|_{C^{k}} \le e^{-\frac{L^2}{33}}.
	\end{align*}

Finally, we define a map $\psi_5: B_{g_E}(0, L+0.3 \bar D) \times S^m \to M$ by
	\begin{align*}
\psi_5(\vec{x}, w):= \psi_3 \circ \psi_4 (\vec{x}, \phi(w))
	\end{align*}
where $\vec{x} \in B_{g_E}(0, L+0.3 \bar D) \subset \R^{n-m}$ and $w \in S^m$. By our construction, it is not hard to show that
	\begin{align} \label{eq:wang1}
\abs{\psi_5^* g-\bar g} \le C \exp \lc C L^{\frac 3 2}-\frac{L^2}{32} \rc
	\end{align}
and for any $k \in [1, 10^{11}n\sigma^{-1}]$,
	\begin{align}\label{eq:wang2}
\|\psi_5^* g-\bar g \|_{C^k} \le e^{-\frac{L^2}{33}}.
	\end{align}
Moreover, we have
	\begin{align}\label{eq:wang3}
\abs{\psi_5^* f-\bar f} \le C \exp \lc C L^{\frac 3 2}-\frac{L^2}{32} \rc
	\end{align}
and for any $k \in [1, 10^{11}n\sigma^{-1}]$,
	\begin{align}\label{eq:wang4}
\|\psi_5^* f-\bar f \|_{C^k} \le e^{-\frac{L^2}{33}}.
	\end{align}
	
On the other hand, we claim that
	\begin{align}\label{eq:wang4xa}
\int_{\{\bar b\leq L \}}\left|\psi_5^*\Phi\right|^2_{\bar g}\,\mathrm{d}V_{\bar f}\leq e^{-\frac{(\rA+0.2\bar D)^2}{4-3\sigma}}.
	\end{align}
Indeed, by our definition of $L$, \eqref{eq:wang2}, \eqref{eq:wang4} and \eqref{eq:absestimate} in Corollary \ref{almostshrinkerlarger}, we have
	\begin{align*}
\left|\psi_5^*\Phi\right|^2_{\bar g} \le C\left|\psi_5^*\Phi\right|^2_{\psi_5^* g} \le C e^{\frac{\psi_5^* f}{2}-\frac{(\rA+0.5\bar D)^2}{8(1-\sigma)}} \le C e^{\frac{\bar f}{2}-\frac{(\rA+0.5\bar D)^2}{8(1-\sigma)}}
	\end{align*}
on $\{\bar b\leq L \}$. Thus, we can estimate
	\begin{align*}
\int_{\{\bar b\leq L \}}\left|\psi_5^*\Phi\right|^2_{\bar g}\,\mathrm{d}V_{\bar f} \le C\abs{\{\bar b\leq L \}} e^{-\frac{(\rA+0.5\bar D)^2}{4(1-\sigma)}} \le e^{-\frac{(\rA+0.5\bar D)^2}{4-3\sigma}},
	\end{align*}
provided that $\hat L$ is sufficiently large. 

Moreover, we obtain from \eqref{eq:wang2} and \eqref{eq:wang4} that
	\begin{align}\label{eq:wang4xb}
\int_{\{L \le \bar b\leq L+0.2 \bar D \}}\left|\psi_5^*\Phi\right|^2_{\bar g}\,\mathrm{d}V_{\bar f}\leq C e^{-\frac{2 L^2}{33}-\frac{L^2}{4}}.
	\end{align}
From \eqref{eq:wang4xa} and \eqref{eq:wang4xb}, we conclude that \eqref{equ:integralphi} holds for $\psi_5$ with radius $L+0.2 \bar D$.

Combining this with \eqref{eq:wang1}, \eqref{eq:wang2}, \eqref{eq:wang3} and \eqref{eq:wang4}, we conclude that $\rBC \ge L+0.2 \bar D$ by using the diffeomorphism $\psi_5$, if $\hat L$ is sufficiently large. By our definition of $L$ and $\bar D$, it follows that $\rBC \ge \rA+0.1 \bar D$.

In sum, the proof of the theorem is complete.
\end{proof}

By using Proposition \ref{prop:con} and Theorem \ref{thm:ext2} iteratively, we obtain the following result. Note that by Theorem \ref{thm:upper} (i), the assumption \eqref{normalization22} holds for a constant $C_V=C_V(n,Y)$.

\begin{thm} \label{thm:ra}
	Let $\XX=\{M^n,(g(t))_{t\in I}\}$ be a closed Ricci flow with entropy bounded below by $-Y$. Assume $x_0^*=(x_0,0)\in \XX$ and $[-10,0]\subset I$. For any small $\sigma\in (0,1/10)$, there exists a constant $L=L(n,Y,\sigma)$ such that if the weighted Riemannian manifold $\left(M,g(-1),f_{x_0^*}(-1)\right)$ satisfies $\rA\geq L$, then
	\begin{align*}
		\min\{\mathbf{r}_A, \rBC\} \ge (1-3\sigma) \mathbf{r}_E.
	\end{align*}
\end{thm}

\begin{proof}
We choose $L :=100\max\{L_1(n, C_V, \sigma)=L_1(n, Y,\sigma), \hat L(n, Y, \sigma), \sigma^{-1}\}$, where $L_1$ and $\hat L$ are from Proposition \ref{prop:con} and Theorem \ref{thm:ext2}, respectively. 

By Theorem \ref{thm:ext2}, if $L\leq \rA\leq (1-2\sigma)\rE$, then $\rBC\geq\rA+\bar D_{100}/10$. 
By Proposition \ref{prop:con}, we obtain $\rA \ge \rBC-2$. If $(1-2\sigma) \mathbf{r}_E < L$, the conclusion obvious holds. Otherwise, we can use Proposition \ref{prop:con} and Theorem \ref{thm:ext2} iteratively to conclude that
	\begin{align*}
		\min\{\mathbf{r}_A, \rBC\} \ge (1-2\sigma) \mathbf{r}_E-2 \ge (1-3\sigma) \mathbf{r}_E.
	\end{align*}
\end{proof}

By the same argument, we also have:
	\begin{thm}\label{thm:rasing}
		Let $(Z, d_Z, \t)$ be the completion of a closed Ricci flow $\XX=\{M^n, (g(t))_{t \in [-10,0)}\}$ with entropy bounded below by $-Y$. Fix $z \in Z_0$. For any small $\sigma\in (0,1/10)$, there exists a constant $L=L(n,Y,\sigma)$ such that if the weighted Riemannian manifold $\left(M,g(-1),f_z(-1)\right)$ satisfies $\rA\geq L$, then
		\begin{align*}
			\min\{\mathbf{r}_A, \rBC\} \ge (1-3\sigma) \mathbf{r}_E.
		\end{align*}
	\end{thm}

\subsection*{Lojasiewicz inequality and summability of entropy gaps}

Combining Theorem \ref{lojawithradius} and Theorem \ref{thm:ra}, we obtain the desired Lojasiewicz inequality.

\begin{thm}[Lojasiewicz inequality] \label{thm:lo}
For any $\beta \in (0,3/4)$, there exist constants $C=C(n,Y)$ and $L_\beta=L(n,Y,\beta)$ such that the following property holds.

Let $\XX=\{M^n,(g(t))_{t\in I} \}$ be a closed Ricci flow with entropy bounded below by $-Y$. Assume $x_0^*=(x_0,t_0)\in\XX$ and $[t_0-10 r^2,t_0]\subset I$. If the weighted Riemannian manifold $\left(M,r^{-2}g(t_0-r^2),f_{x_0^*}(t_0-r^2)\right)$ satisfies $\rA \geq L_\beta$, then
	\begin{align*}
		\left|\mathcal W_{x_0^*}(r^2)-\Theta_m\right| \le C \left( \mathcal W_{x_0^*}(r^2/2)-\mathcal W_{x_0^*}(2 r^2) \right)^{\beta}.
	\end{align*}
\end{thm}
\begin{proof}
By translation and rescaling, we may assume $t_0=0$ and $r=1$.

If $\mathcal W_{x_0^*}(1/2)-\mathcal W_{x_0^*}(2) \ge 1$, the conclusion obviously holds. Thus, we assume $\mathcal W_{x_0^*}(1/2)-\mathcal W_{x_0^*}(2)<1$ so that $\rE$ is well-defined.

Thus, it follows from Theorems \ref{lojawithradius} and \ref{thm:ra} that for $\sigma>0$,
	\begin{align}\label{proofloja1}
	\left|\mathcal W_{x_0^*}(1)-\Theta_m\right| \le C(n,Y) \exp\lc-\frac{3\rBC^2}{16} \rc \leq C(n,Y)\exp\lc-\frac{3(1-3\sigma)^2\rE^2}{16}\rc,
	\end{align}
provided that $\rA \ge L(n, Y, \sigma)$. By Definition \ref{defnrE}, \eqref{proofloja1} implies
	\begin{align*}
		\left|\mathcal W_{x_0^*}(1)-\Theta_m\right| \le  C(n,Y)\left( \mathcal W_{x_0^*}(1/2)-\mathcal W_{x_0^*}(2) \right)^{\frac{3(1-3\sigma)^2}{4}}.
	\end{align*}
	This completes the proof by choosing $\beta=3(1-3\sigma)^2/4$.
\end{proof}

Similarly, from Theorem \ref{thm:rasing}, we have:
\begin{thm}\label{thm:losing}
	For any $\beta \in (0,3/4)$, there exist constants $C=C(n,Y)$ and $L_\beta=L(n,Y,\beta)$ such that the following property holds.	
	Let $(Z, d_Z, \t)$ be the completion of a closed Ricci flow $\XX=\{M^n, (g(t))_{t \in [-10r^2,0)}\}$ with entropy bounded below by $-Y$. Fix $z \in Z_0$. If the weighted Riemannian manifold $\left(M,r^{-2}g(-r^2),f_z(-r^2)\right)$ satisfies $\rA \geq L_\beta$, then
	\begin{align*}
		\left|\mathcal W_{z}(r^2)-\Theta_m\right| \le C \left( \mathcal W_{z}(r^2/2)-\mathcal W_{z}(2 r^2) \right)^{\beta}.
	\end{align*}
\end{thm}

As a corollary, we have the following summability result. For simplicity, we only state it for smooth closed Ricci flows. The corresponding generalization to noncollapsed Ricci flow limit spaces will be proved in \cite{FLrec25}.

\begin{cor}\label{quantisummabilityW}
	For any small $0<\ep\leq\ep(n,Y)$, $\alpha\in (1/4,1)$ and $\delta\leq\delta(n,Y,\ep,\alpha)$, the following holds. Let $\XX=\{M^n,(g(t))_{t\in I}\}$ be a closed Ricci flow with entropy bounded below by $-Y$. Fix $x_0^*=(x_0,t_0)\in \XX$ and constants $s_2>s_1>0$. If
	\begin{align}\label{quansum1W}
		\left|\WW_{x_0^*}(s_1)-\WW_{x_0^*}(s_2)\right|<\delta,
	\end{align}
and, for any $s\in [s_1,s_2]$, $x_0^*$ is $(n-m,\delta,\sqrt{s})$-cylindrical (see Definition \ref{def:almost0}), then
	\begin{align*}
		\sum_{s_1\leq r_j=2^{-j} \leq s_2}\left|\WW_{x_0^*}(r_j)-\WW_{x_0^*}(r_{j-1})\right|^{\alpha}<\ep.
	\end{align*}
\end{cor}
\begin{proof}
Without loss of generality, we assume $t_0=0$. For a given $\alpha \in (1/4, 1)$, we choose constants $\theta \in (1/3, 1)$ and $q$ such that 
	\begin{align} \label{eq:constchoice}
		\max \left\{1, \frac{1-\alpha}{\alpha} \right\}<q< \theta^{-1}.
	\end{align}
Notice that by our assumption on $\alpha$, such choices are possible.

	By assumption, for any $s\in [s_1,s_2]$, $x_0^*$ is $(n-m,\delta,\sqrt{s})$-cylindrical (see \eqref{smoothcylindrical}). Thus, if $\delta \leq \ep(n,Y , \theta)$, it follows from Lemma \ref{lem:smoothcyl} that the weighted Riemannian manifold $(M, s^{-1}g(-s), f(-s))$ satisfies $\rBC \ge L_{\beta}$, where $\beta:=1/(1+\theta)$, $\sigma$ is defined by $\beta=3(1-3\sigma)^2/4$, and $L_{\beta}$ is the same constant in Theorem \ref{thm:lo}. Thus, for any $j$ with $s_1\leq r_j \leq s_2$, we have 
	\begin{align*}
		\left|\WW_{x_0^*}(r_j)-\Theta_m\right|^{1+\theta}\leq C(n,Y) \left(\WW_{x_0^*}(r_{j+1})-\WW_{x_0^*}(r_{j-1})\right).
	\end{align*}

Define $N$ to be the largest number so that $r_N \ge s_1$. Similarly, define $N'$ to be the smallest number so that $r_{N'} \le s_2$. Moreover, we set $j_0 \in [N', N]$ to be the largest number such that $\WW_{x_0^*}(r_j) \le \Theta_m$. If no such $j_0$ exists, we set $j_0=N'$.

By the proof of \cite[Claim 3.3]{fang2025volume} (see also \cite[Lemma 6.9]{colding2015uniqueness}), we obtain 
	\begin{align}\label{polydecay}
		|\WW_{x_0^*}(r_j)-\Theta_m|\leq C(n,Y, \theta) (j+1-N')^{-\frac{1}{\theta}}
	\end{align}
for $j \in [N', j_0]$ and
	\begin{align}\label{polydecay1}
		|\WW_{x_0^*}(r_j)-\Theta_m|\leq C(n,Y, \theta) (N+1-j)^{-\frac{1}{\theta}}
	\end{align}
for $j \in [j_0+1, N]$.
		
	Now we set $b_j:=\WW_{x_0^*}(r_{j})-\WW_{x_0^*}(r_{j-1})$ and $a_j=|\WW_{x_0^*}(r_j)-\Theta_m|$. Moreover, we choose a large number $k \in \mathbb N$ to be determined later.	
	
By a direct calculation, we have
	 	\begin{align}
		\sum_{l=N'+k}^{j_0} b_l (l-N')^q&=\sum_{l=N'+k}^{j_0}(a_{l-1}-a_{l})(l-N')^q\nonumber\\
		&=a_{N'+k-1} k^q-a_{j_0} (j_0-N')^q+\sum_{l=N'+k}^{j_0-1}a_l\left((l+1-N')^q-(l-N')^q\right)\nonumber\\
		&\leq a_{N'+k-1} k^q+q\sum_{l=N'+k}^{j_0-1} a_l (l+1-N')^{q-1} \notag \\
		& \leq  C(n, Y, \theta) \lc k^{q-\theta^{-1}} +q\sum_{l=N'+k}^{j_0-1} (l+1-N')^{q-1-\theta^{-1}} \rc \leq C(n, Y, \theta,q) k^{q-\theta^{-1}}, \label{summability1}
	\end{align}
	 where we used \eqref{polydecay} for the last line. Similarly, by using \eqref{polydecay1}, we obtain
		 	\begin{align*}
		\sum_{l=j_0+1}^{N-k} b_l (N-l)^q&=\sum_{l=j_0+1}^{N-k} (a_{l}-a_{l-1})(N-l)^q\nonumber\\
		&=a_{N-k} k^q -a_{j_0}(N-j_0)^q+\sum_{l=j_0+1}^{N-k-1} a_l\left((N+1-l)^q-(N-l)^q\right)\nonumber\\
		& \leq  C(n, Y, \theta) \lc k^{q-\theta^{-1}} +q\sum_{l=j_0+1}^{N-k-1} (N-l+1)^{q-1-\theta^{-1}} \rc \le C(n, Y, \theta,q) k^{q-\theta^{-1}}.
	\end{align*} 
	
Now, we conclude 
		\begin{align}
	\sum_{l=N'+k}^{j_0} b_l^{\alpha}=&\sum_{l=N'+k}^{j_0}\left(b_l^\alpha (l-N')^{\alpha q}\right) (l-N')^{-\alpha q} \notag\\
	\le& \lc \sum_{l=N'+k}^{j_0} b_l (l-N')^{q} \rc^{\alpha} \lc \sum_{l=N'+k}^{j_0} (l-N')^{-\frac{\alpha q}{1-\alpha}} \rc^{1-\alpha} \notag\\
    \le & C(n, Y, \theta,q,\alpha) k^{1-\alpha(1+ \theta^{-1})}, \label{extrasumma}
	\end{align}
	where the last inequality holds by our choice \eqref{eq:constchoice} and \eqref{summability1}.
	
Similarly, we have
			\begin{align*}
	\sum_{l=j_0+1}^{N-k} b_l^{\alpha} \le C(n, Y, \theta,q,\alpha) k^{1-\alpha(1+ \theta^{-1})}.
	\end{align*}
	
	Combining all these estimates with \eqref{quansum1W}, we obtain
				\begin{align*}
	\sum_{l=N'}^{N} b_l^{\alpha} \le C(n, Y, \theta,q,\alpha) k^{1-\alpha(1+ \theta^{-1})}+2k \delta.
	\end{align*}
	Thus, we can first choose a large $k$ and then choose a small $\delta$ such that the sum above is bounded by $\ep$.
	
	In sum, the proof is complete.
\end{proof}

\section{Strong uniqueness for cylindrical tangent flows} \label{sec:scyf}

Let $\XX=\{M^n, (g(t))_{t \in [-T,0)}\}$ be a closed Ricci flow with entropy bounded below by $-Y$, where $0$ is the first singular time. Suppose $(Z, d_Z, \t)$ is the completion of $\XX$ (see Subsection \ref{sec:cylpts}). Fix a point $z \in Z_0$ and consider the modified Ricci flow $(M, g^z(s),f^z(s))$ with respect to $z$ (see Definition \ref{def:mrf}).

\begin{thm} \label{thm:modi}
Suppose $z$ is a cylindrical singularity (see Definition \ref{def:ccc}). Denote by $\rA(s)$ and $\rBC(s)$ the radius functions in Definition \ref{defnradii} with respect to $(M, g^z(s),f^z(s))$. Then, for any $\ep>0$,
\begin{align*}
\min\{\rA(s), \rBC(s)\} \ge \sqrt{(12-\ep) \log s},
	\end{align*}
	where $\sigma=10^{-10} \ep$, provided that $s$ is sufficiently large.
\end{thm}

\begin{proof}
We set $W(s):=\WW(g^z(s), f^z(s))$. By Theorem \ref{thm:weakunique}, any tangent flow at $z$ is isometric to $\bar{\mathcal C}^k$. In particular, we have
\begin{align*}
\lim_{s \to +\infty} W(s)=\Theta_{n-k}.
	\end{align*}
Moreover, for any $L>1$, we have
\begin{align*}
\min\{\rA(s), \rBC(s)\} \ge L,
	\end{align*}
if $s$ is sufficiently large. By Theorem \ref{thm:losing}, this implies for $\beta=3(1-3\sigma)^2/4$, we have
\begin{align*}
\abs{W(s)-\Theta_{n-k}} \le C(n, Y) \lc W(s+\log 2)-W(s-\log 2) \rc^{\beta}
	\end{align*}
for sufficiently large $s$. By the same argument as in the proof of Corollary \ref{quantisummabilityW} (see \eqref{polydecay}), we conclude that 
\begin{align*}
\abs{W(s)-\Theta_{n-k}} \le C(n, Y, \ep) s^{-\frac{\beta}{1-\beta}}=C(n, Y, \ep) s^{-\frac{3(1-3\sigma)^2}{4-3(1-3\sigma)^2}}
	\end{align*}
	for large $s$, which implies
	\begin{align} \label{eq:esti101}
\exp\lc -\frac{\rE^2(s)}{4}\rc:=\abs{W(s+\log 2)-W(s-\log 2)} \le C(n, Y, \ep) s^{-\frac{3(1-3\sigma)^2}{4-3(1-3\sigma)^2}}.
	\end{align}
By Theorem \ref{thm:rasing} and \eqref{eq:esti101}, it follows that
\begin{align*}
\min\{\rA(s), \rBC(s)\} \ge (1-3\sigma) \sqrt{\frac{12 (1-3\sigma)^2}{4-3(1-3\sigma)^2} \log s-C(n, Y, \ep)}.
	\end{align*}	
	Thus, the conclusion holds by our choice of $\sigma$, if $s$ is sufficiently large.	
\end{proof}

Based on Theorem \ref{thm:modi}, we prove the following strong uniqueness result.
\begin{thm}[Strong uniqueness of the cylindrical tangent flow I]\label{thm:stronguni}
	Suppose $z$ is a cylindrical singularity. Then for any small $\ep>0$, there exists a large constant $\bar s$ such that for any $s_0 \ge \bar s$, there exists a diffeomorphism $\varphi_{s_0}$ from $\left\{\bar b  \le \sqrt{(8-\ep)\log s_0}\right\}$ onto a subset of $M$ such that on $\left\{\bar b \le \sqrt{(8-\ep)\log s_0}\right\}$ and for all $s\geq s_0$, we have
	\begin{align*}
	\left[\varphi_{s_0}^*g^z(s)-\bar g \right]_{[\ep^{-1}]}+\left[\varphi_{s_0}^*f^z(s)-\bar f \right]_{[\ep^{-1}]}\leq C(n,Y,\ep)e^{\frac{\bar f}{2}} s^{-1+\ep}.
	\end{align*} 
	Here, the $C^{[\ep^{-1}]}$-norms are with respect to $\bar g$.
\end{thm}
\begin{proof}
Throughout the proof, we fix a small constant $\ep>0$ and set $\sigma=10^{-10} \ep$.
	
	Let $\rA(s)$ and $\rE(s)$ denote the radius functions in Definitions \ref{defnradii} and \ref{defnrE} with respect to $(M,g^z(s),f^z(s))$. By Theorem \ref{thm:modi}, for sufficiently large $s$, we have
	\begin{align*}
		\rA(s) \geq \sqrt{(12-\ep)\log s}.
	\end{align*}
	Let $\varphi_s$ be the diffeomorphism corresponding to $\rA(s)$ in Definition \ref{defnradii}. Then by Corollary \ref{almostshrinkerlarger} (and its proof) and Theorem \ref{thm:rasing}, for any $l \in \mathbb{N}$, on the set $\{\bar b \leq \sqrt{(8 - \ep/3)\log s}\}$, we have
	\begin{align}
		&\left[\varphi_s^*\lc \frac{g^z(s)}{2}-\na^2f^z(s)-\Ric(g^z(s))\rc\right]_{l} \notag \\
        \leq & C(n,Y,\ep,l)\exp\lc \frac{\bar f}{2}-\frac{\rA(s)^2}{8}\rc\nonumber\\
        \leq & C(n,Y,\ep,l)\exp\lc \frac{\bar f}{2}-\frac{(1-3\sigma)^2\rE(s)^2}{8}\rc\nonumber\\
		= & C(n,Y,\ep,l) e^{\frac{\bar f}{2}} \abs{W(s+\log 2)-W(s-\log 2)}^{\frac{(1-3\sigma)^2}{2}},\label{estishrinkerquanex001a}
	\end{align}
provided $s$ is sufficiently large. Here, as before, we set $W(s):=\WW(g^z(s), f^z(s))$.

Next, we set $\Omega_s:=\varphi_s\lc \left\{\bar b\leq \sqrt{(8-\ep/3)\log s}\right\}\rc$. Then it follows from \eqref{estishrinkerquanex001a} that for sufficiently large $s$, we have on $\Omega_s$,
		\begin{align}\label{estishrinkerquanex001}
		\left[ \frac{g^z(s)}{2}-\na^2f^z(s)-\Ric(g^z(s))\right]_{l} \leq &  C(n,Y,\ep,l) e^{\frac{\bar f \circ \varphi^{-1}_s}{2}} \abs{W(s+\log 2)-W(s-\log 2)}^{\frac{(1-3\sigma)^2}{2}} \notag\\
        \le & C(n,Y,\ep,l) e^{\frac{f^z(s)}{2}} \abs{W(s+\log 2)-W(s-\log 2)}^{\frac{(1-3\sigma)^2}{2}},
	\end{align}	
	where the norm is taken with respect to $g^z(s)$. Taking the trace yields, on $\Omega_s$,
	\begin{align}\label{estishrinkerquan3}
		\left[ \frac{n}{2}-\Delta f^z(s)-\scal(g^z(s))\right]_l \leq C_1(n,Y,\ep,l) e^{\frac{f^z(s)}{2}} \abs{W(s+\log 2)-W(s-\log 2)}^{\frac{(1-3\sigma)^2}{2}}.
	\end{align}

\begin{claim}\label{claimsummation}
For any constants $\alpha \in (1/4,1)$ and sufficiently large $s$, we have
		\begin{align}
    \label{equ:integrability}\int_{s}^\infty \left|W(z+\log 2)-W(z-\log 2)\right|^\alpha \,\mathrm{d} z\leq C(n,Y,\alpha,\ep) s^{1-\frac{4\alpha}{4-3(1-3\sigma)^2}}.	
    \end{align}
	\end{claim}

Indeed, for any sufficiently large constant $s$, it follows from the same argument as in the proof of Corollary \ref{quantisummabilityW} (see \eqref{extrasumma}) that
\begin{align*}
\sum_{j \ge s}^\infty |W((j+1)\log 2)-W(j\log 2)|^\alpha \le C(n, Y,\alpha, \ep) s^{1-\frac{4\alpha}{4-3(1-3\sigma)^2}}.
\end{align*}
As $W(s)$ is increasing, \eqref{equ:integrability} immediately follows, and hence Claim \ref{claimsummation} is proved.

By Definition \ref{defnradii}, any $H_n$-center of $(4\pi)^{-n/2}e^{-f^z(s)}\,\mathrm{d}V_{g(s)}$ lies in $\varphi_s\lc\{\bar b\leq C(n)\}\rc$. Therefore, for sufficiently large $s$, it follows from Theorem \ref{thm:upper} (ii) that any $x\notin \Omega_s$ satisfies
\begin{align}\label{equ:fout}
f^z(x,s)\geq (2-\ep/10)\log s.
\end{align}
Now, we fix $l=[\ep^{-1}]$ and take a constant $\bar s$, which will be determined in Claim \ref{claimpotential}, sufficiently large such that \eqref{estishrinkerquanex001}, \eqref{estishrinkerquan3} and Claim \ref{claimsummation} hold for any $s\geq \bar s$. Moreover, we set $C_1= C_1(n, Y, \ep, [\ep^{-1}])$, which is the constant from \eqref{estishrinkerquan3}. For the rest of the proof, we fix $s_0 \ge \bar s$.	
	
	\begin{claim}\label{claimpotential}
	We have
	\begin{align*}
\Omega^{s_0}:=\varphi_{s_0}\lc\left\{\bar b \leq \sqrt{(8-\ep/2)\log s_0}\right\}\rc \subset \Omega_s, \quad \forall s \ge s_0.
	\end{align*}	
	Moreover, for any $x \in \Omega^{s_0}$,
		\begin{align}\label{estipotential1}
\abs{\exp \lc -\frac{f^z(x, s)}{2} \rc-\exp \lc -\frac{f^z(x, s_0)}{2} \rc} \le C(n, Y, \ep) s_0^{-1+\frac{\ep}{100}}.
		\end{align}
	\end{claim}
	Fix $x\in \Omega^{s_0}$, and suppose $s_1>s_0$ is the maximal time $s$ such that $x\in \Omega_{s}$. If $s_1<\infty$, then by \eqref{estishrinkerquan3} and the evolution equation for $f^z$ (see \eqref{equationMRF}), we obtain
	\begin{align}\label{estipotential2a}
		|\partial_s f^z(x, s)| \le C_1 \exp \lc -\frac{f^z(x, s)}{2} \rc \abs{W(s+\log 2)-W(s-\log 2)}^{\frac{(1-3\sigma)^2}{2}},
	\end{align}
    for any $s \in [s_0, s_1]$. By solving this ordinary differential inequality, we obtain
	\begin{align*}
\exp \lc -\frac{f^z(x, s_1)}{2} \rc \ge \exp \lc -\frac{f^z(x, s_0)}{2} \rc-\frac{C_1}{2} \int_{s_0}^{s_1} \abs{W(s+\log 2)-W(s-\log 2)}^{\frac{(1-3\sigma)^2}{2}}\, \mathrm{d}s.
	\end{align*}
Using \eqref{equ:fout} and Claim \ref{claimsummation} with the choice $\alpha=(1-3\sigma)^2/2$, we obtain
	\begin{align}\label{estipotential2b}
s_1^{-1+\frac{\ep}{20}} \ge s_0^{-1+\frac{\ep}{16}}-C(n, Y, \ep) s_0^{\frac{4-5(1-3\sigma)^2}{4-3(1-3\sigma)^2}}.
	\end{align}
By our choice of $\sigma$, we derive a contradiction from \eqref{estipotential2b} if $\bar s$ is sufficiently large. Hence $s_1=\infty$, and \eqref{estipotential1} follows from \eqref{estipotential2a}. This completes the proof of Claim \ref{claimpotential}.
	
By Claim \ref{claimpotential}, we conclude that $\Omega^{s_0} \subset \Omega_{s}$ for any $s\geq s_0$, so the estimates \eqref{estishrinkerquanex001} and \eqref{estishrinkerquan3} hold on $\Omega^{s_0}$. Thus, it follows from \eqref{estishrinkerquanex001}, \eqref{estishrinkerquan3}, the evolution equation \eqref{equationMRF}, and Claim \ref{claimsummation} that on $\left\{\bar b \leq \sqrt{(8-\ep/2)\log s_0}\right\}$, $\lc \varphi_{s_0}^*g^z(s),\varphi_{s_0}^*f^z(s)\rc$ converges smoothly to $\lc g_\infty, f_\infty\rc$ as $s \to \infty$. Moreover, we obtain for all $s \ge s_0$,
		\begin{align}\label{equ:decay1}
		\left[\varphi_{s_0}^*g^z(s)-g_\infty \right]_{[\ep^{-1}]}+\left[\varphi_{s_0}^*f^z(s)-f_\infty \right]_{[\ep^{-1}]} \leq C(n,Y,\ep)e^{\frac{\bar f}{2}} s^{-1+\frac{\ep}{10}}.
	\end{align}

	\begin{claim}\label{cla:limit}
There exists a self-diffeomorphism $\varphi$ of $\left\{\bar b \le \sqrt{(8-\ep)\log s_0}\right\}$ such that
		\begin{align*}
			g_\infty=\varphi^*\bar g,\quad f_\infty=\varphi^*\bar f.
		\end{align*}
	\end{claim}
By Lemma \ref{lem:smoothcyl}, we have for a sequence $s_i \to \infty$,
\begin{align*}
		\lc M, g^z(s_i),f^z(s_i), p_i\rc\xrightarrow[i\to\infty]{\quad\text{pointed-Cheeger--Gromov} \quad}\lc \bar M,\bar g,\bar f,\bar p\rc.
	\end{align*}
	Here $p_i$ and $\bar p$ are minimum points of $f^z(s_i)$ and $\bar f$, respectively. Hence, for sufficiently large $i$, there exists a diffeomorphism $\psi_{i}:U_{i}\subset \bar M\to M$, such that $U_{i}\subset U_{i+1}$, and $\bigcup_{i}U_{i}=\bar M$, satisfying
	\begin{align}\label{equ:convcyl}
		\psi_{i}(\bar p)=p_i, \quad \psi_{i}^*g^z(s_i)\xrightarrow[i\to\infty]{C_{\loc}^\infty}\bar g,\quad \psi_{i}^*f^z(s_i)\xrightarrow[i\to\infty]{C_{\loc}^\infty}\bar f.
	\end{align}
From \eqref{equ:convcyl}, we obtain
	\begin{align*}
	\lc \psi_{i}^{-1}\circ \varphi_{s_0}\rc^*\bar g \xrightarrow[i\to\infty]{C^\infty_{\mathrm{loc}}} g_\infty \quad \text{and} \quad 
	\lc \psi_{i}^{-1}\circ \varphi_{s_0}\rc^*\bar f \xrightarrow[i\to\infty]{C^\infty_{\mathrm{loc}}} f_\infty.
	\end{align*}
Therefore, by passing to a subsequence, the maps $\psi_{i}^{-1}\circ\varphi_{s_0}$ converge smoothly to a diffeomorphism $\varphi$ from $\left\{\bar b \le \sqrt{(8-\ep)\log s_0}\right\}$ onto itself satisfying $g_\infty=\varphi^*\bar g$ and $f_\infty=\varphi^*\bar f$. This completes the proof of Claim \ref{cla:limit}.

	If we replace $\varphi_{s_0}$ by $\varphi_{s_0}\circ\varphi^{-1}$ (still denoted by $\varphi_{s_0}$), then combining \eqref{equ:decay1} with Claim \ref{cla:limit} yields
	\begin{align*}
	\left[\varphi_{s_0}^*g^z(s)-\bar g \right]_{[\ep^{-1}]}+\left[\varphi_{s_0}^*f^z(s)-\bar f \right]_{[\ep^{-1}]}\leq C(n,Y,\ep)e^{\frac{\bar f}{2}} s^{-1+\frac{\ep}{10}}
	\end{align*}
	for all $s\geq s_0$, on $\left\{\bar b \le \sqrt{(8-\ep)\log s_0}\right\}$. This completes the proof after adjusting constants.
\end{proof}

\begin{lem}\label{lem:gluediff}
In the same setting of Theorem \ref{thm:stronguni}, for any $s_2>s_1 \ge \bar s$, there exists an isometry $\varphi_{s_2,s_1}:(\bar M,\bar g)\to (\bar M,\bar g)$ such that
	\begin{align*}
		\varphi_{s_2,s_1}^*\bar f=\bar f,\quad \varphi_{s_2}\circ\varphi_{s_2,s_1}=\varphi_{s_1}
	\end{align*}
on $\left\{\bar b \le \sqrt{(8-\ep)\log s_1}\right\}$.
\end{lem}
\begin{proof}
By Theorem \ref{thm:stronguni}, we know for $\Omega:=\left\{\bar b \le \sqrt{(8-\ep)\log s_1}\right\}$,
	$$\lc\Omega, \varphi_{s_1}^*g^z(s),\varphi_{s_1}^*f^z(s)\rc\xrightarrow[s\to\infty]{C^\infty_\loc }\lc \Omega, \bar g,\bar f\rc,$$
	and
	$$\lc \Omega, \varphi_{s_2}^*g^z(s),\varphi_{s_2}^*f^z(s)\rc\xrightarrow[s\to\infty]{ C^\infty_\loc }\lc  \Omega, \bar g,\bar f\rc.$$
	
Therefore, $\varphi:=\varphi_{s_2}^{-1}\circ\varphi_{s_1}$ is well-defined on $\Omega$ such that
	\begin{align}\label{equ:gluediff3}
		\varphi^*\bar g=\varphi_{s_1}^*\circ\lc \varphi_{s_2}^{-1}\rc^*\bar g=\lim_{s\to\infty}\varphi_{s_1}^*\circ\lc \varphi_{s_2}^{-1}\rc^*\circ\varphi_{s_2}^*g^z(s)=\lim_{s\to\infty}\varphi_{s_1}^*g^z(s)=\bar g
	\end{align}
and similarly
	\begin{align}\label{equ:gluediff4}
	\varphi^*\bar f=\bar f.
	\end{align}
	
	By \eqref{equ:gluediff3} and \eqref{equ:gluediff4}, $\varphi$ is an isometry of $\Omega$, which preserves $\bar f$. Thus, $\varphi$ is the restriction of a map $\varphi_{s_2, s_1}$ in $\mathrm{O}(n-m) \times \mathrm{Iso}(S^m)$, which is an isometry of $(\bar M, \bar g)$ preserving $\bar f$. This completes the proof.
\end{proof}

\begin{thm}[Strong uniqueness of the cylindrical tangent flow II]\label{thm:stronguni2}
Suppose $z$ is a cylindrical singularity. Then for any small $\ep>0$, there exists a large constant $\bar s$ such that for any $j \ge \bar s$, there exists a diffeomorphism $\psi_{j}$ from $\Omega^j:=\left\{\bar b \le \sqrt{(8-\ep)\log j}\right\}$ onto a subset of $M$ such that
	\begin{align*}
		\psi_{j+1}=\psi_j \quad \mathrm{on}\quad \Omega^j,
	\end{align*}
and for all $s\geq j$, we have on $\Omega^j$,
	\begin{align*}
	\left[\psi_{j}^*g^z(s)-\bar g \right]_{[\ep^{-1}]}+\left[\psi_{j}^*f^z(s)-\bar f \right]_{[\ep^{-1}]}\leq C(n,Y,\ep)e^{\frac{\bar f}{2}} s^{-1+\ep}.
	\end{align*}
\end{thm}
\begin{proof}
	We can define $\psi_j$ inductively. First, we choose $j_0$ with $j_0-1 < \bar s \le j_0$ and set $\psi_{j_0}=\varphi_{j_0}$ as in Theorem \ref{thm:stronguni}. Once $\psi_{j}$ is defined, we can apply Lemma \ref{lem:gluediff} to obtain an isometry $\varphi$ of $(\bar M, \bar g)$ so that on $\Omega^{j}$, we have
		\begin{align*}
		\varphi^*\bar f=\bar f,\quad \varphi_{j+1}\circ\varphi=\psi_{j}.
	\end{align*}
	Then, we set $\psi_{j+1}:=\varphi_{j+1}\circ\varphi$, which is well-defined on $\Omega^{j+1}$. Thus, the conclusion follows from the induction and Theorem \ref{thm:stronguni}.
\end{proof}

Our results also apply to complete ancient solutions to Ricci flow with cylindrical tangent flow at infinity (see \cite[Example 11.5]{fang2025RFlimit}). First, we have the following definition.

\begin{defn}
Let $\XX=\{M^n, (g(t))_{t \in (-\infty, 0]}\}$ be a complete Ricci flow with bounded curvature on any compact time interval of $(-\infty, 0]$ and with entropy bounded below. $\XX$ is said to have a cylindrical tangent flow at infinity if one of the tangent flows at infinity is isometric to $\bar{\mathcal C}^k$.
\end{defn}

The following result is parallel to Theorem \ref{thm:weakunique}, which is proved by the same method using \cite[Theorem 6.5]{li2023rigidity}.

\begin{thm}[Uniqueness of the cylindrical tangent flow at infinity]\label{thm:uniqcyl} 
Let $\XX=\{M^n, (g(t))_{t \in (-\infty, 0]}\}$ be a complete Ricci flow with bounded curvature on any compact time interval of $(-\infty, 0]$ and with entropy bounded below. If $\XX$ has $\bar{\mathcal C}^k$ as a tangent flow at infinity, then any tangent flow at infinity is isometric to $\bar{\mathcal C}^k$. 
\end{thm}

Now, we fix a base point $p^* \in \XX_0$ and set $\mathrm{d}\nu_{p^*;t}=(4\pi|t|)^{-n/2}e^{-f(t)}\,\mathrm{d}V_{g(t)}$. Let $\phi_t$ be the family of diffeomorphisms generated by $-\na_{g(t)} f(t)$ with $\phi_{0}=\mathrm{id}$. Then we define 
  \begin{equation} \label{eq:mcf1}
  \begin{dcases}
    &g' (s):= e^s \phi^*_{-e^{-s}} g(-e^{-s}),   \\
    &f' (s):=\phi^*_{-e^{-s}} f_z(-e^{-s}).
        \end{dcases}
  \end{equation}\index{$(g'(s), f'(s))$}
Clearly, $(g'(s), f'(s))$ satisfies \eqref{equationMRF}. Similar to Theorem \ref{thm:modi} and Theorem \ref{thm:stronguni2}, we have the following results by the same proof.

\begin{thm} \label{thm:modi1}
Denote by $\rA(s)$ and $\rBC(s)$ the radius functions in Definition \ref{defnradii} with respect to $(M, g'(s),f'(s))$. Under the same assumptions of Theorem \ref{thm:uniqcyl}, for any $\ep>0$, we have
\begin{align*}
\min\{\rA(s), \rBC(s)\} \ge \sqrt{(12-\ep) \log |s|},
	\end{align*}
where $\sigma=10^{-10} \ep$, provided that $|s|$ is sufficiently large.
\end{thm}

\begin{thm}[Strong uniqueness of cylindrical tangent flow at infinity]\label{thm:stronguni-infty}
Under the same assumptions of Theorem \ref{thm:uniqcyl}, for any small $\ep>0$, there exists a large constant $\bar s$ such that for any $j \ge \bar s$, there exists a diffeomorphism $\psi_{-j}$ from $\Omega^j=\left\{\bar b \le  \sqrt{(8-\ep)\log j}\right\}$ onto a subset of $M$ such that
\begin{align*}
	\psi_{-j-1}=\psi_{-j} \quad \mathrm{on}\quad \Omega^j,
\end{align*}
and for any $s \le -j$, we have on $\Omega^j$,
\begin{align*}
	\left[\psi_{-j}^*g'(s)-\bar g \right]_{[\ep^{-1}]}+\left[\psi_{-j}^*f'(s)-\bar f \right]_{[\ep^{-1}]}\leq C(n,Y,\ep)e^{\frac{\bar f}{2}} |s|^{-1+\ep}.
\end{align*} 
\end{thm}

\section{Generalization to quotient cylindrical geometries}\label{sec:lojaquo}

In this subsection, we prove the Lojasiewicz inequality near quotient cylinders. First, we set as before
\begin{align*}
\mathcal C^{n-m}_{-1}=(\bar M,\bar g,\bar f)=\left(\R^{n-m}\times S^{m}, g_E \times g_{S^m}, \frac{|\vec{x}|^2}{4}+\frac{m}{2}+\Theta_m \right).
\end{align*}
For any finite group $\Gamma \leqslant \mathrm{Iso}(\mathcal C^{n-m}_{-1})$ acting freely on $\bar M$, we set $\pi: \bar M \to \bar M/\Gamma$ to be the natural quotient map and define
\begin{align} \label{eq:quoc}
\mathcal C^{n-m}(\Gamma)_{-1}=(\bar M_{\Gamma}, \bar g_{\Gamma},\bar f_{\Gamma}),
\end{align}\index{$\mathcal C^{n-m}(\Gamma)_{-1}$}\index{$(\bar M_{\Gamma}, \bar g_{\Gamma},\bar f_{\Gamma})$}
where $\bar M_{\Gamma}=\bar M/\Gamma$, $\bar g_{\Gamma}$ is the quotient metric of $\bar g$, and $\bar f_{\Gamma}$ is defined so that
\begin{align*}
\pi^* \bar f_{\Gamma}=\bar f-\log |\Gamma|.
\end{align*}

It is clear by a direct calculation that
\begin{align*}
\Theta_m(\Gamma):=\WW(\bar g_{\Gamma}, \bar f_{\Gamma})=\Theta_m-\log |\Gamma|.
\end{align*}\index{$\Theta_m(\Gamma)$}

Similar to Definition \ref{defnradii}, we have the following definitions. As before, we set $\bar b=2\sqrt{|\bar f|}$ and $b=2\sqrt{|f|}$.

\begin{defn}\label{defnradiiquo}
	For $\sigma\in (0,1/10)$ and a weighted Riemannian manifold $(M,g,f)$ with $\Phi=\mathbf{\Phi}(g, f)=\dfrac{g}{2}-\Ric(g)-\nabla^2 f$, we define
	\begin{enumerate}[label=\textnormal{(\Alph{*})}]
		\item \emph{($\rA^{|\Gamma|}$-radius)}\index{$\rA^{\vert\Gamma\vert}$} $\rA^{|\Gamma|}$ to be the largest number $L$ for which there exists a smooth covering map $\varphi_A$ of degree $|\Gamma|$ from a domain $U$ containing $\{\bar b \le L\}$ onto a subset of $M$ such that
		\begin{align*}
			\left[\bar g-\varphi_A^* g\right]_5+\left[\bar f-\log |\Gamma|-\varphi_A^* f\right]_5 \leq e^{\frac{\bar f}{4}-\frac{L^2}{16}}.
		\end{align*}
		\item \emph{($\rBC^{|\Gamma|}$-radius)}\index{$\rBC^{\vert\Gamma\vert}$} $\rBC^{|\Gamma|}$ to be the largest number $L$ for which there exists a smooth covering map $\varphi_B$ of degree $|\Gamma|$ from a domain $U$ containing $\{\bar b \le L\}$ onto a subset of $M$ such that
		\begin{align*}
			\left[\bar g-\varphi_B^* g\right]_0+\left[\bar f-\log |\Gamma|-\varphi_B^* f\right]_0\leq e^{-\frac{L^2}{33}},
		\end{align*}
		and
		\begin{equation*}
			\int_{\{\bar b\le L\}}\left|\varphi_B^*\Phi\right|^2 \,\mathrm{d}V_{\bar f}\leq e^{-\frac{L^2}{4-\sigma}}.
		\end{equation*}
		Furthermore, for all $k \in [1, 10^{10} n \sigma^{-1}]$, the $C^k$-norms of $\bar g-\varphi_B^* g$ and $\bar f-\log|\Gamma|-\varphi_B^* f$ are bounded by $1$. 
	\end{enumerate}
\end{defn}

Similar to Proposition \ref{prop:con}, we have the following proposition; see also \cite[Theorem 5.6]{li2023rigidity}.

\begin{prop}\label{contractionquotient}
Let $(M,g,f)$ be a normalized weighted Riemannian manifold such that
\begin{align}\label{normalquo}
\int_{M\setminus B(p,L)} 1 \,\mathrm{d}V_f\leq C_V e^{-\frac{L^2}{15}}, \quad \forall L>1,
\end{align}
where $p$ is a fixed minimum point of $f$. Then there exists a constant $L_1=L_1(n,C_V, |\Gamma|, \sigma)$ such that if $\rBC^{|\Gamma|}\geq L_1$ for $(M,g,f)$, then
	\begin{align*}
		\rA^{|\Gamma|}\geq \rBC^{|\Gamma|}-3.
	\end{align*}
\end{prop}

\begin{proof}
We set $L=\rBC^{|\Gamma|} \ge L_1$. By the definition, there exists a smooth covering map $\varphi_B$ of degree $|\Gamma|$ from a domain $U$ containing $\{\bar b \le L\}$ onto a subset of $M$ such that
	\begin{align*}
		\left[\bar{g} - \varphi_B^* g\right]_0 + \left[\bar{f}-\log |\Gamma|- \varphi_B^* f \right]_0 \leq e^{-\frac{L^2}{33}},
	\end{align*}
	and for $\Phi=g/2-\Ric(g)-\na^2 f$, we have 
	\begin{align*}
\int_{\{\bar b\le L\}}\left|\varphi_B^*\Phi\right|^2 \,\mathrm{d}V_{\bar f}\leq e^{-\frac{L^2}{4-\sigma}}.
	\end{align*}
	Furthermore, all $C^k$-norms of $\bar g-\varphi_B^* g$ and $\bar f-\log|\Gamma|-\varphi_B^* f$ are bounded by $1$ for $k \in [1, 10^{10} n\sigma^{-1}]$. It follows from the assumption \eqref{normalquo} that 
\begin{align} \label{normalquo1}
\int_{\{\bar b \le L-1/3\} \setminus B_{\varphi_B^* g}(p',D)} \,\mathrm{d}V_{\varphi_B^* f}\leq C(n, C_V, |\Gamma|) e^{-\frac{D^2}{15.1}}, \quad \forall D>1,
\end{align}	
	where $p'$ is a fixed point in $\{\varphi_B^{-1}(p)\}$. Indeed, by the definition of $\rB^{\Gamma}$, 
		\begin{align*}
\varphi_B^{-1}(B(p, D)) \subset B_{\varphi_B^* g}(p',D+C(n)).
	\end{align*}	
	Thus, \eqref{normalquo1} follows from \eqref{normalquo} by adjusting $D$.	
	
By Proposition \ref{prop:con} (see \eqref{improvedestimate3a} and Remark \ref{rem:constan1}), we can find a self-diffeomorphism $\psi_1$ of $\{\bar b \le L-1\}$ such that
	\begin{align*}
		\left[\psi_1^*\varphi_B^*g-\bar g\right]_5+\left[\psi_1^*\varphi_B^*f-\bar f-\log |\Gamma| \right]_5\leq e^{\frac{\bar f}{4}-\frac{(L-2)^2}{16}}
	\end{align*}
	on $\{ \bar b \le L-2\}$. We denote the submersion $\varphi_B \circ \psi_1$ by $\varphi_A$. 
	
Then we consider the subset $\{b \le L-5/2\} \bigcap \varphi_B(U)$ and set $U':=\psi_1^{-1}\circ\varphi_B^{-1} \lc \{b \le L-5/2\} \bigcap \varphi_B(U) \rc$. Since $\psi_1$ is almost an identity, it is clear by the definition of $\rBC^{|\Gamma|}$ that
	\begin{align*}
\{ \bar b \le L-3\} \subset U' \subset \{ \bar b \le L-2\}.
	\end{align*}
	
From our construction, the map $\varphi_A$ from $U'$ is a covering map of degree $|\Gamma|$. Consequently, $\varphi_A$ satisfies the conclusions for $\rA^{|\Gamma|}$ in Definition \ref{defnradiiquo}, and hence we have 
	\begin{align*}
\rA^{|\Gamma|}\geq \rBC^{|\Gamma|}-3.
	\end{align*}
\end{proof}

Similar to Theorem \ref{lojawithradius}, we next prove

\begin{thm}\label{lojawithradiusquotient}
There exist constants $L_2= L_2(n, Y,|\Gamma|)>0$ and $C=C(n,Y)>0$ such that the following holds.
	Let $\XX=\{M^n,(g(t))_{t\in I} \}$ be a closed Ricci flow with entropy bounded below by $-Y$. Assume $x_0^*=(x_0,t_0)\in\XX$ and $[t_0-2r^2,t_0]\subset I$. If the weighted Riemannian manifold $\left(M,r^{-2}g(t_0-r^2),f_{x_0^*}(t_0-r^2)\right)$ satisfies $\rBC^{|\Gamma|} \geq L_2$, then
	\begin{equation*}
		\left|\WW_{x_0^*}(r^2)-\Theta_m(\Gamma) \right|\leq C \exp\lc-\frac{3(\rBC^{|\Gamma|})^2}{16}\rc.
	\end{equation*}
\end{thm}
\begin{proof}
Without loss of generality, we assume $t_0=0$, $r=1$ and choose a small parameter $\ep\ll 1$ to be determined later. Also, we set $L=\rBC^{|\Gamma|}$, $g_0=g(-1)$ and $f_0=f_{x_0^*}(-1)$ for simplicity. 
	
	By Definition \ref{defnradiiquo}, we can find a smooth covering map $\varphi_B$ of degree $|\Gamma|$ from a domain $U$ containing $\{\bar b \le L\}$ onto a subset of $M$ such that
	\begin{align*}
		\left[\bar g-\varphi_B^* g_0\right]_0+\left[\bar f-\log|\Gamma|-\varphi_B^* f_0\right]_0\leq e^{-\frac{L^2}{33}},
	\end{align*}
	and
	\begin{align*}
\int_{\{\bar b\le L\}}\left|\varphi_B^*\Phi\right|^2 \,\mathrm{d}V_{\bar f}\leq e^{-\frac{L^2}{4-\sigma}},
	\end{align*}
	where $\Phi=g_0/2-\Ric\left(g_0\right)-\na^2 f_0$. Furthermore, all $C^k$-norms of $\bar g-\varphi_B^* g_0$ and $\bar f-\log|\Gamma|-\varphi_B^* f_0$ are bounded by $1$ for $k \in [1, 10^{10} n \sigma^{-1}]$. By Theorem \ref{thm:upper} (i), the assumption \eqref{normalquo} holds for $C_V=C_V(n,Y)$. Now we choose $L_2=L_1(n, C_V, |\Gamma|)$, where $L_1$ is the same constant as in Proposition \ref{contractionquotient} so that it applies to the weighted Riemannian manifold $\left(M, g(-1),f(-1)\right)$.
	
Choose a cut-off function $\eta$ on $\bar M$ such that $\eta=0$ outside $\{\bar b<L-1\}$ and $\eta =1$ on $\{\bar b<L-2\}$. Let $g_1=\bar g+\eta (\varphi_B^*g_0-\bar g), f_1=\bar f+\eta (\varphi_B^*f_0-\bar f+\log |\Gamma|)$. Then it follows from the proof of Theorem \ref{lojawithradius} that
		\begin{align*}
		\left|\WW(g_1,f_1)-\Theta_m\right|\leq C(n,Y) e^{-\frac{3L^2}{16}}.
	\end{align*}
	
	By Proposition \ref{prop:remainder}, we have
	\begin{equation*}
		\left|\WW(g_0,f_0)-\WW(g_1,f_1)+\log |\Gamma|\right|\leq C(n,Y,\ep) e^{-\frac{(1-\ep)L^2}{4}}.
	\end{equation*}
	
	Combining these estimates, we conclude by choosing a small $\ep>0$ that
	\begin{align*}
		\left|\WW(g_0,f_0)-\Theta_m(\Gamma)\right|\leq C(n,Y) e^{-\frac{3 L^2}{16}}.
	\end{align*}
	Since $L=\rBC^{|\Gamma|}$ and $\WW(g_0, f_0)=\WW_{x_0^*}(1)$, this completes the proof.
\end{proof}

Next, we consider a closed Ricci flow $\XX=\{M^n,(g(t))_{t\in I} \}$ with entropy bounded below by $-Y$ such that $[-10,0]\subset I$. Throughout, we fix a spacetime point $x_0^*=(x_0,0)$, define $\tau=-t$, and set
\begin{align*}
  \begin{dcases}
  & \mathrm{d}\nu_t=\mathrm{d}\nu_{x^*_0;t}=(4\pi \tau)^{-\frac n 2}e^{-f} \,\mathrm{d}V_{g(t)}, \\
		& \Phi= \frac{g}{2}-\tau \left( \Ric+\na^2 f\right),\\
		& F=\tau f.
        \end{dcases}
\end{align*}
Moreover, we define $b=2\sqrt{\left|f(-1)\right|}$.

\begin{defn}[$\r_{C, \delta}^{|\Gamma|}$-radius]\index{$\r_{C, \delta}^{\vert\Gamma\vert}$}\label{defnrCdq}
	For the weighted Riemannian manifold $\left(M,g(-1),f(-1)\right)$, $\r_{C, \delta}^{|\Gamma|}$ is defined as the largest number $L$ for which there exists a smooth covering map $\varphi_C$ of degree $|\Gamma|$ from a domain $U$ containing $\{\bar b \le L\}$ onto a subset of $M$ such that $f \lc \varphi_C(\bar p), -1 \rc \le n$ and
	\begin{align}\label{equ:rCclose}
		\left[\bar g-\varphi_C^* g(-1)\right]_2\leq \delta.
	\end{align}
Moreover, for $x\in U$ with $\bar f(x)\geq 10 n+\log|\Gamma|$, we have
	\begin{align}\label{equ:rClevel}
		(1-\delta)\lc\bar f(x)-\log|\Gamma|\rc\leq f(-1) \circ \varphi_C(x) \leq (1+\delta) \lc \bar f (x)-\log|\Gamma|\rc.
		\end{align} 
\end{defn}

By similar arguments as in Propositions \ref{stabilitymetric}, \ref{stabilitypotential} and Corollary \ref{highestimatepotential1}, we have the following result.

\begin{prop}\label{stabilitymetricq}
	For any small $\ep>0$, there exists $\bar \delta=\bar \delta(n,\ep, |\Gamma|)>0$ such that if $\delta \le \bar \delta$ and $\r^{|\Gamma|}_{C,\delta} \ge \bar \delta^{-2}$, then the following statements hold.
	\begin{enumerate}[label=\textnormal{(\roman{*})}]
	\item On $\left\{\bar b \le \mathbf{r}^{|\Gamma|}_{C,\delta}- \bar \delta^{-1}\right\} \times [-9 , -\ep]$,
		\begin{align*}
		\left[\varphi_C^*g(t)-\bar g(t)\right]_{[\ep^{-1}]} \le \ep,
	\end{align*}
where the norm $[\cdot]_{[\ep^{-1}]}$ is taken with respect to $\bar g(t)$.

\item For $(x,t)\in \left\{\bar b \le \mathbf{r}^{|\Gamma|}_{C,\delta}- \bar \delta^{-1}\right\} \times [-1-\bar \delta , -1+\bar \delta]$, we have
	\begin{align}\label{equ:bdpotquo}
		(1-\ep)\bar F(x)-C(n,Y,\ep)\leq F(\varphi_C(x),t)\leq (1+\ep)\bar F(x)+C(n,Y,\ep).
	\end{align}
Moreover, for any $1 \le l \le \ep^{-1}$,
		\begin{align*}
[\varphi_C^*F]_{l} \le C(n,Y,\ep) e^{\ep l \bar F}
	\end{align*}	
on $\{\bar b \le \r^{|\Gamma|}_{C,\delta}- 2 \bar \delta^{-1}\} \times [-1-\bar \delta ,-1+\bar \delta]$.	
	\end{enumerate}
\end{prop}

	\begin{proof}
(i) follows from the same argument as in Proposition \ref{stabilitymetric}, so we focus on (ii).
		It suffices to prove \eqref{equ:bdpotquo}; once established, the higher-order estimates follow exactly as in Corollary \ref{highestimatepotential1}.

By the same reasoning as in the proof of Proposition \ref{stabilitypotential} (see Claim \ref{extraclaim2}), we have
			\begin{align*}
d_t(p,\varphi_C(x))\leq 	(1+\ep) \bar d_t(\bar p, x),
		\end{align*}	
	for any $(x, t) \in \left\{\bar b \le \mathbf{r}^{|\Gamma|}_{C,\delta}- \bar \delta^{-1}\right\} \times [-9,-\ep]$, where $p=\varphi_C(\bar p)$. This immediately yields the upper bound in \eqref{equ:bdpotquo} by the same argument as in Proposition \ref{stabilitypotential}.
		
However, the lower bound in \eqref{equ:bdpotquo} cannot be obtained directly. Instead, we establish the following: for any $(x, t) \in \left\{\bar b \le \mathbf{r}^{|\Gamma|}_{C,\delta}- \bar \delta^{-1}\right\} \times [-1-\bar \delta,-1+\bar \delta]$,
		\begin{align}\label{potential104quo}
			d_t(p,\varphi_C(x)) \ge (1-\ep)\bar d_t(\bar p, x)-C(n,Y,\ep).
		\end{align}

We first treat the case $t=-1$. It suffices to consider $x \in \left\{\bar b \le \mathbf{r}^{|\Gamma|}_{C,\delta}- \bar \delta^{-1}\right\}$ with $\bar d_{-1}(\bar p,x)\gg 1$. For such $x$, \eqref{equ:rClevel} implies that any $x'\in \varphi_C^{-1}(\varphi_C(x))$ satisfies
		\begin{align*}
(1-2\delta)\bar f(x')\leq \bar f(x)\leq (1+2\delta)\bar f(x'),
	\end{align*}
which implies
		\begin{align}\label{equ:distquo2}
			(1-3\delta)\bar d_{-1}(\bar p,x)\leq \bar d_{-1}(\bar p,x')\leq (1+3\delta)\bar d_{-1}(\bar p,x),
		\end{align}
		provided $\bar d_{-1}(\bar p,x)\geq C(n,Y, \delta)$. 
		
Let $\gamma(s)$, for $s \in [0, L]$, be a minimizing geodesic with respect to $g(-1)$ from $p$ to $\varphi_C(x)$. Set $s_0 \in [0, L]$ to be the smallest parameter so that $\gamma(s_0)$ lies on the boundary of $\varphi_C \lc \{\bar b \le \mathbf{r}^{|\Gamma|}_{C,\delta} \} \rc$; if no such $s_0$ exists, set $s_0=L$. Let $\tilde \gamma$ denote the lift of $\gamma \vert_{[0, s_0]}$ starting from $\bar p$ under $\varphi_C$. 

If $s_0<L$, it is clear from our definition of $s_0$ that the length of $\tilde \gamma$ with respect to $\bar g(-1)$ is at least $\mathbf{r}^{|\Gamma|}_{C,\delta}$. Thus, we obtain by \eqref{equ:rCclose}
		\begin{align*}
d_{-1}(p,\varphi_C(x)) \ge (1-\delta) \mathbf{r}^{|\Gamma|}_{C,\delta} \ge (1-\delta) \bar d_{-1}(\bar p, x).
	\end{align*}

If $s_0=L$, then $\tilde \gamma$ is a curve from $\bar p$ to some $x' \in \varphi_C^{-1}(\varphi_C(x))$. Thus, it follows from \eqref{equ:rCclose} and \eqref{equ:distquo2} that
		\begin{align*}
d_{-1}(p,\varphi_C(x)) \ge (1-\delta) \bar d_{-1}(\bar p, x') \ge (1-3\delta)^2 \bar d_{-1}(\bar p, x).
	\end{align*}

Therefore, \eqref{potential104quo} holds for $t=-1$ if $\delta \le \delta(\ep)$. The general case for $t \in [-1-\bar \delta, -1+\bar \delta]$ follows from the standard distance distortion estimate on $\varphi_C(U) \times [-1-\bar \delta, -1+\bar \delta]$. Here, it suffices to consider minimizing geodesics from $p$ to $x$ lying entirely within $\varphi_C \lc \{\bar b \le \mathbf{r}^{|\Gamma|}_{C,\delta} \} \rc$; otherwise, \eqref{potential104quo} follows by the same argument as above.

Having \eqref{potential104quo}, we can obtain the lower bound in \eqref{equ:bdpotquo} as in the proof of Proposition \ref{stabilitypotential}.
	\end{proof}

Next, we define\index{$\bar \delta_l$}
		\begin{align*}
\bar \delta_l:=\frac{1}{2}\bar \delta (n,10^{-100}n^{-1}l^{-1} \sigma, |\Gamma|),
	\end{align*}
where $\bar \delta$ is from Proposition \ref{stabilitymetricq}. By this choice, we can generalize Theorem \ref{thm:ext1} for the current case.

\begin{thm}\label{thm:ext1q}
For any constants $D>1$ and $\sigma \in (0, 1/10)$, then there exists a large constant $L'=L'(n, Y, \sigma, |\Gamma|, l, D)>1$ satisfying the following property.

Let $\varphi_A$ be the map corresponding to $\r_{A}^{|\Gamma|}$ in Definition \ref{defnradii}. If $\mathbf{r}^{|\Gamma|}_A \in [ L', (1-\sigma) \mathbf{r}_E]$, where $\r_E$ is defined as in Definition \ref{defnrE}, then there exists a covering map $\varphi$ of degree $|\Gamma|$ from a domain $U$ containing $\{\bar b \le \mathbf{r}^{|\Gamma|}_A+D \}$ onto a subset of $M$ such that $\varphi=\varphi_A$ on $\{\bar b \le \rA^{|\Gamma|}-2 \bar \delta_l^{-1}\}$ and
	\begin{align*}
		\left[\varphi^* g(-1)-\bar g\right]_2\leq \bar \delta_l
	\end{align*}
on	$U$. Moreover, if $x\in U$ satisfies $\bar f(x)\geq 10n+\log|\Gamma|$, then
\begin{align*}
	(1-\bar \delta_l)\lc\bar f(x)-\log|\Gamma|\rc \leq f(-1)\circ \varphi_C(x)\leq (1+\bar \delta_l) \lc \bar f(x)-\log|\Gamma|\rc.
\end{align*}
In particular, we have
	\begin{align*} 
		\mathbf{r}^{|\Gamma|}_{C,\bar \delta_l} \ge \mathbf{r}^{|\Gamma|}_A+D.
	\end{align*}
\end{thm}

\begin{proof}
By the definitions of $\rA^{|\Gamma|}$ and $\r^{|\Gamma|}_{C, \bar \delta_l}$, we conclude that $\r^{|\Gamma|}_{C, \bar \delta_l} \ge \mathbf{r}^{|\Gamma|}_A -1$, provided that $L'$ is sufficiently large.

Suppose the conclusion fails. Then there exists a sequence of Ricci flows $\XX^i=\{M^n_i, (g_i(t))_{t\in [-10,0]}\}$ with entropy bounded below by $-Y$. Moreover, there exist base points $x_{0,i}^*=(x_{0,i},0)\in\XX^i$ such that $\rA^{|\Gamma|, i}\to +\infty$ and $\rA^{|\Gamma|, i} \le (1-\sigma)\mathbf{r}^i_E$. For each $i$, there exists a covering map $\varphi_{A,i}$ of degree $|\Gamma|$ from a domain $U_i$ containing $\{\bar b \le \rA^{|\Gamma|, i}\}$ onto a subset of $M$ corresponding to $\rA^{|\Gamma|, i}$ in Definition \ref{defnradiiquo}. However, it is not possible to find a covering map from a domain containing $\{\bar b \le \rA^{|\Gamma|, i}+D_0\}$ onto a subset of $M_i$ for some constant $D_0>0$ that satisfies the required properties. 

We define $D_1:=D_0+2 \bar \delta_l^{-1}$, $L_i:=\rA^{|\Gamma|, i}-2 \bar \delta_l^{-1}$ and the hypersurface $\Sigma_i:=\varphi_{A, i}(U_i) \cap \lc \{b_i=L_i-1\} \rc$, where $b_i:=2\sqrt{|f_i(-1)|}$. Moreover, we set $\Sigma_i':=\varphi_{A, i}^{-1}(\Sigma_i)$ and denote by $\Omega_i(-1)$ the domain enclosed by $\Sigma_i$. Let $\Omega_i=\Omega_i(-1) \times [-2, -1/2]$. It is clear that
	\begin{align*}
 \Sigma_i' \subset \{\bar b \le L_i\} \setminus \{\bar b \ge L_i-2\}.
	\end{align*}
We choose a base point $q_i^*=(q_i, -1/2) \in \XX^i$ such that $q_i \in \Sigma_i$.

Define the time intervals $\III^{++}=[-10, 0]$, $\III^+=[-9.9, 0]$, $\III=[-9.8, 0]$ and $\III^-=(-9.8, 0]$. Then, by Theorem \ref{thm:intro1}, passing to a subsequence if necessary, we have
	\begin{align*}
		\left(\XX_\III^i,d_i^*,q_i^*,\t_i\right)\xrightarrow[i\to\infty]{\quad \mathrm{pGH}\quad} \left(Z, d_Z, q,\t\right),
	\end{align*}
	where $d_i^*$ denotes the spacetime distance induced by $g_i(t)$ (see Definition \ref{defnd*distance}). We assume that the regular part $\RR$ of $Z$ carries a structure of Ricci flow spacetime $(\RR, \t, \partial_\t, g^Z)$. Let $\phi_i$ denote the diffeomorphisms given in Theorem \ref{thm:intro3}.

As in the proof of Theorem \ref{thm:ext1}, $(\Omega_i, g_i(t), q_i^*)$, via the diffeomorphisms $\phi_i$, converge smoothly to a domain $\Omega \subset Z$ containing the point $q$. Moreover, for any $q_i' \in \varphi_{A, i}^{-1}(q_i)$, we have
	\begin{align*}
\lc \varphi_{A, i}^{-1}\lc \Omega_i \rc , \varphi_{A, i}^* (\phi_i^{-1})^* g^Z_t, q_i' \rc \xrightarrow[i\to\infty]{C^{\infty}} \lc \R_-\times\R^{n-m-1}\times S^m \times [-2, -1/2], g_c(t), q_\infty' \rc,
	\end{align*}
where $g_c(t)$ is a family of standard metrics on $\R_-\times\R^{n-m-1}\times S^m$. By the same argument as in the proof of Theorem \ref{thm:ext1}, we conclude that $Z_{(-1-\ep,-1+\ep)}=\RR'_{(-1-\ep, -1+\ep)}$, where $\RR':=\iota_q(\RR^q)$ and $\ep$ is a small constant, depending on $n$, $Y$, $|\Gamma|$, $\sigma$ and $l$. 

In addition, there exists a vector field $V_i$ such that on $B_{g_i(-1)}(\Sigma_i, 100 D_1)$ such that
		\begin{align*}
	\abs{|V_i|-1} \le \Psi(i^{-1}) \quad \text{and} \quad  [\na_i V_i]_{100} \le \Psi(i^{-1}).
	\end{align*} 
Moreover, on $B_{g_i(-1)}(\Sigma_i, 100D_1)$, the vector $\lc \varphi_{A, i} \rc^{-1}_*V_i$ converge smoothly to the direction in the first $\R$ of $\R \times\R^{n-m-1}\times S^m$. By using $V_i$, one can define $\varphi'_i$ and $\varphi_i''$ in exactly the same way as \eqref{gluemetric0} and \eqref{gluemetric1}. Thus, it follows from the construction that $\varphi_i''$ is a smooth covering map of degree $|\Gamma|$ from a domain containing $\{\bar b \le L_i+20D_1\}$ onto a subset of $M_i$ such that
	\begin{align}\label{equ:extmetric}
		\left[(\varphi_i'')^*g_{i}(-1)-\bar g(-1)\right]_2\leq \Psi(i^{-1}).
	\end{align}
Moreover, for any $y$ in the domain of $\varphi_i''$ with $\varphi_i''(y) \in B_{g_i(-1)}(\Sigma_i, 100D_1)$, we have
		\begin{align}\label{equ:stapot3}
(1-\Psi(i^{-1})) \frac{L_i^2}{4} \le f_i(\varphi_i''(y), -1) \leq (1+\Psi(i^{-1})) \frac{L_i^2}{4}.
	\end{align} 
Indeed, the upper bound in \eqref{equ:stapot3} follows from \eqref{equ:extmetric} together with the argument used in the proof of Proposition \ref{stabilitymetricq} (ii). For the lower bound, the proof of Proposition \ref{stabilitypotential} shows that for any $y$ in the domain of $\varphi_i''$ with $\varphi_i''(y) \in B_{g_i(-1)}(\Sigma_i, 100D_1)$, we have
		\begin{align}\label{equ:stapot4}
f_i(\varphi_i''(y), -1) \ge (1-\Psi(i^{-1})) \frac{d^2_{i,-1}(\varphi_i''(\bar p), \varphi_i''(y))}{4},
	\end{align} 
	where $d_{i, -1}$ denotes the distance function of $g_i(-1)$. Moreover, for any $w$ with $\varphi_i''(w) \in \Sigma_i$, our definition of $L_i$ yields
		\begin{align*}
f_i(\varphi_i''(w), -1)=\frac{(L_i-1)^2}{4}.
	\end{align*} 	
In addition, by the same reasoning as in the derivation of \eqref{potential104quo}, we have
		\begin{align*}
d_{i, -1}(\varphi_i''(\bar p), \varphi_i''(w)) \ge (1-\Psi(i^{-1})) \bar d_{-1}(\bar p, w) \ge (1-\Psi(i^{-1})) L_i.
	\end{align*} 	
	
Since $V_i$ is almost a unit vector field, combining this with \eqref{equ:stapot4} gives
\begin{align*}
f_i(\varphi_i''(y), -1) \ge (1-\Psi(i^{-1})) \frac{L_i^2}{4},
\end{align*}
which establishes the desired lower bound in \eqref{equ:stapot3}.
		
Since $L_i+20 D_1 > \rA^{|\Gamma|, i}+D_0$, combining \eqref{equ:extmetric}, \eqref{equ:stapot3} and using the fact that $\varphi_i''$ is an extension of $\varphi_{A,i}$, we obtain a contradiction for sufficiently large $i$.
\end{proof}

Next, we define
	\begin{align} \label{eq:dconstantquo}
\bar D_l=\bar D_l(n, \sigma)=10^4 \bar \delta_l^{-1} \gg 1.
	\end{align}\index{$\bar D_l$}
By Theorem \ref{thm:ext1q}, we fix a smooth covering map $\tilde \varphi_A$ of degree $|\Gamma|$ from a domain $U$ containing $\{\bar b \le \rA^{|\Gamma|}+\bar D_l\}$ onto a subset of $M$ such that
	\begin{align*}
		\left[\tilde \varphi_A^* g(-1)-\bar g\right]_2\leq \bar \delta_l,
	\end{align*}
on	$U$, and for any $x$ with $\bar f(x)\geq 10n+\log|\Gamma|$, it holds 
\begin{align*}
	(1-\bar \delta_l)\lc \bar f(x)-\log|\Gamma|\rc \leq f(-1)\circ \tilde \varphi_A^*(x)\leq (1+\bar \delta_l) \lc\bar f(x)-\log|\Gamma|\rc.
\end{align*}
And	on	$\{\bar b \le \mathbf{r}^{|\Gamma|}_A-3 \bar \delta_l^{-1} \}$, we have
\begin{align*}
\left[\bar g-\tilde \varphi_A^* g(-1)\right]_5+\left[\bar f-\log|\Gamma|-\tilde \varphi_A^* f(-1)\right]_5\leq e^{\frac{\bar f}{4}-\frac{\rA^2}{16}}.
	\end{align*}

Next, we extend Theorem \ref{thm:ext2}; see also \cite[Theorem 5.7]{li2023rigidity}.

\begin{thm}\label{thm:ext2q}
There exists a large constant $\hat L=\hat L(n, Y, \sigma, |\Gamma|)>1$ such that if $\mathbf{r}^{|\Gamma|}_A \in [ \hat L, (1-2\sigma) \mathbf{r}_E]$, then
	\begin{align*}
		\rBC^{|\Gamma|} \ge \mathbf{r}^{|\Gamma|}_A+ \bar D_{100}/10.
	\end{align*}
	Here, both $\rA^{|\Gamma|}$ and $\rBC^{|\Gamma|}$ denote the radius functions for $(M, g(-1), f(-1))$ (cf. Definition \ref{defnradiiquo}), and $\bar D_{100}$ is the constant defined in \eqref{eq:dconstantquo} with $l=100$.
\end{thm}

\begin{proof}
We set $L=\mathbf{r}^{|\Gamma|}_A-3\bar \delta_{100}^{-1}$ and $\bar D=\bar D_{100}$. Applying the same proof of Theorem \ref{thm:ext2} to 
	\begin{align*}
(g_0, f_0):=\lc \tilde \varphi_A^* g(-1), \tilde \varphi_A^* f(-1)+\log|\Gamma| \rc,
	\end{align*}
we obtain a diffeomorphism $\psi_5$ from $B_{g_E}(0, L+0.3 \bar D) \times S^m$ into $\bar M$ (see \eqref{eq:wang1}, \eqref{eq:wang2}, \eqref{eq:wang3} and \eqref{eq:wang4}) such that
	\begin{align*} 
\abs{\psi_5^* g_0-\bar g}+\abs{\psi_5^* f_0-\bar f} \le C \exp \lc C L^{\frac 3 2}-\frac{L^2}{32} \rc
	\end{align*}
and for any $k \in [1, 10^{11}n\sigma^{-1}]$,
	\begin{align*} 
\|\psi_5^* g_0-\bar g \|_{C^k}+\|\psi_5^* f_0-\bar f \|_{C^k} \le e^{-\frac{L^2}{33}}.
	\end{align*}

We consider $\Omega:=\{b \le L+0.2 \bar D\} \cap \tilde \varphi_A\lc U \rc$, where $b:=2\sqrt{|f(-1)|}$. Then it is clear that $U':=\{\tilde \varphi_A^{-1} \lc \Omega \rc\}$ satisfies
	\begin{align*} 
\{\bar b \le L+0.1 \bar D\} \subset U' \subset \{\bar b \le L+0.3 \bar D\}.
	\end{align*}

Consequently, if we define $\varphi_B:=\tilde \varphi_A \circ \psi_5$ from $\psi_5^{-1}(U')$, then it is clear that $\varphi_B$ satisfies the conclusions for $\rBC^{|\Gamma|}$ in Definition \ref{defnradiiquo}, and hence we have 
	\begin{align*}
\rBC^{|\Gamma|}\geq \rA^{|\Gamma|}+0.1 \bar D.
	\end{align*}

In sum, the proof is complete.	
\end{proof}

The following result is parallel to Theorem \ref{thm:ra} by using Proposition \ref{contractionquotient} and Theorem \ref{thm:ext2q} iteratively.

\begin{thm} 
	Let $\XX=\{M^n,(g(t))_{t\in I}\}$ be a closed Ricci flow with entropy bounded below by $-Y$. Assume $x_0^*=(x_0,0)\in \XX$ and $[-10,0]\subset I$.  For any small $\sigma\in (0,1/10)$, there exists a constant $L=L(n,Y,\sigma, |\Gamma|)$ such that if the weighted Riemannian manifold $\left(M,g(-1),f_{x_0^*}(-1)\right)$ satisfies $\rA^{|\Gamma|}\geq L$, then
	\begin{align*}
		\min\left\{\mathbf{r}^{|\Gamma|}_A, \rBC^{|\Gamma|}\right\} \ge (1-3\sigma) \mathbf{r}_E.
	\end{align*}
\end{thm}

From the above results, we obtain the following Lojasiewicz inequality near quotient cylinders, whose proof is almost identical to that of Theorem \ref{thm:lo}.

\begin{thm}[Lojasiewicz inequality] \label{thm:loquotient}
For any $\beta \in (0,3/4)$, there exist constants $C=C(n,Y)$ and $L_\beta=L(n,Y,\beta, |\Gamma|)$ such that the following property holds.

Let $\XX=\{M,g(t)\}_{t\in I}$ be a closed Ricci flow with entropy bounded below by $-Y$. Assume $x_0^*=(x_0,t_0)\in\XX$ and $[t_0-10 r^2,t_0]\subset I$. If the weighted Riemannian manifold $\left(M,r^{-2}g(t_0-r^2),f_{x_0^*}(t_0-r^2)\right)$ satisfies $\rA^{|\Gamma|} \geq L_\beta$, then
	\begin{align*}
		\left|\mathcal W_{x_0^*}(r^2)-\Theta_m(\Gamma)\right| \le C \left( \mathcal W_{x_0^*}(r^2/2)-\mathcal W_{x_0^*}(2 r^2) \right)^{\beta}.
	\end{align*}
\end{thm}

We remark that Theorems \ref{thm:rasing} and \ref{thm:losing} also hold in the quotient case and leave the precise statements to the interested readers.

\subsection*{Strong uniqueness for quotient cylindrical singularities}

Let $\mathcal C^k(\Gamma)_{-1}$ be a quotient cylinder as in \eqref{eq:quoc}. We define $\mathcal C^k(\Gamma):=(\bar M, (\bar g_{\Gamma}(t))_{t<0}, (\bar f_{\Gamma}(t))_{t<0})$ to be the associated Ricci flow, where $t=0$ corresponds to the singular time, and the potential function is given by
\begin{align*}
\pi^* \bar f_{\Gamma}(t)=\frac{|\vec{x}|^2}{4|t|}+\frac{n-k}{4}+\Theta_{n-k}-\log |\Gamma|,
\end{align*}
where $\pi$ is the quotient map from $\mathcal C^k$ to $\mathcal C^k(\Gamma)$.

Let $\bar{\mathcal C}^k(\Gamma)$ be the completion of $\mathcal C^k(\Gamma)$, equipped with the spacetime distance $d_{\mathcal C}^*$. We also define the base point $p^*$ as the limit of $(\bar p, t)$ as $t \nearrow 0$ with respect to $d_{\mathcal C}^*$, where $\bar p \in \mathcal C^k(\Gamma)$ is a minimum point of $\bar f_{\Gamma}(-1)$. For any $t<0$, we have
	\begin{align*}
\nu_{p^*;t}=(4\pi |t|)^{-\frac n 2} e^{-\bar f_{\Gamma}(t)} \,\mathrm{d}V_{\bar g_{\Gamma}(t)}.
	\end{align*}

Next, we consider a general noncollapsed Ricci flow limit space $(Z, d_Z, \t)$ over $\III$, obtained as the limit of a sequence in $\mathcal M(n, T, Y)$. Then, we have the following definition similar to Definition \ref{def:ccc}.

\begin{defn}
A point $z \in Z_{\III^-}$ is called \textbf{a quotient cylindrical point with respect to $\bar{\mathcal C}^k(\Gamma)$} if a tangent flow at $z$ is $\bar{\mathcal C}^k(\Gamma)$ for some $k$ and $\Gamma$.
\end{defn}

The following theorem regarding the uniqueness of the quotient cylindrical tangent flow can be proved as Theorem \ref{thm:weakunique} by using \cite[Theorem 6.2, Remark 6.3]{li2023rigidity}.

\begin{thm}[Uniqueness of the quotient cylindrical tangent flow] 
Let $(Z, d_Z, \t)$ be a noncollapsed Ricci flow limit space over $\III$, obtained as the limit of a sequence in $\mathcal M(n, T, Y)$. For any $z \in Z_{\III^-}$, if a tangent flow at $z$ is $\bar{\mathcal C}^k(\Gamma)$, then any tangent flow at $z$ is also $\bar{\mathcal C}^k(\Gamma)$.
\end{thm}

Next, we define almost quotient cylindrical points similar to Definition \ref{def:almost0}.

\begin{defn}
Let $(Z, d_Z, \t)$ be a noncollapsed Ricci flow limit space over $\III$, obtained as the limit of a sequence in $\mathcal M(n, T, Y)$. A point $z \in Z_{\III^-}$ is called \textbf{$(\ep, r)$-close to $\bar{\mathcal C}^k(\Gamma)$} if $\t(z)-\ep^{-1} r^2 \in \III^-$ and
	\begin{align*}
		\big(Z, r^{-1}d_Z,z, r^{-2}(\t-\t(z)) \big) \quad \text{is $\ep$-close to}\quad \big(\bar{\mathcal C}^k(\Gamma),d_{\mathcal C}^*, p^*,\t \big) \quad \text{over} \quad [-\ep^{-1}, \ep^{-1}].
	\end{align*}
\end{defn}

By the identical proofs, we obtain the following results parallel to Corollary \ref{quantisummabilityW}, Theorems \ref{thm:modi} and \ref{thm:stronguni2}, respectively.

 \begin{cor}\label{quantisummabilityWq}
	For any constants $0<\ep\leq\ep(n,Y)$, $\alpha\in (1/4,1)$ and $\delta\leq\delta(n,Y,\ep,\alpha)$, the following holds. Let $\XX=\{M^n,(g(t))_{t\in I}\}$ be a closed Ricci flow with entropy bounded below by $-Y$. Fix $x_0^*=(x_0,t_0)\in \XX$ and constants $s_2>s_1>0$. If
	\begin{align*}
		\left|\WW_{x_0^*}(s_1)-\WW_{x_0^*}(s_2)\right|<\delta,
	\end{align*}
and, for any $s\in [s_1,s_2]$, $x_0^*$ is $(\delta, \sqrt{s})$-close to $\bar{\mathcal C}^k(\Gamma)$, then
	\begin{align*}
		\sum_{s_1\leq r_j=2^{-j} \leq s_2}\left|\WW_{x_0^*}(r_j)-\WW_{x_0^*}(r_{j-1})\right|^{\alpha}<\ep.
	\end{align*}
\end{cor}

Let $\XX=\{M^n, (g(t))_{t \in [-T,0)}\}$ be a closed Ricci flow with entropy bounded below by $-Y$, where $0$ is the first singular time. Suppose $(Z, d_Z, \t)$ is the completion of $\XX$. Fix a point $z \in Z_0$ and consider the modified Ricci flow $(M, g^z(s),f^z(s))$ with respect to $z$. We denote by $\rA^{|\Gamma|}(s)$ and $\rBC^{|\Gamma|}(s)$ the radius functions as in Definition \ref{defnradiiquo} with respect to $(M, g^z(s),f^z(s))$.

\begin{thm} \label{thm:radiusquo}
Suppose $z$ is a quotient cylindrical singularity with respect to $\bar{\mathcal C}^k(\Gamma)$. Then for any $\ep>0$, we have
\begin{align*}
\min \left\{\rA^{|\Gamma|}(s), \rBC^{|\Gamma|}(s) \right\} \ge \sqrt{(12-\ep) \log s},
	\end{align*}
where $\sigma=10^{-10} \ep$, provided that $s$ is sufficiently large.
\end{thm}

Next, we prove the strong uniqueness of the quotient cylindrical tangent flow. For simplicity, we set $\bar b_{\Gamma}:=2\sqrt{|\bar f_{\Gamma}|}$ on $\mathcal C^k(\Gamma)$.\index{$\bar b_{\Gamma}$}

\begin{thm}[Strong uniqueness of the quotient cylindrical tangent flow]\label{thm:stronguni2quo}
Suppose $z$ is a quotient cylindrical singularity with respect to $\bar{\mathcal C}^k(\Gamma)$. Then for any small $\ep>0$, there exists a large constant $\bar s$ such that for any $j \ge \bar s$, there exists a diffeomorphism from $\Omega^j:=\left\{\bar b_{\Gamma} \leq \sqrt{(8-\ep)\log j}\right\} \subset \bar M/\Gamma$ onto a subset of $M$ such that
	\begin{align*}
		\psi_{j+1}=\psi_j \quad \mathrm{on}\quad \Omega^j,
	\end{align*}
and for all $s\geq j$, we have on $\Omega^j$,
	\begin{align*}
	\left[\psi_{j}^*g^z(s)-\bar g_{\Gamma} \right]_{[\ep^{-1}]}+\left[\psi_{j}^*f^z(s)-\bar f_{\Gamma} \right]_{[\ep^{-1}]}\leq C(n,Y,\ep)e^{\frac{\bar f_{\Gamma}}{2}} s^{-1+\ep}.
	\end{align*}
Here, the $C^{[\ep^{-1}]}$-norms are with respect to $\bar g_{\Gamma}$.
\end{thm}

Theorem \ref{thm:stronguni2quo} follows, by the same argument as in Theorem \ref{thm:stronguni2}, immediately from the next result.

\begin{thm}
Suppose $z$ is a quotient cylindrical singularity with respect to $\bar{\mathcal C}^k(\Gamma)$. Then for any small $\ep>0$, there exists a large constant $\bar s$ such that for any $s_0 \ge \bar s$, there exists a diffeomorphism $\varphi^{\Gamma}_{s_0}$ from $\left\{\bar b_{\Gamma}  \le \sqrt{(8-\ep)\log s_0}\right\}$ onto a subset of $M$ such that on $\left\{\bar b_{\Gamma} \le \sqrt{(8-\ep)\log s_0}\right\}$ and for all $s\geq s_0$, we have
	\begin{align*}
	\left[(\varphi^{\Gamma}_{s_0})^*g^z(s)-\bar g_{\Gamma} \right]_{[\ep^{-1}]}+\left[(\varphi^{\Gamma}_{s_0})^*f^z(s)-\bar f_{\Gamma} \right]_{[\ep^{-1}]}\leq C(n,Y,\ep)e^{\frac{\bar f_{\Gamma}}{2}} s^{-1+\ep}.
	\end{align*} 
	Here, the $C^{[\ep^{-1}]}$-norms are with respect to $\bar g_{\Gamma}$.
\end{thm}

\begin{proof}
Because the proof is similar to that of Theorem \ref{thm:stronguni}, we sketch the argument and detail only the new ingredients.

Throughout the proof, fix a small constant $\ep>0$ and set
	\begin{align*}
\Omega_s:=\varphi_s\lc \left\{\bar b\leq  \sqrt{(8-\ep/3)\log s}\right\}\rc \subset M,
	\end{align*} 
where $\varphi_s$ denotes the map, restricted to $\left\{\bar b\leq  \sqrt{(8-\ep/3)\log s}\right\}$, corresponding to $\rA^{|\Gamma|}(s)$ in Definition \ref{defnradiiquo}.

From Definition \ref{defnradiiquo} and Theorem \ref{thm:radiusquo}, for all sufficiently large $s$,
		\begin{align}\label{eq:extraqo001}
\left[\bar g-\varphi_s^* g^z(s)\right]_5+\left[\bar f-\log |\Gamma|-\varphi_s^* f^z(s) \right]_5 \leq e^{\frac{\bar f}{4}} s^{-\frac{3}{4}+\ep}.
		\end{align}

Let $\bar p$ be a minimum point of $\bar f$, and set $p_s=\varphi_s(\bar p)$. Then, by \eqref{eq:extraqo001}, we have $f^z(s)(p_s) \le C(n, Y)$ for all large $s$.

\begin{claim}\label{cla:limitqo1}
For all sufficiently large $s$, we have
		\begin{align*}
\left\{x \mid d_{g^z(s)}(x, p_s) \le  \sqrt{(8-0.6 \ep) \log s}\right\} \subset \varphi_s\lc \left\{\bar b < \sqrt{(8-\ep/2)\log s}\right\}\rc.
		\end{align*}
	\end{claim}

Fix $x$ with $d_{g^z(s)}(x, p_s) \le \sqrt{(8-0.6 \ep) \log s}$, and let $\gamma(z):[0, L] \to M$ be a minimizing geodesic with respect to $g^z(s)$ from $p_s$ to $x$, where $L=d_{g^z(s)}(x, p_s)$. Suppose, toward a contradiction, that
		\begin{align*}
x \notin \varphi_s\lc \left\{\bar b<  \sqrt{(8-\ep/2)\log s}\right\}\rc.
		\end{align*}
Define $z_0$ to be the largest number $z \in [0, L]$ so that $\gamma$, when restricted to $[0, z]$, is contained in $\varphi_s\lc \left\{\bar b <  \sqrt{(8-\ep/2)\log s}\right\}\rc$. Since $\varphi_s$ is a local diffeomorphism on $\left\{\bar b<  \sqrt{(8-\ep/2)\log s}\right\}$, we can lift $\gamma \vert_{[0, z_0)}$ to a curve $\tilde \gamma$ in $\left\{\bar b <  \sqrt{(8-\ep/2)\log s}\right\}$ starting from $\bar p$. By our assumption and \eqref{eq:extraqo001}, the length of $\tilde \gamma$, with respect to $\bar g$, is at most $\sqrt{(8-0.55 \ep) \log s}$, for sufficiently large $s$. However, this implies, by \eqref{eq:extraqo001} again, that $\tilde \gamma$ is contained in $\left\{\bar b\leq  \sqrt{(8-0.51\ep)\log s} \right\}$. This contradicts the definition of $z_0$, proving Claim \ref{cla:limitqo1}.

As in the proof of Theorem \ref{thm:stronguni}, Theorem \ref{thm:radiusquo} yields, on $\Omega_s$,
		\begin{align}\label{estishrinkerquanex001qo}
		\left[ \frac{g^z(s)}{2}-\na^2f^z(s)-\Ric(g^z(s))\right]_{l} \leq  C(n,Y,\ep,l) s^{-1-\frac{\ep}{24}},
	\end{align}	
	and
	\begin{align}\label{estishrinkerquan3qo}
		\left[ \frac{n}{2}-\Delta f^z(s)-\scal(g^z(s))\right]_l \leq  C_1(n,Y,\ep,l) s^{-1-\frac{\ep}{24}}
	\end{align}
	for any $l \in \mathbb N$, provided that $s$ is sufficiently large.
	
Choose $\bar s$ large so that \eqref{estishrinkerquanex001qo}, \eqref{estishrinkerquan3qo} and Claim \ref{cla:limitqo1} hold with $l = [\ep^{-1}]$ for all $s \geq \bar s$. Moreover, as in Claim \ref{claimpotential}, we may assume that for a fixed $s_0 \ge \bar s$, 
	\begin{align*}
\Omega^{s_0}:=\varphi_{s_0}\lc\left\{\bar b \leq \sqrt{(8-\ep/2)\log s_0}\right\}\rc \subset \Omega_s, \quad \forall s \ge s_0.
	\end{align*}	

Thus, $\lc \varphi_{s_0}^*g^z(s),\varphi_{s_0}^*f^z(s)\rc$ converges smoothly to $\lc g_\infty, f_\infty\rc$ on $\left\{\bar b \leq \sqrt{(8-\ep/2)\log s_0}\right\}$ as $s \to \infty$. In addition, as in the proof of Theorem \ref{thm:stronguni} (see \eqref{equ:decay1}), we obtain for all $s \ge s_0$,
		\begin{align}\label{equ:decay1qo}
		\left[\varphi_{s_0}^*g^z(s)-g_\infty \right]_{[\ep^{-1}]}+\left[\varphi_{s_0}^*f^z(s)-f_\infty\right]_{[\ep^{-1}]}\leq C(n,Y,\ep)e^{\frac{\bar f}{2}} s^{-1+\frac{\ep}{10}}.
	\end{align}

Take a sequence $s_i \to \infty$, then by our assumption, 
\begin{align*}
		\lc M, g^z(s_i),f^z(s_i), p_{s_i}\rc\xrightarrow[i\to\infty]{\quad\text{pointed-Cheeger--Gromov} \quad}\lc \bar M/\Gamma,\bar g_{\Gamma},\bar f_{\Gamma},\bar p_{\Gamma}\rc,
	\end{align*}
where $\bar p_{\Gamma}$ is a minimum point of $\bar f_{\Gamma}$. Thus, for sufficiently large $i$, there exists a diffeomorphism $\psi_{i}:U_{i}\subset \bar M/\Gamma\to M$, such that $U_{i}\subset U_{i+1}$, and $\bigcup_{i}U_{i}=\bar M/\Gamma$, satisfying
	\begin{align}\label{equ:convcylqo}
		\psi_{i}(\bar p_{\Gamma})=p_{s_i}, \quad \psi_{i}^*g^z(s_i)\xrightarrow[i\to\infty]{C_{\loc}^\infty}\bar g_{\Gamma},\quad \psi_{i}^*f^z(s_i)\xrightarrow[i\to\infty]{C_{\loc}^\infty}\bar f_{\Gamma}.
	\end{align}

Next, we set $V_i:=\psi_i^{-1} \lc \Omega^{s_0} \rc \subset U_i$, which is well-defined for large $i$. Combining Claim \ref{cla:limitqo1} with \eqref{equ:convcylqo}, we obtain, for all sufficiently large $i$,
		\begin{align}\label{eq:extraqo002}
\Omega':=\left\{\bar b_{\Gamma} \le \sqrt{(8-\ep)\log s_0}\right\} \subset V_i.
		\end{align}
		
Since $\lc \varphi_{s_0}^*g^z(s),\varphi_{s_0}^*f^z(s)\rc$ converge smoothly to $\lc g_\infty, f_\infty\rc$ on $\left\{\bar b \leq \sqrt{(8-\ep/2)\log s_0}\right\}$, it follows from \eqref{equ:convcylqo} and \eqref{eq:extraqo002} that $\psi_i: \Omega' \to \Omega^{s_0}$ converge smoothly to a diffeomorphism $\varphi^{\Gamma}_{s_0}:\Omega' \to \Omega^{s_0}$ such that
	\begin{align*}
\lc (\varphi^{\Gamma}_{s_0})^{-1} \circ \varphi_{s_0} \rc^* \bar g_{\Gamma}=g_\infty \quad \text{and} \quad \lc (\varphi^{\Gamma}_{s_0})^{-1} \circ \varphi_{s_0} \rc^* \bar f_{\Gamma}=f_\infty.
	\end{align*}
	
Consequently, the desired estimate follows from \eqref{equ:decay1qo} (after adjusting constant $\ep$), completing the proof.
\end{proof}

Let $\XX=\{M^n, (g(t))_{t \in (-\infty, 0]}\}$ be a complete Ricci flow with bounded curvature on any compact time interval of $(-\infty, 0]$ and with entropy bounded below. For a fixed base point $p^* \in \XX_0$ and set $\mathrm{d}\nu_{p^*;t}=(4\pi|t|)^{-n/2}e^{-f(t)}\,\mathrm{d}V_{g(t)}$. Let $\phi_t$ be the family of diffeomorphisms generated by $-\na_{g(t)} f(t)$ with $\phi_{0}=\mathrm{id}$. As \eqref{eq:mcf1}, we consider
  \begin{equation*}
  \begin{dcases}
    &g' (s):= e^s \phi^*_{-e^{-s}} g(-e^{-s}),   \\
    &f' (s):=\phi^*_{-e^{-s}} f_z(-e^{-s}).\\
        \end{dcases}
  \end{equation*}
  
The following results are parallel to Theorems \ref{thm:modi1} and \ref{thm:stronguni-infty}.

\begin{thm}\label{thm:quotinfinit1}
Suppose $\XX$ admits $\bar{\mathcal C}^k(\Gamma)$ as a tangent flow at infinity. Let $\rA^{|\Gamma|}(s)$ and $\rBC^{|\Gamma|}(s)$ be the radius functions from Definition \ref{defnradiiquo} with respect to $(M, g'(s),f'(s))$. Then,
\begin{align*}
\min\left\{\rA^{|\Gamma|}(s), \rBC^{|\Gamma|}(s) \right\} \ge \sqrt{(12-\ep) \log |s|},
	\end{align*}
where $\sigma=10^{-10} \ep$, provided that $|s|$ is sufficiently large.
\end{thm}

\begin{thm}[Strong uniqueness of quotient cylindrical tangent flow at infinity]\label{thm:stronguni-inftyquo}
Under the same assumptions of Theorem \ref{thm:quotinfinit1}, for any small $\ep>0$, there exists a diffeomorphism from $\Omega^j:=\left\{\bar b_{\Gamma} \leq \sqrt{(8-\ep)\log j}\right\} \subset \bar M/\Gamma$ onto a subset of $M$ such that
\begin{align*}
	\psi_{-j-1}=\psi_{-j} \quad \mathrm{on}\quad \Omega^j,
\end{align*}
and for any $s \le -j$, we have on $\Omega^j$,
\begin{align*}
	\left[\psi_{-j}^*g'(s)-\bar g_{\Gamma} \right]_{[\ep^{-1}]}+\left[\psi_{-j}^*f'(s)-\bar f_{\Gamma} \right]_{[\ep^{-1}]}\leq C(n,Y,\ep)e^{\frac{\bar f_{\Gamma}}{2}} |s|^{-1+\ep}.
\end{align*} 
\end{thm}

\newpage

    \printindex

\bibliographystyle{alpha}
\bibliography{lojaref2}

\vskip10pt

Hanbing Fang, Mathematics Department, Stony Brook University, Stony Brook, NY 11794, United States; Email: hanbing.fang@stonybrook.edu;\\

Yu Li, Institute of Geometry and Physics, University of Science and Technology of China, No. 96 Jinzhai Road, Hefei, Anhui Province, 230026, China; Hefei National Laboratory, No. 5099 West Wangjiang Road, Hefei, Anhui Province, 230088, China; Email: yuli21@ustc.edu.cn. \\

\end{CJK}
\end{document}